\documentclass[twoside,11pt,reqno]{amsart}
\usepackage{amsmath,amssymb,amscd,mathrsfs,epic,wasysym,latexsym,tikz,mathrsfs,cite,hyperref}
\usepackage{pb-diagram}
\usepackage[matrix,arrow]{xy}
\usepackage{diagbox}
\usepackage{comment}
\usepackage{hyperref}

\excludecomment{answer}      %<----- Comment to show details

\usepackage{ stmaryrd }
\usepackage[enableskew]{youngtab}
\usepackage{ytableau}
\usepackage{color}
\usepackage{enumerate}
\usepackage[margin=0.9in]{geometry}
\usetikzlibrary{shapes,snakes,calendar,matrix,backgrounds,folding}
\usepackage[font=small,labelfont=bf]{caption}
\usepackage{young}
\makeatletter

\usepackage{tikz}
\usetikzlibrary{decorations.pathreplacing,decorations.pathmorphing,calc}

\numberwithin{equation}{section}
\newtheorem{Proposition}[equation]{Proposition}
\newtheorem{Lemma}[equation]{Lemma}
\newtheorem{Theorem}[equation]{Theorem}
\newtheorem{Corollary}[equation]{Corollary}
\theoremstyle{definition}  %% makes all of the theorem environments which follow appear in \rm
\newtheorem{Definition}[equation]{Definition}
\newtheorem{Remark}[equation]{Remark}

\newtheorem{Example}[equation]{Example}

\let\<\langle
\let\>\rangle

\newcommand\Comment[2][\relax]{\space\par\medskip\noindent%
   \fbox{\begin{minipage}{\textwidth}\textbf{Comment\ifx\relax#1\else---#1\fi}\newline%
        #2\end{minipage}}\medskip
}

\newcommand{\hackcenter}[1]{
 \xy (0,0)*{#1}; \endxy}

\def\b1{\text{\boldmath$1$}}

\def\phi{{\varphi}}

\newcommand{\ZZ}{{\mathbb Z}}

\def\b{\mathfrak{b}}
\def\k{\Bbbk}

\theoremstyle{remark}

{\catcode`\|=\active
  \gdef\set#1{\mathinner{\lbrace\,{\mathcode`\|"8000%
  \let|\midvert #1}\,\rbrace}}
}
\def\midvert{\egroup\mid\bgroup}

\colorlet{darkgreen}{green!50!black}
\tikzset{dots/.style={very thick,loosely dotted},
         greendot/.style={fill,circle,color=darkgreen,inner sep=1.5pt,outer sep=0}
}
\def\greendot(#1,#2){\node[greendot] at(#1,#2){}}

\newenvironment{braid}{% sets defaults for the braid diagrams
  \begin{tikzpicture}[baseline=6mm,blue,line width=1pt, scale=0.4,
                      draw/.append style={rounded corners},
                      every node/.append style={font=\fontsize{5}{5}\selectfont}]%
  }{\end{tikzpicture}
}

\def\Grid(#1,#2){%  draws a coordinate grid inside a braid diagram
  \draw[very thin,gray,step=2mm] (0,0)grid(#1,#2);
  \draw[very thin,darkgreen,step=10mm] (0,0)grid(#1,#2);
}

\newcommand\Tableau[2][\relax]{
  \begin{tikzpicture}[scale=0.5,draw/.append style={thick,black}]
    \ifx\relax#1\relax%
    \else % shade the boxes in #1
      \foreach\box in {#1} { \filldraw[blue!30]\box+(-.5,-.5)rectangle++(.5,.5); }
    \fi
    \newcount\row\newcount\col
    \row=0
    \foreach \Row in {#2} {
       \col=1
       \foreach\k in \Row {
          \draw(\the\col,\the\row)+(-.5,-.5)rectangle++(.5,.5);
          \draw(\the\col,\the\row)node{\k};
          \global\advance\col by 1
       }
       \global\advance\row by -1
    }
  \end{tikzpicture}
}

\newcommand\YoungDiagram[2][\relax]{
  \begin{tikzpicture}[scale=0.5,draw/.append style={thick,black}]
    \ifx\relax#1\relax%
    \else % shade the boxes in #1
    \foreach\box in {#1} {
      \filldraw[blue!30]\box rectangle ++(1,1);
    }
    \fi
    \newcount\row
    \row=0
    \foreach \col in {#2} {
       \draw(1,\the\row)grid ++(\col,1);
       \global\advance\row by -1
    }
  \end{tikzpicture}
}

\begin{document}

%\Comment[AM]{I've added a \texttt{$\backslash$Comment\{\}} macro to help us write standout notes/comments/queries/etc to each other in the file. To mark them as your comments use something like \texttt{$\backslash$Comment[Sasha]\{\dots\}}  or  \texttt{$\backslash$Comment[Arun]\{\dots\}} etc.}

%%fakesection { title }
\title[Injectively \MakeLowercase{\(k\)}-colored rooted forests]{{\bf Injectively \MakeLowercase{\(k\)}-colored rooted forests}}

\author{\sc Thomas Einolf}
\address{Washington \& Jefferson College\\ Washington\\ PA~15301, USA}
\email{einolft@washjeff.edu}

\author{\sc Robert Muth}
\address{Department of Mathematics\\ Washington \& Jefferson College\\ Washington\\ PA~15301, USA}
\email{rmuth@washjeff.edu}

\author{\sc Jeffrey Wilkinson}
\address{Washington \& Jefferson College\\ Washington\\ PA~15301, USA}
\email{wilkinsonjw@washjeff.edu}

%\subjclass[2010]{16G99}

%\thanks{Research supported by the NSF grant DMS-1161094 and the Humboldt Foundation.}

\begin{abstract}
We enumerate injectively \(k\)-colored rooted forests with a given number of vertices of each color and a given sequence of root colors. We obtain from this result some new multi-parameter distributions of Fuss-Catalan numbers. As an additional application we enumerate triangulations of regular convex polygons according to their proper 3-coloring type. 
\end{abstract}

\maketitle

\section{Introduction}
A \(k\)-coloring \(\ell: V \to \{ 1, 2, \ldots,k\}\) of a directed graph \(D = (V,A)\) is {\em \(\overline{N^+}\)-injective} (henceforth, just {\em injective}) if \(\ell\) is injective on all closed out-neighborhoods in \(D\). 
Study of such colorings was initiated by Courcelle in \cite{Courcelle} (who termed them {\em semi-strong} in the \(\overline{N^-}\)-injective setting), and this work was expanded upon by Raspaud and Sopena in \cite{Rasp}. Injective graph colorings have been studied from multiple perspectives, such as chromatic numbers, locally injective homomorphisms, and related solvability and complexity problems, in both the oriented and unoriented settings; see for instance \cite{Mac, Hahn, Yuehua, Sudev, Mac2}. In this paper we take an enumerative focus, studying injective \(k\)-colorings of {\em rooted forests}, which we consider as directed graphs with arrows oriented out from an ordered sequence of roots. Refined enumerations of rooted trees and forests via various statistics have appeared in \cite{Deutsch, Takacs, Riordan}, to name a few related studies.

Given a composition \(\lambda = (\lambda_1, \ldots, \lambda_k) \vDash n\), we say a \(k\)-colored directed graph has {\em character} \(\lambda\) provided that it has \(\lambda_i\) vertices of color \(i\) for \(i \in \{1, \ldots, k\}\). Given moreover a sequence \(\underline{c} = (c_1, \ldots, c_m) \in \{1, \ldots, k\}^m\), we say that an injectively \(k\)-colored rooted forest \(F\) is a {\em \((\lambda, \underline{c})\)-forest} if \(F\) has character \(\lambda\) and roots colored by the sequence \(\underline c\). See Figures~\ref{fig:4coltree}, \ref{fig:20for} for examples.
Our main result is the following, which appears as Theorem~\ref{MainForThm} and Corollary~\ref{TreeCor} in the text:

\begin{Theorem}\label{BigThm}
Let \(k,m,n \in \mathbb{N}\), \(\lambda = (\lambda_1, \ldots, \lambda_k) \vDash n\),  and \(\underline{c} =(c_1, \ldots, c_m) \in \{1, \ldots, k\}^m\). 
Setting \(i_{\underline c} = \#\{j \mid c_j = i\}\) for \(i \in \{1, \ldots, k\}\), the number \(f_{\lambda, \underline{c}}\) of \((\lambda, \underline{c})\)-forests (up to isomorphism, see \S\ref{lacdef}) is given by \(f_{\lambda, \underline c} = \delta_{m,n}\delta_{m,i_{\underline c}}\) if \(\lambda_i = n\) for some \(i \in \{1,\ldots, k\}\). Otherwise:
\begin{align*}
f_{\lambda,\underline c} = P_k(\lambda_1, \ldots, \lambda_k, 1_{\underline c}, \ldots, k_{\underline c})
\prod_{i=1}^k \frac{1}{n - \lambda_i}{n- \lambda_i \choose \lambda_i - i_{\underline c}},
\end{align*}
where \(P_k\) is a homogeneous \(k\)-degree polynomial in \(2k\) variables, defined by:
\begin{align}\label{ThePoly1}
P_k(x_1, \ldots, x_k, y_1, \ldots, y_k)
=
\hspace{-2mm}
\sum_{S \subseteq [1,k]}
\hspace{-1mm}
\left[
(-1)^{|S|}
\hspace{-1mm}
\left(
\hspace{0.5mm}
\prod_{i \in S}y_i
\right)
\hspace{-1mm}
\left( 
\sum_{i = 1}^k (\delta_{i \in S} - |S|)x_i 
\right)
\hspace{-1mm}
\left(
\sum_{i=1}^k
x_i
\right)^{\hspace{-1mm}k-|S|-1}
\right]\hspace{-1mm}.
\end{align}
In particular, the number of injectively \(k\)-colored {\em trees} with \(n\) vertices, character \(\lambda\), and root colored by \(c \in \{1, \ldots, k\}\) is given by \(t_{\lambda, c} = 0\) if \(\lambda_c = 0\), and otherwise:
\begin{align}\label{treeform}
t_{\lambda,c} & = 
n^{k-2}{ n- \lambda_c \choose \lambda_c -1}
\prod_{i  \neq c} \frac{1}{n-\lambda_i} {n- \lambda_i \choose \lambda_i}.
\end{align}
\end{Theorem}

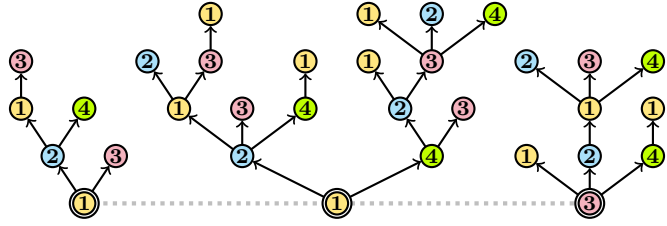
\begin{figure}[h]
\begin{align*}
\hackcenter{
\begin{tikzpicture}[scale=0.42]
\draw[ultra thick, dotted, lightgray] (-8,-1)--(8,-1);
 \draw[thick,->, shorten >=0.16cm]  (0,-1)--(3,0.5);
  \draw[thick,->, shorten >=0.16cm]  (0,-1)--(-3,0.5);
          %%%%
           \draw[thick,->, shorten >=0.16cm]  (-3,0.5)--(-5,2);
            \draw[thick,->, shorten >=0.16cm]  (-3,0.5)--(-3,2);
             \draw[thick,->, shorten >=0.16cm]  (-3,0.5)--(-1,2);
               %%%%%
                 \draw[thick,->, shorten >=0.16cm]  (3,0.5)--(2,2);
                    \draw[thick,->, shorten >=0.16cm]  (3,0.5)--(4,2);
                 %%%%%%%%
                  \draw[thick,->, shorten >=0.16cm]  (2,2)--(1,3.5);
                     \draw[thick,->, shorten >=0.16cm]  (2,2)--(3,3.5);
                     %%%%
                       \draw[thick,->, shorten >=0.16cm]  (3,3.5)--(5,5);
                       \draw[thick,->, shorten >=0.16cm]  (3,3.5)--(3,5);
                       \draw[thick,->, shorten >=0.16cm]  (3,3.5)--(1,5);
                       \draw[thick,->, shorten >=0.16cm]  (-4,3.5)--(-4,5);
                   %%%%%
                     \draw[thick,->, shorten >=0.16cm]  (-1,2)--(-1,3.5);
                   %%%%%
                     \draw[thick,->, shorten >=0.16cm]  (-5,2)--(-6,3.5);
                       \draw[thick,->, shorten >=0.16cm]  (-5,2)--(-4,3.5);
                   \draw[ thick, black, fill = cyan!30] (-6,3.5) circle (10pt);
                        \node[] at (-6,3.5){ $\scriptstyle{\mathbf{2}}$};
                    \draw[ thick, black, fill = violet!30!red!30] (-4,3.5) circle (10pt);
                       \node[] at (-4,3.5){ $\scriptstyle{\mathbf{3}}$};
                         \draw[ thick, black, fill = white] (0,-1) circle (13pt);
                          \draw[ thick, black, fill = pink!50!yellow!80] (0,-1) circle (10pt);
                                    \node[] at (0,-1){ $\scriptstyle{\mathbf{1}}$};
           \draw[ thick, black, fill = lime] (3,0.5) circle (10pt);
                     \node[] at (3,0.5){ $\scriptstyle{\mathbf{4}}$};
               \draw[ thick, black, fill = cyan!30] (-3,0.5) circle (10pt);
                         \node[] at (-3,0.5){ $\scriptstyle{\mathbf{2}}$};
                                 \draw[ thick, black, fill = pink!50!yellow!80] (-5,2) circle (10pt);
                                    \node[] at (-5,2){ $\scriptstyle{\mathbf{1}}$};
           \draw[ thick, black, fill = violet!30!red!30] (-3,2) circle (10pt);
              \node[] at (-3,2){ $\scriptstyle{\mathbf{3}}$};
               \draw[ thick, black, fill = lime] (-1,2) circle (10pt);
                  \node[] at (-1,2){ $\scriptstyle{\mathbf{4}}$};
                             \draw[ thick, black, fill = cyan!30] (2,2) circle (10pt);
                                  \node[] at (2,2){ $\scriptstyle{\mathbf{2}}$};
                                         \draw[ thick, black, fill = violet!30!red!30] (4,2) circle (10pt);
                                            \node[] at (4,2){ $\scriptstyle{\mathbf{3}}$};
 \draw[ thick, black, fill =violet!30!red!30] (3,3.5) circle (10pt);
    \node[] at (3,3.5){ $\scriptstyle{\mathbf{3}}$};
                    \draw[ thick, black, fill = pink!50!yellow!80] (-1,3.5) circle (10pt);
                       \node[] at (-1,3.5){ $\scriptstyle{\mathbf{1}}$};
                              \draw[ thick, black, fill = pink!50!yellow!80] (1,3.5) circle (10pt);
                                 \node[] at (1,3.5){ $\scriptstyle{\mathbf{1}}$};
                                  \draw[ thick, black, fill = lime] (5,5) circle (10pt);
                                 \node[] at (5,5){ $\scriptstyle{\mathbf{4}}$};
                                  \draw[ thick, black, fill = cyan!30] (3,5) circle (10pt);
                                 \node[] at (3,5){ $\scriptstyle{\mathbf{2}}$};
                                   \draw[ thick, black, fill = pink!50!yellow!80] (1,5) circle (10pt);
                                 \node[] at (1,5){ $\scriptstyle{\mathbf{1}}$};
                                   \draw[ thick, black, fill = pink!50!yellow!80] (-4,5) circle (10pt);
                                 \node[] at (-4,5){ $\scriptstyle{\mathbf{1}}$};
  %%%%%%%%%%%%%%%%%%%%%%%%%%%%%%
  %%%%%%%%%%%%%%%%%%%%%%%%%%%%%%
  %%%%%%%%%%%%%NEXT TREE%%%%%%%%%%%%%%%%
  %%%%%%%%%%%%%%%%%%%%%%%%%%%%%%
  %%%%%%%%%%%%%%%%%%%%%%%%%%%%%%
   \draw[thick,->, shorten >=0.16cm]  (0-8,-1)--(1-8,0.5);
      \draw[thick,->, shorten >=0.16cm]  (0-8,-1)--(-1-8,0.5);
         \draw[thick,->, shorten >=0.16cm]  (-1-8,0.5)--(-2-8,2);
            \draw[thick,->, shorten >=0.16cm]  (-1-8,0.5)--(0-8,2);
              \draw[thick,->, shorten >=0.16cm]  (-2-8,2)--(-2-8,3.5);
       \draw[ thick, black, fill = white] (0-8,-1) circle (13pt);
     \draw[ thick, black, fill = pink!50!yellow!80] (0-8,-1) circle (10pt);
         \node[] at (0-8,-1){ $\scriptstyle{\mathbf{1}}$};  
           \draw[ thick, black, fill = violet!30!red!30] (1-8,0.5) circle (10pt);
         \node[] at (1-8,0.5){ $\scriptstyle{\mathbf{3}}$};
            \draw[ thick, black, fill = cyan!30] (-1-8,0.5) circle (10pt);
         \node[] at (-1-8,0.5){ $\scriptstyle{\mathbf{2}}$};    
            \draw[ thick, black, fill = pink!50!yellow!80] (-2-8,2)circle (10pt);
         \node[] at (-2-8,2){ $\scriptstyle{\mathbf{1}}$};    
            \draw[ thick, black, fill = lime] (0-8,2) circle (10pt);
         \node[] at (0-8,2){ $\scriptstyle{\mathbf{4}}$};    
            \draw[ thick, black, fill =violet!30!red!30] (-2-8,3.5) circle (10pt);
         \node[] at (-2-8,3.5){ $\scriptstyle{\mathbf{3}}$};      
        %%%%%%%%%%%%%%%%%%%%%%%%%%%%%%
  %%%%%%%%%%%%%%%%%%%%%%%%%%%%%%
  %%%%%%%%%%%%%NEXT TREE%%%%%%%%%%%%%%%%
  %%%%%%%%%%%%%%%%%%%%%%%%%%%%%%
  %%%%%%%%%%%%%%%%%%%%%%%%%%%%%%    
     \draw[thick,->, shorten >=0.16cm]  (0+8,-1)--(-2+8,0.5);
      \draw[thick,->, shorten >=0.16cm]  (0+8,-1)--(-0+8,0.5);
            \draw[thick,->, shorten >=0.16cm]  (0+8,-1)--(2+8,0.5);
             \draw[thick,->, shorten >=0.16cm]  (0+8,0.5)--(-0+8,2);
             \draw[thick,->, shorten >=0.16cm]  (-0+8,2)--(-2+8,3.5);
                   \draw[thick,->, shorten >=0.16cm]  (-0+8,2)--(-0+8,3.5);
                         \draw[thick,->, shorten >=0.16cm]  (-0+8,2)--(2+8,3.5);
                           \draw[thick,->, shorten >=0.16cm]  (2+8,0.5)--(2+8,2);
           \draw[ thick, black, fill = white] (0+8,-1) circle (13pt);
          \draw[ thick, black, fill = violet!30!red!30] (0+8,-1) circle (10pt);
         \node[] at (0+8,-1){ $\scriptstyle{\mathbf{3}}$}; 
             \draw[ thick, black, fill = pink!50!yellow!80] (-2+8,0.5) circle (10pt);
         \node[] at (-2+8,0.5){ $\scriptstyle{\mathbf{1}}$}; 
            \draw[ thick, black, fill = lime] (2+8,0.5) circle (10pt);
         \node[] at (2+8,0.5){ $\scriptstyle{\mathbf{4}}$}; 
            \draw[ thick, black, fill = cyan!30] (0+8,0.5) circle (10pt);
         \node[] at (0+8,0.5){ $\scriptstyle{\mathbf{2}}$}; 
            \draw[ thick, black, fill = pink!50!yellow!80] (2+8,2) circle (10pt);
         \node[] at (2+8,2){ $\scriptstyle{\mathbf{1}}$}; 
           \draw[ thick, black, fill = violet!30!red!30] (0+8,3.5) circle (10pt);
         \node[] at (0+8,3.5){ $\scriptstyle{\mathbf{3}}$}; 
          \draw[ thick, black, fill = pink!50!yellow!80] (0+8,2) circle (10pt);
         \node[] at (0+8,2){ $\scriptstyle{\mathbf{1}}$}; 
               \draw[ thick, black, fill = lime] (2+8,3.5) circle (10pt);
         \node[] at (2+8,3.5){ $\scriptstyle{\mathbf{4}}$}; 
               \draw[ thick, black, fill = cyan!30] (-2+8,3.5) circle (10pt);
         \node[] at (-2+8,3.5){ $\scriptstyle{\mathbf{2}}$}; 
\end{tikzpicture}
}
\end{align*}
\caption{An \(((11,7,8,6),(1,1,3))\)-forest. There are, up to directed graph isomorphisms which fix  roots and preserve vertex colors, \(f_{(11,7,8,6),(1,1,3)} = 2223687758798502796800\) such forests.}
\label{fig:4coltree}       % Give a unique label
\end{figure}

As we explain in \S\ref{BigCatSec}, the numbers \(f_{\lambda, \underline c}\) and \(t_{\lambda, c}\) are of some wider interest. The set of all injectively \(k\)-colored rooted forests with \(n\) vertices and a fixed root color sequence of length \(m\) are counted by the Fuss-Catalan number \(A_{n-m}(k-1,km-m)\). We use this fact to obtain some families of \(p\)-parameter distributions of the Fuss-Catalan numbers  \(A_n(p,1)\) and \(A_n(p,p\ell)\) in \S\ref{Ap1distSec},\ref{FCdistSec}. These distributions may be compared with (but are a different flavor than) the \((p-1)\)-parameter distributions of \(A_n(p,1)\) studied in \cite{Aval}---compare Remark 3.2, loc. cit. with our Proposition~\ref{Ap1dist}.
We describe in Remark~\ref{RemConnex} some additional connections of \(t_{\lambda,c}\) with other known sequences and combinatorial objects.

Finally, in \S\ref{TriSec} we give a correspondence between injectively 3-colored rooted trees and triangulations of regular convex polygons, and use this to enumerate triangulations by their proper 3-coloring type in Theorem~\ref{TriThm}.

\subsection{ArXiv Version}\label{ArxivVersion}  In order to keep the paper from becoming overlong we chose to relegate some of the more lengthy but straightforward calculations to the arXiv version of the paper.  Readers interested in seeing these additional details can download the \LaTeX~source file from the arXiv and find a toggle near the beginning of the file which allows one to compile the paper with these calculations included.

\subsection{Acknowledgements}
The authors wish to thank Sam Zbarsky for suggesting the consideration of colored forests as a route to establishing that formula (\ref{treeform}) holds for trees in a more specialized setting. This suggestion led the authors to the general result featured in this paper.

\section{Preliminaries}

\subsection{Partitions and compositions}\label{CombSec}
We fix \(k \in \mathbb{N}\) throughout. For integers \(a \leq b\), we will write \([a,b]:=\{a,a+1, \ldots, b\}\). 
We set \(\Lambda_k := \mathbb{Z}^k_{\geq 0}\), and treat \(\Lambda_k\) as an abelian monoid under componentwise addition. We take the usual generators of \(\Lambda_k\):
\begin{align*}
\varepsilon_i = (0, \ldots,0, 1,0, \ldots, 0), \qquad (i \in [1,k]),
\end{align*}
where \(1\) is in the \(i\)th slot. We write \(\mathbf{0}:= (0, \ldots, 0)\) for the identity in \(\Lambda_k\). 
For \(\lambda = (\lambda_1, \ldots, \lambda_k) \in \Lambda_k\), we write \(|\lambda| = \lambda_1 + \cdots+ \lambda_k\). We also write
\begin{align*}
\Lambda_k(n) &:= \{ \lambda \in \Lambda_k \mid |\lambda| = n\};\qquad\;\;\;
\Lambda_k^+(n) := \{ \lambda \in \Lambda_k \mid |\lambda| = n, \,\lambda_1 \geq \cdots \geq \lambda_k\};\\
\Lambda_k(<n) &:= \{ \lambda \in \Lambda_k \mid |\lambda| < n\}; \qquad
\Lambda_k(\leq n) := \{ \lambda \in \Lambda_k \mid |\lambda| \leq n\},
\end{align*}
where \(\Lambda_k(n)\) and \(\Lambda_k^+(n)\) are the {\em \(k\)-part compositions} and {\em \(k\)-part partitions} of \(n\), respectively. For \(\lambda \in \Lambda^+_k(n)\), we write \(\mathcal{O}(\lambda)\) for the orbit of \(\lambda\) under the permutation action of the symmetric group \(\mathfrak{S}_k\) on \(\Lambda_k\), and  \(\textup{o}(\lambda) = |\mathcal{O}(\lambda)|\).

\subsection{Fuss-Catalan numbers}\label{FussCatSec}
For \(n,p \in \mathbb{Z}_{\geq 0}\), \(r \in \mathbb{N}\), the associated {\em Fuss-Catalan number} (or {\em Raney number}, or {\em Hagen/Rothe coefficient}) \(A_n(p,r)\) is defined by:
\begin{align*}
A_n(p,r) := \frac{r}{np+r}{np+r \choose n}.
\end{align*}
We write \(B_{p,r}(x) = \sum_{n = 0}^\infty A_n(p,r)x^n\) for the generating function of the sequence \((A_n(p,r))_{n=0}^\infty\).

For \(r=1\), the sequence \((A_n(p,1))_{n =0}^\infty\) may be alternatively defined via the recurrence
\begin{align}\label{CatRecur}
A_0(p,1) = 1, \qquad 
A_n(p,1) = \sum_{\lambda \in \Lambda_p(n-1)} A_{\lambda_1}(p,1) \cdots A_{\lambda_p}(p,1).
\end{align}
Thus \(B_{p,1}(x)\) is the power series satisfying the relation \(B_{p,1}(x) = xB_{p,1}(x)^p +1\). Moreover, we have \(B_{p,r}(x) = B_{p,1}(x)^r\) for all \(r \in \mathbb{N}\).

Fuss-Catalan numbers arise in a large number of combinatorial contexts, enumerating \(p\)-ary trees, Dyck paths, standard tableaux and core partitions, to name a few; see for instance \cite{Hilton, 
Aval, Zhou}.
The most famous of this family of sequences are the {\em Catalan numbers} \(C_n = A_n(2,1)\). Though named for Eug\`ene Catalan, who investigated these numbers in the context of polygon triangulation (see \S\ref{TriSec}), it appears that the sequence was previously known and studied by Minggatu, see \cite{Luo, Larcombe2}.

\subsection{Basic definitions for directed graphs}

\begin{Definition}
A directed graph \(D = (V_D, A_D)\) is the data of a set of {\em vertices} \(V_D\) and a set of {\em arrows} \(A_D \subseteq V_D\times V_D\). We depict \((v,w) \in A_D\) with an arrow \(v \to w\). An  {\em oriented path} from \(v\) to \(w\) in \(D\) is a sequence of arrows
\begin{align*}
(u_0,u_1),(u_1,u_2),(u_2,u_3), \ldots, (u_{j-1},u_j)
\end{align*}
in \(A_D\), for some \(j \in \mathbb{N}\) such that \(u_0 = v\) and \(u_j = w\). We also consider the empty sequence of arrows to be an oriented path from \(v\) to \(v\) for any \(v \in V_D\).

An isomorphism \(\phi: D \xrightarrow{\sim} D'\) of directed graphs is the data of a bijection \(\phi: V_D \to V_{D'}\) of vertices that induces a bijection \(\phi \times \phi: A_D \to A_{D'}\) on arrows.

\end{Definition}

\begin{Definition}
Let \(m \in \mathbb{N}\). An {\em \(m\)-rooted forest} \(F\) is a directed graph \((V_F,A_F)\) together with an injective function \(\mathsf{r}_F:[1,m] \to V_F\) such that for all \(v \in V_F\) there exists a unique oriented path from an element of \(\mathsf{r}_F([1,m])\) to \(v\). We refer to \(\mathsf{r}_F\) as a {\em root order} and \(\mathsf{r}_F([1,m])\) as the {\em roots} of \(F\). If \(m=1\), we refer to \(F\) as a {\em rooted tree}.
\end{Definition}

\begin{Definition}
For a directed graph \(D = (V_D,A_D)\), the {\em out-neighborhood of \(v \in V_D\)} is the set \(N^+(v):=\{w \in V_D \mid (v,w) \in A_D\}\). The {\em in-neighborhood of \(v \in V_D\)} is the set \(N^-(v):=\{w \in V_D \mid (w,v) \in A_D\} \). The {\em closed} out-neighborhood and {\em closed} in-neighborhood we denote by \(\overline{N^+}(v) := N^+(v) \cup \{v\}\) and  \(\overline{N^-}(v) := N^-(v) \cup \{v\}\) respectively.
\end{Definition}

\subsection{Properties of \(k\)-colorings}

\begin{Definition}
For a set \(V\), a {\em \(k\)-coloring} \(\ell\) is a function \(\ell: V \to [1,k]\). For \(S \subseteq V\) and \(c \in \{1, \ldots, k\}\), we will write \(S_{\ell, c} := \ell^{-1}(c) \cap S\), or just \(S_{c}\) when there is no chance of confusion. 
The associated {\em character} \(\textup{char}(\ell)\) of the \(k\)-coloring \(\ell\) is the tuple 
\begin{align*}
\textup{char}(\ell)=
(\,|V_1|,\ldots, |V_k|\,) \in \Lambda_k.
\end{align*}
\end{Definition}

\begin{Definition}
For a graph \(G = (V_G, E_G)\), a \(k\)-coloring \(\ell: V_G \to \{1, \ldots, k\}\) is said to be {\em proper} provided that \(\ell(v) \neq \ell(w)\) whenever \(v,w \in V_G\) are adjacent in \(G\).
\end{Definition}

\begin{Definition}
For a directed graph \(D = (V_D, A_D)\), a \(k\)-coloring \(\ell: V_D \to \{1, \ldots, k\}\) is said to be \(\overline{N^+}\)-{\em injective} (or, henceforth, just {\em injective} when there is no chance of confusion) provided that the closed out-neighborhood of any vertex contains at most one vertex of any given color. I.e., \(|\overline{N^+}(v)_c| \leq 1\) for all \(v \in V_D\), \(c \in [1,k]\).\end{Definition}

\begin{Remark}
It follows from definitions that if \(\ell\) is an injective \(k\)-coloring for the directed graph \(D\), then \(\ell\) is a proper \(k\)-coloring for \(D\) as an undirected graph. 
\end{Remark}

\subsection{Defining \((\lambda, \underline{c})\)-forests}\label{lacdef}
We now introduce the main combinatorial objects of interest in this paper.

\begin{Definition}\label{TheDef}
Let \(k,m \in \mathbb{N}\),  \(\lambda \in \Lambda_k\), and \(\underline{c} = (c_1, \ldots, c_m) \in [1,k]^m\). 
A \((\lambda, \underline{c})\)-forest \(F\) is an \(m\)-rooted forest \((V_F, A_F)\) equipped with root order \(\mathsf{r}_F:[1,m] \to V_F\) and \(\overline{ N^+}\)-injective \(k\)-coloring \(\ell_F:V_F \to [1,k]\), such that:
\begin{align*}
\textup{char}(\ell_F) = \lambda;
\qquad
\textup{and}
\qquad
\ell_F(\mathsf{r}_F(i)) = c_i \; \textup{for all }i \in [1,m].
\end{align*}

Two \((\lambda, \underline{c})\)-forests \(F \), \(F'\) are {\em isomorphic} if there exists a directed graph isomorphism \(\phi: (V_F,A_F) \xrightarrow{\sim} (V_{F'},A_{F'})\) such that \(\phi(\mathsf{r}_F(i)) = \mathsf{r}_{F'}(i)\) for all \(i \in [1,m]\) and \(\ell_{F'}(\phi(v)) = \ell_{F}(v)\) for all \(v \in V_F\). 
We write \(\mathcal{F}_{\lambda,\underline c}\) for the set of (isomorphism classes of) \((\lambda, \underline{c})\)-forests. In the case where \(m=1\) and \(c \in [1,k]\), we sometimes write \(\mathcal{T}_{\lambda,c}:= \mathcal{F}_{\lambda, (c)}\), and refer to these objects as {\em \((\lambda, c)\)-trees}. We will also write:
\begin{align*}
\mathcal{F}_{n, \underline c} := \bigsqcup_{\lambda \in \Lambda_k(n)} \mathcal{F}_{\lambda, \underline c};
\qquad
\mathcal{T}_{n,  c} := \bigsqcup_{\lambda \in \Lambda_k(n)} \mathcal{T}_{\lambda,  c};
\end{align*}
and
\begin{align*}
f_{\lambda, \underline c} = \big|\mathcal{F}_{\lambda, \underline c}\big|;
\qquad
f_{n, \underline c} = \big|\mathcal{F}_{n, \underline c}\big|;
\qquad
t_{\lambda, c} = \big|\mathcal{T}_{\lambda, c}\big|;
\qquad
t_{n, c} = \big|\mathcal{T}_{n, c}\big|.
\end{align*}
\end{Definition}

\begin{Remark} We refer the reader to Figure~\ref{fig:4coltree} for an example of an \(((11,7,8,6),(1,1,3))\)-forest. Our convention is to order the roots from left to right along a horizontal dotted line.
In Figure~\ref{fig:20for} we display all twenty non-isomorphic \(((3,1,1),(1,1))\)-forests.
\end{Remark}

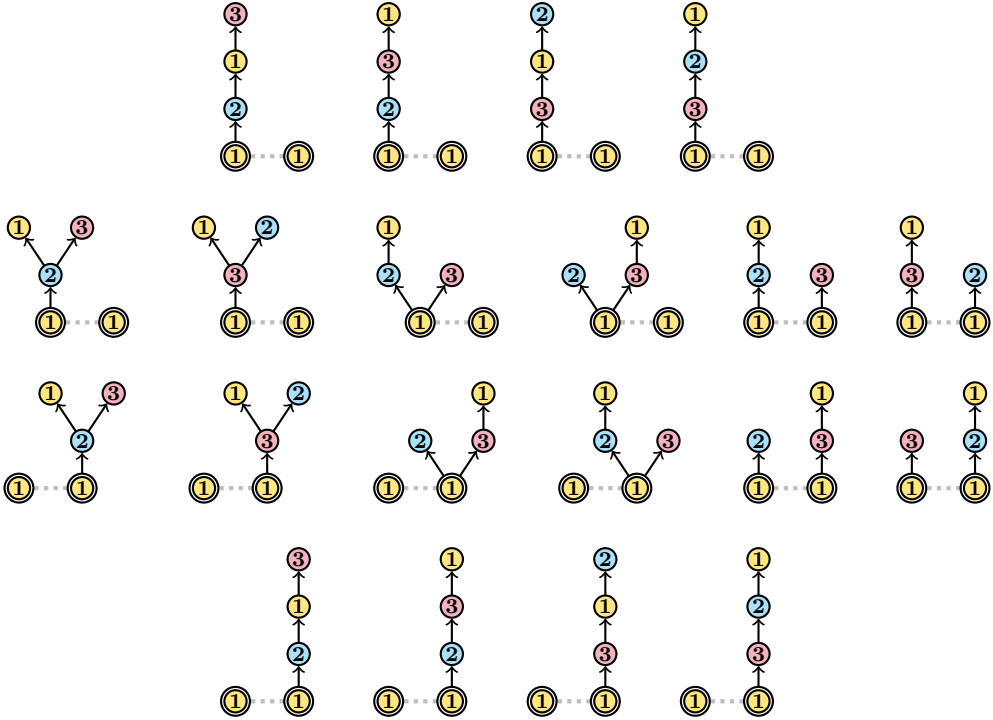
\begin{figure}[h]
\begin{align*}
\begin{tikzpicture}[scale=0.42]
\draw[ultra thick, dotted, lightgray] (-1.11,0)--(1,0);
 \draw[thick,->, shorten >=0.16cm]  (-1,0)--(-1,1.5);
  \draw[thick,->, shorten >=0.16cm]  (-1,1.5)--(-1,3);
    \draw[thick,->, shorten >=0.16cm]  (-1,3)--(-1,4.5);
    %%%%%%%%
         \draw[ thick, black, fill = cyan!30] (-1,1.5) circle (10pt);
         \node[] at (-1,1.5){ $\scriptstyle{\mathbf{2}}$};  
           \draw[ thick, black, fill = pink!50!yellow!80] (-1,3) circle (10pt);
         \node[] at (-1,3){ $\scriptstyle{\mathbf{1}}$};  
           \draw[ thick, black, fill = violet!30!red!30] (-1,4.5) circle (10pt);
         \node[] at (-1,4.5){ $\scriptstyle{\mathbf{3}}$};  
    %%%%%%%%
        \draw[ thick, black, fill = white] (-1,0) circle (13pt);
     \draw[ thick, black, fill = pink!50!yellow!80] (-1,0) circle (10pt);
         \node[] at (-1,0){ $\scriptstyle{\mathbf{1}}$};  
                 \draw[ thick, black, fill = white] (1,0) circle (13pt);
     \draw[ thick, black, fill = pink!50!yellow!80] (1,0) circle (10pt);
         \node[] at (1,0){ $\scriptstyle{\mathbf{1}}$};  
         %
         % \draw[ thick, black, fill = violet!30!red!30] (0+8,-1) circle (10pt);
            % \draw[ thick, black, fill = pink!50!yellow!80] (-2+8,0.5) circle (10pt);
           % \draw[ thick, black, fill = lime] (2+8,0.5) circle (10pt);
           % \draw[ thick, black, fill = cyan!30] (0+8,0.5) circle (10pt);
\end{tikzpicture}
\qquad\begin{tikzpicture}[scale=0.42]
\draw[ultra thick, dotted, lightgray] (-1.11,0)--(1,0);
 \draw[thick,->, shorten >=0.16cm]  (-1,0)--(-1,1.5);
  \draw[thick,->, shorten >=0.16cm]  (-1,1.5)--(-1,3);
    \draw[thick,->, shorten >=0.16cm]  (-1,3)--(-1,4.5);
    %%%%%%%%
         \draw[ thick, black, fill = cyan!30] (-1,1.5) circle (10pt);
         \node[] at (-1,1.5){ $\scriptstyle{\mathbf{2}}$};  
           \draw[ thick, black, fill = violet!30!red!30] (-1,3) circle (10pt);
         \node[] at (-1,3){ $\scriptstyle{\mathbf{3}}$};  
           \draw[ thick, black, fill = pink!50!yellow!80] (-1,4.5) circle (10pt);
         \node[] at (-1,4.5){ $\scriptstyle{\mathbf{1}}$};  
    %%%%%%%%
        \draw[ thick, black, fill = white] (-1,0) circle (13pt);
     \draw[ thick, black, fill = pink!50!yellow!80] (-1,0) circle (10pt);
         \node[] at (-1,0){ $\scriptstyle{\mathbf{1}}$};  
                 \draw[ thick, black, fill = white] (1,0) circle (13pt);
     \draw[ thick, black, fill = pink!50!yellow!80] (1,0) circle (10pt);
         \node[] at (1,0){ $\scriptstyle{\mathbf{1}}$};  
         %
         % \draw[ thick, black, fill = violet!30!red!30] (0+8,-1) circle (10pt);
            % \draw[ thick, black, fill = pink!50!yellow!80] (-2+8,0.5) circle (10pt);
           % \draw[ thick, black, fill = lime] (2+8,0.5) circle (10pt);
           % \draw[ thick, black, fill = cyan!30] (0+8,0.5) circle (10pt);
\end{tikzpicture}
\qquad\begin{tikzpicture}[scale=0.42]
\draw[ultra thick, dotted, lightgray] (-1.11,0)--(1,0);
 \draw[thick,->, shorten >=0.16cm]  (-1,0)--(-1,1.5);
  \draw[thick,->, shorten >=0.16cm]  (-1,1.5)--(-1,3);
    \draw[thick,->, shorten >=0.16cm]  (-1,3)--(-1,4.5);
    %%%%%%%%
         \draw[ thick, black, fill =  violet!30!red!30] (-1,1.5) circle (10pt);
         \node[] at (-1,1.5){ $\scriptstyle{\mathbf{3}}$};  
           \draw[ thick, black, fill = pink!50!yellow!80] (-1,3) circle (10pt);
         \node[] at (-1,3){ $\scriptstyle{\mathbf{1}}$};  
           \draw[ thick, black, fill = cyan!30] (-1,4.5) circle (10pt);
         \node[] at (-1,4.5){ $\scriptstyle{\mathbf{2}}$};  
    %%%%%%%%
        \draw[ thick, black, fill = white] (-1,0) circle (13pt);
     \draw[ thick, black, fill = pink!50!yellow!80] (-1,0) circle (10pt);
         \node[] at (-1,0){ $\scriptstyle{\mathbf{1}}$};  
                 \draw[ thick, black, fill = white] (1,0) circle (13pt);
     \draw[ thick, black, fill = pink!50!yellow!80] (1,0) circle (10pt);
         \node[] at (1,0){ $\scriptstyle{\mathbf{1}}$};  
         %
         % \draw[ thick, black, fill = violet!30!red!30] (0+8,-1) circle (10pt);
            % \draw[ thick, black, fill = pink!50!yellow!80] (-2+8,0.5) circle (10pt);
           % \draw[ thick, black, fill = lime] (2+8,0.5) circle (10pt);
           % \draw[ thick, black, fill = cyan!30] (0+8,0.5) circle (10pt);
\end{tikzpicture}
\qquad\begin{tikzpicture}[scale=0.42]
\draw[ultra thick, dotted, lightgray] (-1.11,0)--(1,0);
 \draw[thick,->, shorten >=0.16cm]  (-1,0)--(-1,1.5);
  \draw[thick,->, shorten >=0.16cm]  (-1,1.5)--(-1,3);
    \draw[thick,->, shorten >=0.16cm]  (-1,3)--(-1,4.5);
    %%%%%%%%
         \draw[ thick, black, fill =  violet!30!red!30] (-1,1.5) circle (10pt);
         \node[] at (-1,1.5){ $\scriptstyle{\mathbf{3}}$};  
           \draw[ thick, black, fill = cyan!30] (-1,3) circle (10pt);
         \node[] at (-1,3){ $\scriptstyle{\mathbf{2}}$};  
           \draw[ thick, black, fill =pink!50!yellow!80 ] (-1,4.5) circle (10pt);
         \node[] at (-1,4.5){ $\scriptstyle{\mathbf{1}}$};  
    %%%%%%%%
        \draw[ thick, black, fill = white] (-1,0) circle (13pt);
     \draw[ thick, black, fill = pink!50!yellow!80] (-1,0) circle (10pt);
         \node[] at (-1,0){ $\scriptstyle{\mathbf{1}}$};  
                 \draw[ thick, black, fill = white] (1,0) circle (13pt);
     \draw[ thick, black, fill = pink!50!yellow!80] (1,0) circle (10pt);
         \node[] at (1,0){ $\scriptstyle{\mathbf{1}}$};  
         %
         % \draw[ thick, black, fill = violet!30!red!30] (0+8,-1) circle (10pt);
            % \draw[ thick, black, fill = pink!50!yellow!80] (-2+8,0.5) circle (10pt);
           % \draw[ thick, black, fill = lime] (2+8,0.5) circle (10pt);
           % \draw[ thick, black, fill = cyan!30] (0+8,0.5) circle (10pt);
\end{tikzpicture}
\end{align*}
\begin{align*}
\begin{tikzpicture}[scale=0.42]
\draw[ultra thick, dotted, lightgray] (-1.11,0)--(1,0);
 \draw[thick,->, shorten >=0.16cm]  (-1,0)--(-1,1.5);
  \draw[thick,->, shorten >=0.16cm]  (-1,1.5)--(0,3);
    \draw[thick,->, shorten >=0.16cm]  (-1,1.5)--(-2,3);
    %%%%%%%%
         \draw[ thick, black, fill =  cyan!30 ] (-1,1.5) circle (10pt);
         \node[] at (-1,1.5){ $\scriptstyle{\mathbf{2}}$};  
           \draw[ thick, black, fill = pink!50!yellow!80] (-2,3) circle (10pt);
         \node[] at (-2,3){ $\scriptstyle{\mathbf{1}}$};  
           \draw[ thick, black, fill =  violet!30!red!30 ] (0,3) circle (10pt);
         \node[] at (0,3){ $\scriptstyle{\mathbf{3}}$};  
    %%%%%%%%
        \draw[ thick, black, fill = white] (-1,0) circle (13pt);
     \draw[ thick, black, fill = pink!50!yellow!80] (-1,0) circle (10pt);
         \node[] at (-1,0){ $\scriptstyle{\mathbf{1}}$};  
                 \draw[ thick, black, fill = white] (1,0) circle (13pt);
     \draw[ thick, black, fill = pink!50!yellow!80] (1,0) circle (10pt);
         \node[] at (1,0){ $\scriptstyle{\mathbf{1}}$};  
         %
         % \draw[ thick, black, fill = violet!30!red!30] (0+8,-1) circle (10pt);
            % \draw[ thick, black, fill = pink!50!yellow!80] (-2+8,0.5) circle (10pt);
           % \draw[ thick, black, fill = lime] (2+8,0.5) circle (10pt);
           % \draw[ thick, black, fill = cyan!30] (0+8,0.5) circle (10pt);
\end{tikzpicture}
\qquad
\begin{tikzpicture}[scale=0.42]
\draw[ultra thick, dotted, lightgray] (-1.11,0)--(1,0);
 \draw[thick,->, shorten >=0.16cm]  (-1,0)--(-1,1.5);
  \draw[thick,->, shorten >=0.16cm]  (-1,1.5)--(0,3);
    \draw[thick,->, shorten >=0.16cm]  (-1,1.5)--(-2,3);
    %%%%%%%%
         \draw[ thick, black, fill =  violet!30!red!30] (-1,1.5) circle (10pt);
         \node[] at (-1,1.5){ $\scriptstyle{\mathbf{3}}$};  
           \draw[ thick, black, fill = pink!50!yellow!80] (-2,3) circle (10pt);
         \node[] at (-2,3){ $\scriptstyle{\mathbf{1}}$};  
           \draw[ thick, black, fill = cyan!30  ] (0,3) circle (10pt);
         \node[] at (0,3){ $\scriptstyle{\mathbf{2}}$};  
    %%%%%%%%
        \draw[ thick, black, fill = white] (-1,0) circle (13pt);
     \draw[ thick, black, fill = pink!50!yellow!80] (-1,0) circle (10pt);
         \node[] at (-1,0){ $\scriptstyle{\mathbf{1}}$};  
                 \draw[ thick, black, fill = white] (1,0) circle (13pt);
     \draw[ thick, black, fill = pink!50!yellow!80] (1,0) circle (10pt);
         \node[] at (1,0){ $\scriptstyle{\mathbf{1}}$};  
         %
         % \draw[ thick, black, fill = violet!30!red!30] (0+8,-1) circle (10pt);
            % \draw[ thick, black, fill = pink!50!yellow!80] (-2+8,0.5) circle (10pt);
           % \draw[ thick, black, fill = lime] (2+8,0.5) circle (10pt);
           % \draw[ thick, black, fill = cyan!30] (0+8,0.5) circle (10pt);
\end{tikzpicture}
\qquad
\begin{tikzpicture}[scale=0.42]
\draw[ultra thick, dotted, lightgray] (-1.11,0)--(1,0);
 \draw[thick,->, shorten >=0.16cm]  (-1,0)--(-2,1.5);
  \draw[thick,->, shorten >=0.16cm]  (-1,0)--(0,1.5);
    \draw[thick,->, shorten >=0.16cm]  (-2,1.5)--(-2,3);
    %%%%%%%%
         \draw[ thick, black, fill =  cyan!30] (-2,1.5) circle (10pt);
         \node[] at (-2,1.5){ $\scriptstyle{\mathbf{2}}$};  
           \draw[ thick, black, fill =violet!30!red!30 ]  (0,1.5) circle (10pt);
         \node[] at (0,1.5){ $\scriptstyle{\mathbf{3}}$};  
           \draw[ thick, black, fill =pink!50!yellow!80 ] (-2,3) circle (10pt);
         \node[] at (-2,3){ $\scriptstyle{\mathbf{1}}$};  
    %%%%%%%%
        \draw[ thick, black, fill = white] (-1,0) circle (13pt);
     \draw[ thick, black, fill = pink!50!yellow!80] (-1,0) circle (10pt);
         \node[] at (-1,0){ $\scriptstyle{\mathbf{1}}$};  
                 \draw[ thick, black, fill = white] (1,0) circle (13pt);
     \draw[ thick, black, fill = pink!50!yellow!80] (1,0) circle (10pt);
         \node[] at (1,0){ $\scriptstyle{\mathbf{1}}$};  
         %
         % \draw[ thick, black, fill = violet!30!red!30] (0+8,-1) circle (10pt);
            % \draw[ thick, black, fill = pink!50!yellow!80] (-2+8,0.5) circle (10pt);
           % \draw[ thick, black, fill = lime] (2+8,0.5) circle (10pt);
           % \draw[ thick, black, fill = cyan!30] (0+8,0.5) circle (10pt);
\end{tikzpicture}
\qquad
\begin{tikzpicture}[scale=0.42]
\draw[ultra thick, dotted, lightgray] (-1.11,0)--(1,0);
 \draw[thick,->, shorten >=0.16cm]  (-1,0)--(-2,1.5);
  \draw[thick,->, shorten >=0.16cm]  (-1,0)--(0,1.5);
    \draw[thick,->, shorten >=0.16cm]  (0,1.5)--(0,3);
    %%%%%%%%
         \draw[ thick, black, fill =  cyan!30] (-2,1.5) circle (10pt);
         \node[] at (-2,1.5){ $\scriptstyle{\mathbf{2}}$};  
           \draw[ thick, black, fill =violet!30!red!30 ]  (0,1.5) circle (10pt);
         \node[] at (0,1.5){ $\scriptstyle{\mathbf{3}}$};  
           \draw[ thick, black, fill =pink!50!yellow!80 ] (0,3) circle (10pt);
         \node[] at (0,3){ $\scriptstyle{\mathbf{1}}$};  
    %%%%%%%%
        \draw[ thick, black, fill = white] (-1,0) circle (13pt);
     \draw[ thick, black, fill = pink!50!yellow!80] (-1,0) circle (10pt);
         \node[] at (-1,0){ $\scriptstyle{\mathbf{1}}$};  
                 \draw[ thick, black, fill = white] (1,0) circle (13pt);
     \draw[ thick, black, fill = pink!50!yellow!80] (1,0) circle (10pt);
         \node[] at (1,0){ $\scriptstyle{\mathbf{1}}$};  
         %
         % \draw[ thick, black, fill = violet!30!red!30] (0+8,-1) circle (10pt);
            % \draw[ thick, black, fill = pink!50!yellow!80] (-2+8,0.5) circle (10pt);
           % \draw[ thick, black, fill = lime] (2+8,0.5) circle (10pt);
           % \draw[ thick, black, fill = cyan!30] (0+8,0.5) circle (10pt);
\end{tikzpicture}
\qquad
\begin{tikzpicture}[scale=0.42]
\draw[ultra thick, dotted, lightgray] (-1.11,0)--(1,0);
 \draw[thick,->, shorten >=0.16cm]  (-1,0)--(-1,1.5);
  \draw[thick,->, shorten >=0.16cm]  (-1,1.5)--(-1,3);
    \draw[thick,->, shorten >=0.16cm]  (1,0)--(1,1.5);
    %%%%%%%%
         \draw[ thick, black, fill = cyan!30] (-1,1.5) circle (10pt);
         \node[] at (-1,1.5){ $\scriptstyle{\mathbf{2}}$};  
           \draw[ thick, black, fill = pink!50!yellow!80] (-1,3) circle (10pt);
         \node[] at (-1,3){ $\scriptstyle{\mathbf{1}}$};  
           \draw[ thick, black, fill = violet!30!red!30] (1,1.5) circle (10pt);
         \node[] at (1,1.5){ $\scriptstyle{\mathbf{3}}$};  
    %%%%%%%%
        \draw[ thick, black, fill = white] (-1,0) circle (13pt);
     \draw[ thick, black, fill = pink!50!yellow!80] (-1,0) circle (10pt);
         \node[] at (-1,0){ $\scriptstyle{\mathbf{1}}$};  
                 \draw[ thick, black, fill = white] (1,0) circle (13pt);
     \draw[ thick, black, fill = pink!50!yellow!80] (1,0) circle (10pt);
         \node[] at (1,0){ $\scriptstyle{\mathbf{1}}$};  
         %
         % \draw[ thick, black, fill = violet!30!red!30] (0+8,-1) circle (10pt);
            % \draw[ thick, black, fill = pink!50!yellow!80] (-2+8,0.5) circle (10pt);
           % \draw[ thick, black, fill = lime] (2+8,0.5) circle (10pt);
           % \draw[ thick, black, fill = cyan!30] (0+8,0.5) circle (10pt);
\end{tikzpicture}
\qquad
\begin{tikzpicture}[scale=0.42]
\draw[ultra thick, dotted, lightgray] (-1.11,0)--(1,0);
 \draw[thick,->, shorten >=0.16cm]  (-1,0)--(-1,1.5);
  \draw[thick,->, shorten >=0.16cm]  (-1,1.5)--(-1,3);
    \draw[thick,->, shorten >=0.16cm]  (1,0)--(1,1.5);
    %%%%%%%%
         \draw[ thick, black, fill =violet!30!red!30 ] (-1,1.5) circle (10pt);
         \node[] at (-1,1.5){ $\scriptstyle{\mathbf{3}}$};  
           \draw[ thick, black, fill = pink!50!yellow!80] (-1,3) circle (10pt);
         \node[] at (-1,3){ $\scriptstyle{\mathbf{1}}$};  
           \draw[ thick, black, fill = cyan!30 ] (1,1.5) circle (10pt);
         \node[] at (1,1.5){ $\scriptstyle{\mathbf{2}}$};  
    %%%%%%%%
        \draw[ thick, black, fill = white] (-1,0) circle (13pt);
     \draw[ thick, black, fill = pink!50!yellow!80] (-1,0) circle (10pt);
         \node[] at (-1,0){ $\scriptstyle{\mathbf{1}}$};  
                 \draw[ thick, black, fill = white] (1,0) circle (13pt);
     \draw[ thick, black, fill = pink!50!yellow!80] (1,0) circle (10pt);
         \node[] at (1,0){ $\scriptstyle{\mathbf{1}}$};  
         %
         % \draw[ thick, black, fill = violet!30!red!30] (0+8,-1) circle (10pt);
            % \draw[ thick, black, fill = pink!50!yellow!80] (-2+8,0.5) circle (10pt);
           % \draw[ thick, black, fill = lime] (2+8,0.5) circle (10pt);
           % \draw[ thick, black, fill = cyan!30] (0+8,0.5) circle (10pt);
\end{tikzpicture}
\end{align*}
\begin{align*}
\begin{tikzpicture}[scale=0.42]
\draw[ultra thick, dotted, lightgray] (-1.11,0)--(1,0);
 \draw[thick,->, shorten >=0.16cm]  (1,0)--(1,1.5);
  \draw[thick,->, shorten >=0.16cm]  (1,1.5)--(0,3);
    \draw[thick,->, shorten >=0.16cm]  (1,1.5)--(2,3);
    %%%%%%%%
         \draw[ thick, black, fill = cyan!30 ] (1,1.5) circle (10pt);
         \node[] at (1,1.5){ $\scriptstyle{\mathbf{2}}$};  
           \draw[ thick, black, fill = pink!50!yellow!80] (0,3) circle (10pt);
         \node[] at (0,3){ $\scriptstyle{\mathbf{1}}$};  
           \draw[ thick, black, fill = violet!30!red!30 ] (2,3) circle (10pt);
         \node[] at (2,3){ $\scriptstyle{\mathbf{3}}$};  
    %%%%%%%%
        \draw[ thick, black, fill = white] (-1,0) circle (13pt);
     \draw[ thick, black, fill = pink!50!yellow!80] (-1,0) circle (10pt);
         \node[] at (-1,0){ $\scriptstyle{\mathbf{1}}$};  
                 \draw[ thick, black, fill = white] (1,0) circle (13pt);
     \draw[ thick, black, fill = pink!50!yellow!80] (1,0) circle (10pt);
         \node[] at (1,0){ $\scriptstyle{\mathbf{1}}$};  
         %
         % \draw[ thick, black, fill = violet!30!red!30] (0+8,-1) circle (10pt);
            % \draw[ thick, black, fill = pink!50!yellow!80] (-2+8,0.5) circle (10pt);
           % \draw[ thick, black, fill = lime] (2+8,0.5) circle (10pt);
           % \draw[ thick, black, fill = cyan!30] (0+8,0.5) circle (10pt);
\end{tikzpicture}
\qquad
\begin{tikzpicture}[scale=0.42]
\draw[ultra thick, dotted, lightgray] (-1.11,0)--(1,0);
 \draw[thick,->, shorten >=0.16cm]  (1,0)--(1,1.5);
  \draw[thick,->, shorten >=0.16cm]  (1,1.5)--(0,3);
    \draw[thick,->, shorten >=0.16cm]  (1,1.5)--(2,3);
    %%%%%%%%
         \draw[ thick, black, fill =  violet!30!red!30] (1,1.5) circle (10pt);
         \node[] at (1,1.5){ $\scriptstyle{\mathbf{3}}$};  
           \draw[ thick, black, fill = pink!50!yellow!80] (0,3) circle (10pt);
         \node[] at (0,3){ $\scriptstyle{\mathbf{1}}$};  
           \draw[ thick, black, fill = cyan!30  ] (2,3) circle (10pt);
         \node[] at (2,3){ $\scriptstyle{\mathbf{2}}$};  
    %%%%%%%%
        \draw[ thick, black, fill = white] (-1,0) circle (13pt);
     \draw[ thick, black, fill = pink!50!yellow!80] (-1,0) circle (10pt);
         \node[] at (-1,0){ $\scriptstyle{\mathbf{1}}$};  
                 \draw[ thick, black, fill = white] (1,0) circle (13pt);
     \draw[ thick, black, fill = pink!50!yellow!80] (1,0) circle (10pt);
         \node[] at (1,0){ $\scriptstyle{\mathbf{1}}$};  
         %
         % \draw[ thick, black, fill = violet!30!red!30] (0+8,-1) circle (10pt);
            % \draw[ thick, black, fill = pink!50!yellow!80] (-2+8,0.5) circle (10pt);
           % \draw[ thick, black, fill = lime] (2+8,0.5) circle (10pt);
           % \draw[ thick, black, fill = cyan!30] (0+8,0.5) circle (10pt);
\end{tikzpicture}
\qquad
\begin{tikzpicture}[scale=0.42]
\draw[ultra thick, dotted, lightgray] (-1.11,0)--(1,0);
 \draw[thick,->, shorten >=0.16cm]  (1,0)--(2,1.5);
  \draw[thick,->, shorten >=0.16cm]  (1,0)--(0,1.5);
    \draw[thick,->, shorten >=0.16cm]  (2,1.5)--(2,3);
    %%%%%%%%
         \draw[ thick, black, fill = violet!30!red!30 ] (2,1.5) circle (10pt);
         \node[] at (2,1.5){ $\scriptstyle{\mathbf{3}}$};  
           \draw[ thick, black, fill =cyan!30  ]  (0,1.5) circle (10pt);
         \node[] at (0,1.5){ $\scriptstyle{\mathbf{2}}$};  
           \draw[ thick, black, fill =pink!50!yellow!80 ] (2,3) circle (10pt);
         \node[] at (2,3){ $\scriptstyle{\mathbf{1}}$};  
    %%%%%%%%
        \draw[ thick, black, fill = white] (1,0) circle (13pt);
     \draw[ thick, black, fill = pink!50!yellow!80] (1,0) circle (10pt);
         \node[] at (1,0){ $\scriptstyle{\mathbf{1}}$};  
                 \draw[ thick, black, fill = white] (-1,0) circle (13pt);
     \draw[ thick, black, fill = pink!50!yellow!80] (-1,0) circle (10pt);
         \node[] at (-1,0){ $\scriptstyle{\mathbf{1}}$};  
         %
         % \draw[ thick, black, fill = violet!30!red!30] (0+8,1) circle (10pt);
            % \draw[ thick, black, fill = pink!50!yellow!80] (-2+8,0.5) circle (10pt);
           % \draw[ thick, black, fill = lime] (2+8,0.5) circle (10pt);
           % \draw[ thick, black, fill = cyan!30] (0+8,0.5) circle (10pt);
\end{tikzpicture}
\qquad
\begin{tikzpicture}[scale=0.42]
\draw[ultra thick, dotted, lightgray] (-1.11,0)--(1,0);
 \draw[thick,->, shorten >=0.16cm]  (1,0)--(2,1.5);
  \draw[thick,->, shorten >=0.16cm]  (1,0)--(0,1.5);
    \draw[thick,->, shorten >=0.16cm]  (0,1.5)--(0,3);
    %%%%%%%%
         \draw[ thick, black, fill = violet!30!red!30 ] (2,1.5) circle (10pt);
         \node[] at (2,1.5){ $\scriptstyle{\mathbf{3}}$};  
           \draw[ thick, black, fill = cyan!30  ]  (0,1.5) circle (10pt);
         \node[] at (0,1.5){ $\scriptstyle{\mathbf{2}}$};  
           \draw[ thick, black, fill =pink!50!yellow!80 ] (0,3) circle (10pt);
         \node[] at (0,3){ $\scriptstyle{\mathbf{1}}$};  
    %%%%%%%%
        \draw[ thick, black, fill = white] (1,0) circle (13pt);
     \draw[ thick, black, fill = pink!50!yellow!80] (1,0) circle (10pt);
         \node[] at (1,0){ $\scriptstyle{\mathbf{1}}$};  
                 \draw[ thick, black, fill = white] (-1,0) circle (13pt);
     \draw[ thick, black, fill = pink!50!yellow!80] (-1,0) circle (10pt);
         \node[] at (-1,0){ $\scriptstyle{\mathbf{1}}$};  
         %
         % \draw[ thick, black, fill = violet!30!red!30] (0+8,-1) circle (10pt);
            % \draw[ thick, black, fill = pink!50!yellow!80] (-2+8,0.5) circle (10pt);
           % \draw[ thick, black, fill = lime] (2+8,0.5) circle (10pt);
           % \draw[ thick, black, fill = cyan!30] (0+8,0.5) circle (10pt);
\end{tikzpicture}
\qquad
\begin{tikzpicture}[scale=0.42]
\draw[ultra thick, dotted, lightgray] (-1.11,0)--(1,0);
 \draw[thick,->, shorten >=0.16cm]  (1,0)--(1,1.5);
  \draw[thick,->, shorten >=0.16cm]  (1,1.5)--(1,3);
    \draw[thick,->, shorten >=0.16cm]  (-1,0)--(-1,1.5);
    %%%%%%%%
         \draw[ thick, black, fill =violet!30!red!30 ] (1,1.5) circle (10pt);
         \node[] at (1,1.5){ $\scriptstyle{\mathbf{3}}$};  
           \draw[ thick, black, fill = pink!50!yellow!80] (1,3) circle (10pt);
         \node[] at (1,3){ $\scriptstyle{\mathbf{1}}$};  
           \draw[ thick, black, fill = cyan!30 ] (-1,1.5) circle (10pt);
         \node[] at (-1,1.5){ $\scriptstyle{\mathbf{2}}$};  
    %%%%%%%%
        \draw[ thick, black, fill = white] (-1,0) circle (13pt);
     \draw[ thick, black, fill = pink!50!yellow!80] (-1,0) circle (10pt);
         \node[] at (-1,0){ $\scriptstyle{\mathbf{1}}$};  
                 \draw[ thick, black, fill = white] (1,0) circle (13pt);
     \draw[ thick, black, fill = pink!50!yellow!80] (1,0) circle (10pt);
         \node[] at (1,0){ $\scriptstyle{\mathbf{1}}$};  
         %
         % \draw[ thick, black, fill = violet!30!red!30] (0+8,-1) circle (10pt);
            % \draw[ thick, black, fill = pink!50!yellow!80] (-2+8,0.5) circle (10pt);
           % \draw[ thick, black, fill = lime] (2+8,0.5) circle (10pt);
           % \draw[ thick, black, fill = cyan!30] (0+8,0.5) circle (10pt);
\end{tikzpicture}
\qquad
\begin{tikzpicture}[scale=0.42]
\draw[ultra thick, dotted, lightgray] (-1.11,0)--(1,0);
 \draw[thick,->, shorten >=0.16cm]  (1,0)--(1,1.5);
  \draw[thick,->, shorten >=0.16cm]  (1,1.5)--(1,3);
    \draw[thick,->, shorten >=0.16cm]  (-1,0)--(-1,1.5);
    %%%%%%%%
         \draw[ thick, black, fill = cyan!30 ] (1,1.5) circle (10pt);
         \node[] at (1,1.5){ $\scriptstyle{\mathbf{2}}$};  
           \draw[ thick, black, fill = pink!50!yellow!80] (1,3) circle (10pt);
         \node[] at (1,3){ $\scriptstyle{\mathbf{1}}$};  
           \draw[ thick, black, fill = violet!30!red!30 ] (-1,1.5) circle (10pt);
         \node[] at (-1,1.5){ $\scriptstyle{\mathbf{3}}$};  
    %%%%%%%%
        \draw[ thick, black, fill = white] (-1,0) circle (13pt);
     \draw[ thick, black, fill = pink!50!yellow!80] (-1,0) circle (10pt);
         \node[] at (-1,0){ $\scriptstyle{\mathbf{1}}$};  
                 \draw[ thick, black, fill = white] (1,0) circle (13pt);
     \draw[ thick, black, fill = pink!50!yellow!80] (1,0) circle (10pt);
         \node[] at (1,0){ $\scriptstyle{\mathbf{1}}$};  
         %
         % \draw[ thick, black, fill = violet!30!red!30] (0+8,-1) circle (10pt);
            % \draw[ thick, black, fill = pink!50!yellow!80] (-2+8,0.5) circle (10pt);
           % \draw[ thick, black, fill = lime] (2+8,0.5) circle (10pt);
           % \draw[ thick, black, fill = cyan!30] (0+8,0.5) circle (10pt);
\end{tikzpicture}
\end{align*}
\begin{align*}
\begin{tikzpicture}[scale=0.42]
\draw[ultra thick, dotted, lightgray] (-1.11,0)--(1,0);
 \draw[thick,->, shorten >=0.16cm]  (1,0)--(1,1.5);
  \draw[thick,->, shorten >=0.16cm]  (1,1.5)--(1,3);
    \draw[thick,->, shorten >=0.16cm]  (1,3)--(1,4.5);
    %%%%%%%%
         \draw[ thick, black, fill = cyan!30] (1,1.5) circle (10pt);
         \node[] at (1,1.5){ $\scriptstyle{\mathbf{2}}$};  
           \draw[ thick, black, fill = pink!50!yellow!80] (1,3) circle (10pt);
         \node[] at (1,3){ $\scriptstyle{\mathbf{1}}$};  
           \draw[ thick, black, fill = violet!30!red!30] (1,4.5) circle (10pt);
         \node[] at (1,4.5){ $\scriptstyle{\mathbf{3}}$};  
    %%%%%%%%
        \draw[ thick, black, fill = white] (1,0) circle (13pt);
     \draw[ thick, black, fill = pink!50!yellow!80] (1,0) circle (10pt);
         \node[] at (1,0){ $\scriptstyle{\mathbf{1}}$};  
                 \draw[ thick, black, fill = white] (-1,0) circle (13pt);
     \draw[ thick, black, fill = pink!50!yellow!80] (-1,0) circle (10pt);
         \node[] at (-1,0){ $\scriptstyle{\mathbf{1}}$};  
         %
         % \draw[ thick, black, fill = violet!30!red!30] (0+8,-1) circle (10pt);
            % \draw[ thick, black, fill = pink!50!yellow!80] (-2+8,0.5) circle (10pt);
           % \draw[ thick, black, fill = lime] (2+8,0.5) circle (10pt);
           % \draw[ thick, black, fill = cyan!30] (0+8,0.5) circle (10pt);
\end{tikzpicture}
\qquad\begin{tikzpicture}[scale=0.42]
\draw[ultra thick, dotted, lightgray] (-1.11,0)--(1,0);
 \draw[thick,->, shorten >=0.16cm]  (1,0)--(1,1.5);
  \draw[thick,->, shorten >=0.16cm]  (1,1.5)--(1,3);
    \draw[thick,->, shorten >=0.16cm]  (1,3)--(1,4.5);
    %%%%%%%%
         \draw[ thick, black, fill = cyan!30] (1,1.5) circle (10pt);
         \node[] at (1,1.5){ $\scriptstyle{\mathbf{2}}$};  
           \draw[ thick, black, fill = violet!30!red!30] (1,3) circle (10pt);
         \node[] at (1,3){ $\scriptstyle{\mathbf{3}}$};  
           \draw[ thick, black, fill = pink!50!yellow!80] (1,4.5) circle (10pt);
         \node[] at (1,4.5){ $\scriptstyle{\mathbf{1}}$};  
    %%%%%%%%
        \draw[ thick, black, fill = white] (1,0) circle (13pt);
     \draw[ thick, black, fill = pink!50!yellow!80] (1,0) circle (10pt);
         \node[] at (1,0){ $\scriptstyle{\mathbf{1}}$};  
                 \draw[ thick, black, fill = white] (-1,0) circle (13pt);
     \draw[ thick, black, fill = pink!50!yellow!80] (-1,0) circle (10pt);
         \node[] at (-1,0){ $\scriptstyle{\mathbf{1}}$};  
         %
         % \draw[ thick, black, fill = violet!30!red!30] (0+8,1) circle (10pt);
            % \draw[ thick, black, fill = pink!50!yellow!80] (-2+8,0.5) circle (10pt);
           % \draw[ thick, black, fill = lime] (2+8,0.5) circle (10pt);
           % \draw[ thick, black, fill = cyan!30] (0+8,0.5) circle (10pt);
\end{tikzpicture}
\qquad\begin{tikzpicture}[scale=0.42]
\draw[ultra thick, dotted, lightgray] (-1.11,0)--(1,0);
 \draw[thick,->, shorten >=0.16cm]  (1,0)--(1,1.5);
  \draw[thick,->, shorten >=0.16cm]  (1,1.5)--(1,3);
    \draw[thick,->, shorten >=0.16cm]  (1,3)--(1,4.5);
    %%%%%%%%
         \draw[ thick, black, fill =  violet!30!red!30] (1,1.5) circle (10pt);
         \node[] at (1,1.5){ $\scriptstyle{\mathbf{3}}$};  
           \draw[ thick, black, fill = pink!50!yellow!80] (1,3) circle (10pt);
         \node[] at (1,3){ $\scriptstyle{\mathbf{1}}$};  
           \draw[ thick, black, fill = cyan!30] (1,4.5) circle (10pt);
         \node[] at (1,4.5){ $\scriptstyle{\mathbf{2}}$};  
    %%%%%%%%
        \draw[ thick, black, fill = white] (1,0) circle (13pt);
     \draw[ thick, black, fill = pink!50!yellow!80] (1,0) circle (10pt);
         \node[] at (1,0){ $\scriptstyle{\mathbf{1}}$};  
                 \draw[ thick, black, fill = white] (-1,0) circle (13pt);
     \draw[ thick, black, fill = pink!50!yellow!80] (-1,0) circle (10pt);
         \node[] at (-1,0){ $\scriptstyle{\mathbf{1}}$};  
         %
         % \draw[ thick, black, fill = violet!30!red!30] (0+8,1) circle (10pt);
            % \draw[ thick, black, fill = pink!50!yellow!80] (-2+8,0.5) circle (10pt);
           % \draw[ thick, black, fill = lime] (2+8,0.5) circle (10pt);
           % \draw[ thick, black, fill = cyan!30] (0+8,0.5) circle (10pt);
\end{tikzpicture}
\qquad\begin{tikzpicture}[scale=0.42]
\draw[ultra thick, dotted, lightgray] (-1.11,0)--(1,0);
 \draw[thick,->, shorten >=0.16cm]  (1,0)--(1,1.5);
  \draw[thick,->, shorten >=0.16cm]  (1,1.5)--(1,3);
    \draw[thick,->, shorten >=0.16cm]  (1,3)--(1,4.5);
    %%%%%%%%
         \draw[ thick, black, fill =  violet!30!red!30] (1,1.5) circle (10pt);
         \node[] at (1,1.5){ $\scriptstyle{\mathbf{3}}$};  
           \draw[ thick, black, fill = cyan!30] (1,3) circle (10pt);
         \node[] at (1,3){ $\scriptstyle{\mathbf{2}}$};  
           \draw[ thick, black, fill =pink!50!yellow!80 ] (1,4.5) circle (10pt);
         \node[] at (1,4.5){ $\scriptstyle{\mathbf{1}}$};  
    %%%%%%%%
        \draw[ thick, black, fill = white] (1,0) circle (13pt);
     \draw[ thick, black, fill = pink!50!yellow!80] (1,0) circle (10pt);
         \node[] at (1,0){ $\scriptstyle{\mathbf{1}}$};  
                 \draw[ thick, black, fill = white] (-1,0) circle (13pt);
     \draw[ thick, black, fill = pink!50!yellow!80] (-1,0) circle (10pt);
         \node[] at (-1,0){ $\scriptstyle{\mathbf{1}}$};  
         %
         % \draw[ thick, black, fill = violet!30!red!30] (0+8,1) circle (10pt);
            % \draw[ thick, black, fill = pink!50!yellow!80] (-2+8,0.5) circle (10pt);
           % \draw[ thick, black, fill = lime] (2+8,0.5) circle (10pt);
           % \draw[ thick, black, fill = cyan!30] (0+8,0.5) circle (10pt);
\end{tikzpicture}
\end{align*}
\caption{All twenty non-isomorphic \(((3,1,1),(1,1))\)-forests.}
\label{fig:20for}       % Give a unique label
\end{figure}

\begin{Remark}
It follows from definitions (and we will implicitly use the fact throughout) that if \(F\) is a \((\lambda, \underline{c})\)-forest, then \(v \in V_F\) will have in-degree 0 if and only if \(v\) is a root, and \(v\) will have in-degree \(1\) otherwise. The out-degree of \(v\) may range from 0 to \(k-1\).
\end{Remark}

\section{Enumerating \((\lambda, \underline c)\)-forests}

\subsection{A necessary criterion for \((\lambda, \underline{c})\)-forests}
For \(\underline{c} \in [1,k]^m\), \(i \in [1,k]\), we set \[i_{\underline c} = \#\{j \in [1,m] \mid c_j = i\}.\] It is easy to see that \(f_{\lambda,\underline c} = f_{\lambda,\underline d}\) provided that \(i_{\underline c} = i_{\underline d}\) for all \(i \in [1,k]\). 
\begin{Lemma}\label{nkreqLem} For \(\lambda \in \Lambda_k\) and \(\underline{c} \in [1,k]^m\), we have \(f_{\lambda,\underline c} \neq 0\) only if 
\begin{align}\label{nkreq}
0 \leq \lambda_i - i_{\underline c} \leq |\lambda| -\lambda_i\;\; \textup{ for }i \in [1,k].
\end{align}
\end{Lemma}
\begin{proof}
The left inequality is immediate from definitions. For the right inequality, assume \(F\) is a \((\lambda,\underline{c})\)-forest. Let \(i \in [1,k]\), and set \(X_i= (V_F)_i \backslash \mathsf{r}([1,m])\); i.e., \(X_i\) is the the set of \(i\)-colored non-roots in \(F\). It follows that  \(|X_i| = \lambda_i - i_{\underline{c}}\). Therefore \(\left|\bigcup_{x \in X}N^-(x)\right| = \lambda_i - i_{\underline{c}}\) as well, since every vertex in \(X_i\) is the target of a single arrow, and no single vertex can be the source of two distinct arrows which target elements of \(X_i\) because \(F\) is injectively \(k\)-colored. Moreover \(\bigcup_{x \in X_i}N^-(x)\) must consist entirely of non-\(i\)-colored vertices, of which there are \(|\lambda|- \lambda_i\) in total, giving the result.
\end{proof}

\subsection{A recurrence for \((\lambda, \underline c)\)-forests}

%We write \(\mathcal{F}^{R,Y,B}_{R_0, Y_0, B_0}\) for the subset of directed forests in \(\mathcal{D}^{R,Y,B}_{R_0, Y_0, B_0}\). We likewise set \(f^{\,r,y,b}_{r_0, y_0,b_0} := \big| \mathcal{F}^{R,Y,B}_{R_0, Y_0, B_0} \big|\). In studying this value, we begin by demonstrating a recursion:

For \(\underline{c} \in [1,k]^m\) and \(S = \{s_1 < \cdots < s_{|S|}\} \subseteq [1,k-1]\), it will be useful to denote:
\begin{align*}
\underline{c}(S) := (c_1, \ldots, c_{m-1}, s_1, \ldots, s_{|S|}) \in [1,k]^{m-1 + |S|}.
\end{align*}

\begin{Lemma}\label{recurf}
Assume \(\lambda \in \Lambda_k\) and \(\underline{c} \in [1,k]^m\) satisfy (\ref{nkreq}), with \(c_m = k\). Then we have
\begin{align*}
f_{\lambda,\underline c}
=
\sum_{S \subseteq [1,k-1]}
f_{\lambda- \varepsilon_k,\underline{c}(S)}.
\end{align*}
\end{Lemma}
\begin{proof}
Let \(F\) be a \((\lambda, \underline{c})\)-forest. Assume \(N^+(\mathsf{r}_F(m)) = \{x_1, \ldots, x_{|X|}\}\), with elements labeled so that \(\ell_F(x_1) < \cdots < \ell_F(x_{|X|})\). Note that \(S:=\{\ell_F(x) \mid x \in X\}\) is a subset of \([1,k-1]\) of cardinality \(|X|\). Then it is straightforward to check that setting:
\begin{align*}
V_{F'} = V_F \backslash \{\mathsf{r}(m)\};
\qquad
A' = A \cap (V_{F'} \times V_{F'});
\qquad
\ell_{F'} = \ell_F|_{V_{F'}};
\end{align*}
\begin{align*}
\mathsf{r}_{F'}(j) =
\begin{cases}
\mathsf{r}_F(j) & \textup{if }1 \leq j \leq m-1;\\
x_{j-m+1} & \textup{if } m \leq j \leq m-1+|X|,
\end{cases}
\end{align*}
defines a \((\lambda - \varepsilon_k, \underline{c}(S))\)-forest \(F'\). Informally, \(F'\) is achieved by deleting from \(F\) the \(k\)-colored \(m\)th root and all arrows from this root; see Figure~\ref{fig:recurfun} for a depiction. The assignment \(\theta: F \mapsto F'\) defines a function 
\begin{align*}
\theta:\mathcal{F}_{\lambda,\underline c} \to \bigcup_{S \subseteq [1,k-1]} \mathcal{F}_{\lambda- \varepsilon_k,\underline{c}(S)},
\end{align*}
noting that \(\theta\) is well-defined on isomorphism classes of \((\lambda, \underline{c})\)-forests.

\begin{figure}[h]
\begin{align*}
{}
\hackcenter{
\begin{tikzpicture}[scale=0.42]
\draw[ultra thick, dotted, lightgray] (-3,-1)--(2,-1);
  %%%%%%%%%%%%%%%%%%%%%%%%%%%%%%
  %%%%%%%%%%%%%%%%%%%%%%%%%%%%%%
  %%%%%%%%%%%%%NEXT TREE%%%%%%%%%%%%%%%%
  %%%%%%%%%%%%%%%%%%%%%%%%%%%%%%
  %%%%%%%%%%%%%%%%%%%%%%%%%%%%%%
   \draw[thick,->, shorten >=0.16cm]  (0-3,-1)--(1-3,0.5);
      \draw[thick,->, shorten >=0.16cm]  (0-3,-1)--(-1-3,0.5);
         \draw[thick,->, shorten >=0.16cm]  (-1-3,0.5)--(-2-3,2);
            \draw[thick,->, shorten >=0.16cm]  (-1-3,0.5)--(0-3,2);
       \draw[ thick, black, fill = white] (0-3,-1) circle (13pt);
     \draw[ thick, black, fill = cyan!30] (0-3,-1) circle (10pt);
         \node[] at (0-3,-1){ $\scriptstyle{\mathbf{2}}$};  
           \draw[ thick, black, fill = violet!30!red!30] (1-3,0.5) circle (10pt);
         \node[] at (1-3,0.5){ $\scriptstyle{\mathbf{3}}$};
            \draw[ thick, black, fill = orange!50] (-1-3,0.5) circle (10pt);
         \node[] at (-1-3,0.5){ $\scriptstyle{\mathbf{5}}$};    
            \draw[ thick, black, fill =cyan!30] (-2-3,2)circle (10pt);
         \node[] at (-2-3,2){ $\scriptstyle{\mathbf{2}}$};    
            \draw[ thick, black, fill = lime] (0-3,2) circle (10pt);
         \node[] at (0-3,2){ $\scriptstyle{\mathbf{4}}$};     
        %%%%%%%%%%%%%%%%%%%%%%%%%%%%%%
  %%%%%%%%%%%%%%%%%%%%%%%%%%%%%%
  %%%%%%%%%%%%%NEXT TREE%%%%%%%%%%%%%%%%
  %%%%%%%%%%%%%%%%%%%%%%%%%%%%%%
  %%%%%%%%%%%%%%%%%%%%%%%%%%%%%%    
     \draw[thick,->, red, shorten >=0.16cm]  (0+2,-1)--(-2+2,0.5);
      \draw[thick,->, red, shorten >=0.16cm]  (0+2,-1)--(-0+2,0.5);
            \draw[thick,->, red, shorten >=0.16cm]  (0+2,-1)--(2+2,0.5);
             \draw[thick,->, shorten >=0.16cm]  (0+2,0.5)--(-0+2,2);
             \draw[thick,->, shorten >=0.16cm]  (-0+2,2)--(-2+2,3.5);
                   \draw[thick,->, shorten >=0.16cm]  (-0+2,2)--(-0+2,3.5);
                         \draw[thick,->, shorten >=0.16cm]  (-0+2,2)--(2+2,3.5);
                           \draw[thick,->, shorten >=0.16cm]  (2+2,0.5)--(2+2,2);
           \draw[ thick, red, fill = white] (0+2,-1) circle (13pt);
          \draw[ thick, red, fill = orange!50] (0+2,-1) circle (10pt);
         \node[] at (0+2,-1){ $\scriptstyle{\mathbf{5}}$}; 
             \draw[ thick, black, fill = pink!50!yellow!80] (-2+2,0.5) circle (10pt);
         \node[] at (-2+2,0.5){ $\scriptstyle{\mathbf{1}}$}; 
            \draw[ thick, black, fill = lime] (2+2,0.5) circle (10pt);
         \node[] at (2+2,0.5){ $\scriptstyle{\mathbf{4}}$}; 
            \draw[ thick, black, fill = cyan!30] (0+2,0.5) circle (10pt);
         \node[] at (0+2,0.5){ $\scriptstyle{\mathbf{2}}$}; 
            \draw[ thick, black, fill = violet!30!red!30] (2+2,2) circle (10pt);
         \node[] at (2+2,2){ $\scriptstyle{\mathbf{3}}$}; 
           \draw[ thick, black, fill = violet!30!red!30] (0+2,3.5) circle (10pt);
         \node[] at (0+2,3.5){ $\scriptstyle{\mathbf{3}}$}; 
          \draw[ thick, black, fill = pink!50!yellow!80] (0+2,2) circle (10pt);
         \node[] at (0+2,2){ $\scriptstyle{\mathbf{1}}$}; 
               \draw[ thick, black, fill = lime] (2+2,3.5) circle (10pt);
         \node[] at (2+2,3.5){ $\scriptstyle{\mathbf{4}}$}; 
               \draw[ thick, black, fill = cyan!30] (-2+2,3.5) circle (10pt);
         \node[] at (-2+2,3.5){ $\scriptstyle{\mathbf{2}}$}; 
\end{tikzpicture}
}
\qquad
\substack{\displaystyle\theta\\
\displaystyle \mapsto
}
\qquad
\hackcenter{
\begin{tikzpicture}[scale=0.42]
\draw[ultra thick, dotted, lightgray] (-2,-1)--(4,-1);
  %%%%%%%%%%%%%%%%%%%%%%%%%%%%%%
  %%%%%%%%%%%%%%%%%%%%%%%%%%%%%%
  %%%%%%%%%%%%%NEXT TREE%%%%%%%%%%%%%%%%
  %%%%%%%%%%%%%%%%%%%%%%%%%%%%%%
  %%%%%%%%%%%%%%%%%%%%%%%%%%%%%%
   \draw[thick,->, shorten >=0.16cm]  (0-2,-1)--(1-2,0.5);
      \draw[thick,->, shorten >=0.16cm]  (0-2,-1)--(-1-2,0.5);
         \draw[thick,->, shorten >=0.16cm]  (-1-2,0.5)--(-2-2,2);
            \draw[thick,->, shorten >=0.16cm]  (-1-2,0.5)--(0-2,2);
       \draw[ thick, black, fill = white] (0-2,-1) circle (13pt);
     \draw[ thick, black, fill = cyan!30] (0-2,-1) circle (10pt);
         \node[] at (0-2,-1){ $\scriptstyle{\mathbf{2}}$};  
           \draw[ thick, black, fill = violet!30!red!30] (1-2,0.5) circle (10pt);
         \node[] at (1-2,0.5){ $\scriptstyle{\mathbf{3}}$};
            \draw[ thick, black, fill = orange!50] (-1-2,0.5) circle (10pt);
         \node[] at (-1-2,0.5){ $\scriptstyle{\mathbf{5}}$};    
            \draw[ thick, black, fill =cyan!30] (-2-2,2)circle (10pt);
         \node[] at (-2-2,2){ $\scriptstyle{\mathbf{2}}$};    
            \draw[ thick, black, fill = lime] (0-2,2) circle (10pt);
         \node[] at (0-2,2){ $\scriptstyle{\mathbf{4}}$};     
        %%%%%%%%%%%%%%%%%%%%%%%%%%%%%%
  %%%%%%%%%%%%%%%%%%%%%%%%%%%%%%
  %%%%%%%%%%%%%NEXT TREE%%%%%%%%%%%%%%%%
  %%%%%%%%%%%%%%%%%%%%%%%%%%%%%%
  %%%%%%%%%%%%%%%%%%%%%%%%%%%%%%    
             \draw[thick,->, shorten >=0.16cm]  (0+2,0.5-1.5)--(-0+2,2-1.5);
             \draw[thick,->, shorten >=0.16cm]  (-0+2,2-1.5)--(-2+2,3.5-1.5);
                   \draw[thick,->, shorten >=0.16cm]  (-0+2,2-1.5)--(-0+2,3.5-1.5);
                         \draw[thick,->, shorten >=0.16cm]  (-0+2,2-1.5)--(2+2,3.5-1.5);
                           \draw[thick,->, shorten >=0.16cm]  (2+2,0.5-1.5)--(2+2,2-1.5);
          \draw[ thick, black, fill = white] (-2+2,-1) circle (13pt);
            \draw[ thick, black, fill = white] (0+2,-1) circle (13pt);
           \draw[ thick, black, fill = white] (2+2,-1) circle (13pt);
             \draw[ thick, black, fill = pink!50!yellow!80] (-2+2,0.5-1.5) circle (10pt);
         \node[] at (-2+2,0.5-1.5){ $\scriptstyle{\mathbf{1}}$}; 
            \draw[ thick, black, fill =lime] (2+2,0.5-1.5) circle (10pt);
         \node[] at (2+2,0.5-1.5){ $\scriptstyle{\mathbf{4}}$}; 
            \draw[ thick, black, fill = cyan!30] (0+2,0.5-1.5) circle (10pt);
         \node[] at (0+2,0.5-1.5){ $\scriptstyle{\mathbf{2}}$}; 
            \draw[ thick, black, fill = violet!30!red!30] (2+2,2-1.5) circle (10pt);
         \node[] at (2+2,2-1.5){ $\scriptstyle{\mathbf{3}}$}; 
           \draw[ thick, black, fill = violet!30!red!30] (0+2,3.5-1.5) circle (10pt);
         \node[] at (0+2,3.5-1.5){ $\scriptstyle{\mathbf{3}}$}; 
          \draw[ thick, black, fill = pink!50!yellow!80] (0+2,2-1.5) circle (10pt);
         \node[] at (0+2,2-1.5){ $\scriptstyle{\mathbf{1}}$}; 
               \draw[ thick, black, fill = lime] (2+2,3.5-1.5) circle (10pt);
         \node[] at (2+2,3.5-1.5){ $\scriptstyle{\mathbf{4}}$}; 
               \draw[ thick, black, fill = cyan!30] (-2+2,3.5-1.5) circle (10pt);
         \node[] at (-2+2,3.5-1.5){ $\scriptstyle{\mathbf{2}}$}; 
          \draw[ thick, white, fill = white] (-2+2,3.5) circle (10pt);
\end{tikzpicture}
}
\end{align*}
\caption{On the left, a forest \(F \in \mathcal{F}_{(2, 4, 3, 3, 2),(2,5)}\), on the right \(\theta(F) \in \mathcal{F}_{(2,4,3,3,1),(2,1,2,4)}\). The map \(\theta\) deletes the last root in \(F\) and all arrows from this root.
}
\label{fig:recurfun}       % Give a unique label
\end{figure}
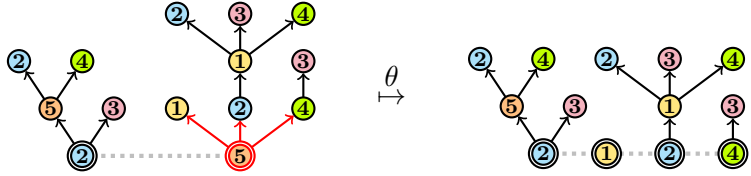

In the other direction, assume \(F'\) is a \((\lambda - \varepsilon_k, \underline{c}(S))\)-forest for some \(S \subseteq [1,k-1]\). Again it is straightforward to check that setting
\begin{align*}
V_F = V_{F'} \cup \{\bullet \};
\qquad
A_F = A_{F'} \cup\{ (\bullet, \mathsf{r}_{F'}(j)) \mid m \leq j \leq m-1 + |S|\};
\end{align*}
\begin{align*}
\mathsf{r}_F =
\begin{cases}
\mathsf{r}_{F'}(j) & \textup{if }1 \leq j \leq m-1;\\
\bullet & \textup{if } j = m;
\end{cases}
\qquad
\ell_F(v) =
\begin{cases}
\ell_{F'}(v) & \textup{if }v \in V_{F'};\\
k & \textup{if }v = \bullet,
\end{cases}
\end{align*}
defines a \((\lambda, \underline{c})\)-forest \(F\). Informally, \(F\) is achieved by adding an additional \(k\)-colored root along with arrows from this root to the last \(|S|\) roots in \(F'\). The assignment \(\theta' : F' \mapsto F\) defines a function
\begin{align*}
\theta': \bigcup_{S \subseteq [1,k-1]} \mathcal{F}_{\lambda- \varepsilon_k,\underline{c}(S)} \to \mathcal{F}_{\lambda,\underline c},
\end{align*}
noting again that \(\theta'\) is well-defined on isomorphism classes.
As constructed, \(\theta, \theta'\) are mutual inverses, so the result follows.
\end{proof}

\subsection{A key polynomial}
Write \(R_k = \mathbb{Z}[x_1, \ldots, x_k, y_1, \ldots, y_k]\) for the ring of integral polynomials in \(2k\) variables. For \(S \in [1,k]\), we define:
\begin{align*}
x_{(S,+)} := \sum_{s \in S} x_s; 
\qquad
x_{(S,\times)} := \prod_{s \in S} x_s;
\qquad
y_{(S,+)} := \sum_{s \in S} y_s; 
\qquad
y_{(S,\times)} := \prod_{s \in S} y_s.
\end{align*}
The polynomial \(P_k \in R_k\) defined in (\ref{ThePoly1}) will play a key role in this paper. Via the shorthand polynomials described above, we may alternatively write this polynomial more compactly:
\begin{align}\label{ThePoly2}
P_k(x_1, \ldots, x_k,y_1, \ldots, y_k) =
\sum_{
\substack{
S \subseteq [1,k]\\
}
}
(-1)^{|S|}
\big(x_{(S,+)}-|S|x_{([1,k],+)}\big)
x_{([1,k],+)}^{k-|S|-1}
y_{(S,\times)}.
\end{align}
For clarity, note that we interpret \(x_{(\varnothing,+)} = 0\), \(x_{(\varnothing, \times)} = 1\), and similarly for \(y\). We also make sense of the \(S = [1,k]\) summand in (\ref{ThePoly2}) by noting that:
\begin{align*}
(-1)^{k}(x_{([1,k],+)} - kx_{([1,k],+)} )x_{([1,k],+)}^{-1}y_{([1,k],\times)}
= (-1)^k(1-k)y_1 \cdots y_k.
\end{align*}

\begin{Example}
We have, for instance, the polynomials
\begin{align*}
P_2(x_1,x_2,y_1,y_2) = x_1 y_2 + x_2 y_1 - y_1 y_2,
\end{align*}
and
\begin{align*}
P_3(x_1, x_2, x_3, y_1, y_2, y_3) &= x_1^2y_2 + x_1^2y_3 + x_2^2y_1 + x_2^2y_3 + x_3^2y_1 + x_3^2y_2 + 2y_1y_2y_3 \\
&
\hspace{5mm}+
x_1x_2y_1 + x_1x_2y_2 + 2x_1x_2y_3
+x_1x_3y_1 + 2x_1x_3y_2 +x_1x_3y_3\\
&
\hspace{10mm}+2x_2x_3y_1 + x_2x_3y_2 + x_2 x_3y_3
-x_1y_1y_2 -x_1y_1y_3 -2x_1y_2y_3\\
&
\hspace{15mm}-x_2y_1y_2 -2x_2y_1y_3 - x_2y_2y_3
-2x_3y_1y_2-x_3y_1y_3 -x_3y_2y_3.
\end{align*}
\end{Example}
For \(I \subseteq [1,k-1]\) we further define:
\begin{align}\label{TweakedPoly}
P^{(I)}_k := P(x_1, \ldots, x_{k-1}, x_k-1, y_1 + \delta_{1 \in I}, \ldots, y_{k-1} + \delta_{(k-1) \in I}, y_k-1) \in R_k.
\end{align}

Parts (i)--(iii) of the following lemma are straightforward to verify from (\ref{ThePoly2}). Part (iv) is much less so, and as the verification of this result is long and technical, we relegate the proof to the appendix \S\ref{AppSec}, and Corollary~\ref{ProLiv} in particular.

\begin{Lemma}\label{ThePolyLem}
Let \(k \in \mathbb{N}\). Then:
\begin{enumerate}
\item We have
\(
P_k(x_1, \ldots, x_k, 0, \ldots, 0) = 0.
\)
\item For all \(1 \leq a< b \leq k\),  we have
\begin{align*}
P_k(\delta_{1 \in \{a,b\}}, \ldots, \delta_{k \in \{a,b\}}, \delta_{1 \in \{a,b\}}y_1, \ldots, \delta_{k \in \{a,b\}}y_k)=2^{k-2}(y_a + y_b - y_ay_b).
\end{align*}
\item For all \(\sigma \in \mathfrak{S}_k\), we have
\begin{align*}
P_k(x_1, \ldots, x_k, y_1, \ldots, y_k) = P_k(x_{\sigma 1}, \ldots, x_{\sigma k}, y_{\sigma 1}, \ldots, y_{\sigma k}).
\end{align*}
\item We have
\begin{align}
P_k \cdot \prod_{i=1}^{k-1}(x_{([1,k],+)}-x_i-1) =\sum_{I \subseteq [1,k-1]}
P^{(I)}_k \cdot
\prod_{i \in I}(x_i - y_i) \hspace{-2mm} \prod_{j \in [1,k-1]\backslash I}(x_{([1,k],+)}-2x_j+y_j).
\label{TheEqn}
\end{align}
\end{enumerate}
\end{Lemma}

\subsection{Enumerating \((\lambda, \underline{c})\)-forests}\label{enumfor}
Now we prove our main theorem.

\begin{Theorem}\label{MainForThm}
Let \(k \in \mathbb{N}\), \(m \in \mathbb{Z}_{\geq 0}\), \(\lambda \in \Lambda_k\), and \(\underline{c} \in [1,k]^m\). If \(|\lambda| = \lambda_i\) for some \(i \in [1,k]\), then \(f_{\lambda,\underline c} = \delta_{m, |\lambda|} \delta_{m, i_{\underline c}}\). Otherwise:
\begin{align}\label{forclosed}
f_{\lambda,\underline c} = P_k(\lambda_1, \ldots, \lambda_k, 1_{\underline c}, \ldots, k_{\underline c})
\prod_{i=1}^k \frac{1}{|\lambda| - \lambda_i}{|\lambda| - \lambda_i \choose \lambda_i - i_{\underline c}}.
\end{align}
\end{Theorem}
\begin{proof}
First assume that \(|\lambda| = \lambda_i\) for some \(i \in [1,k]\). By Lemma~\ref{nkreqLem}, \(f_{\lambda,\underline c} \neq 0\) implies that \(i_{\underline c} = \lambda_i\) and \(j_{\underline c} = 0\) for \(i \neq j \in [1,k]\), and so \(\underline c = i^{\lambda_i}\). On the other hand, \(f_{\lambda,i^{\lambda_i}}=1\) as there is one \((\lambda, i^{\lambda_i})\)-forest which consists of \(\lambda_i\) \(i\)-colored roots and no other vertices. This proves the first claim.

We prove the second claim by induction on \(|\lambda|\). Write \(g_{\lambda,\underline{c}}\) for the right side of (\ref{forclosed}).
The base case is \(|\lambda|= 2\), and we may assume, since \(|\lambda|>\lambda_i\) for all \(i \in [1,k]\), that \(\lambda = \varepsilon_a + \varepsilon_b\) for some \(1 \leq a < b \leq k\). If \(\lambda, \underline{c}\) do not satisfy (\ref{nkreq}), then we have \(f_{\lambda,\underline c} = 0 = g_{\lambda,\underline c}\) by Lemma~\ref{nkreqLem}, so we may assume \(a_{\underline c}, b_{\underline{c}} \in \{0,1\}\) and \(i_{\underline c} = 0\) for all \(i \in [1,k] \backslash \{a,b\}\). It is straightforward to check from the definition of \((\lambda, \underline c)\)-forests that
\begin{align*}
f_{\lambda,\underline c} = \begin{cases}
0 & \textup{if } a_{\underline c} = b_{\underline c}= 0;\\
1 & \textup{otherwise}.
\end{cases}
\end{align*}
On the other hand, by Lemma~\ref{ThePolyLem}(ii) we have
\begin{align*}
g_{\lambda,\underline c}&= 
P_k(\delta_{1 \in \{a,b\}}, \ldots, \delta_{k \in \{a,b\}}, \delta_{1 \in \{a,b\}} 1_{\underline c}, \ldots, \delta_{k \in \{a,b\}}k_{\underline c})
\prod_{i=1}^k
\frac{1}{2-\delta_{i \in \{a,b\}}}\\
&=
2^{k-2}(a_{\underline c} + b_{\underline c}- a_{\underline c}b_{\underline c})\cdot \frac{1}{2^{k-2}}
=
a_{\underline c} + b_{\underline c}- a_{\underline c}b_{\underline c}.
\end{align*}
As \(a_{\underline c}, b_{\underline c} \in \{0,1\}\), we have that \(f_{\lambda,\underline c}= g_{\lambda,\underline c}\) in any case. This completes the proof of the base case.

Now for the induction step, fix \(\lambda, \underline{c}\) such that \(|\lambda| > \lambda_i\) for \(i \in [1,k]\), and assume that \(|\lambda|>2\). We may moreover assume that \(\lambda, \underline c\) satisfy (\ref{nkreq}), else \(f_{\lambda,\underline c}= 0 = g_{\lambda,\underline c}\) by Lemma~\ref{nkreqLem}. We make the induction assumption that the theorem statement holds for all \(\lambda', \underline{c}'\) with \(|\lambda'| < |\lambda|\). Permuting colors if necessary, we may assume that \(k_{\underline c} >0\), and \(|\lambda| -1 > \lambda_j\) for all \(j \in [1,k-1]\), since \(f_{\lambda,\underline c}\) and \(g_{\lambda,\underline c}\) are invariant under color permutation thanks to Lemma~\ref{ThePolyLem}(iii).

Thus we have:
\begin{align*}
f_{\lambda,\underline c} &= \sum_{S \subseteq [1,k-1]}
f_{\lambda- \varepsilon_k,\underline{c}(S)}\\
&=
\sum_{S \subseteq [1,k-1]}
P_k(\lambda_1, \ldots, \lambda_{k-1}, \lambda_k -1, 1_{\underline c}+\delta_{1 \in S}, \ldots, (k-1)_{\underline c} + \delta_{k-1 \in S}, k_{\underline c}-1)\\
& \hspace{60mm} \times
\prod_{i=1}^{k}  
\frac{1}{|\lambda| - \lambda_i-\delta_{i \in [1,k-1]}}
{|\lambda| - \lambda_i -\delta_{i \in [1,k-1]}\choose \lambda_i - i_{\underline{c}} - \delta_{i \in S}}\\
&=\sum_{S \subseteq [1,k-1]}
P^{(S)}_k(\lambda_1, \ldots, \lambda_k, 1_{\underline c}, \ldots, k_{\underline c})
\prod_{i=1}^k 
\frac{1}{|\lambda| - \lambda_i}
{|\lambda| - \lambda_i \choose \lambda_i - i_{\underline{c}}}\\
& \hspace{60mm} \times
\prod_{i \in S}
\frac{\lambda_i - i_{\underline{c}}}{|\lambda|- \lambda_i -1}
\prod_{i \in [1,k-1] \backslash S}
\frac{|\lambda|- 2 \lambda_i + i_{\underline c}}{|\lambda|-\lambda_i -1}\\
&=P_k(\lambda_1, \ldots, \lambda_k, 1_{\underline c}, \ldots, k_{\underline c})
\prod_{i=1}^k \frac{1}{|\lambda| - \lambda_i}{|\lambda| - \lambda_i \choose \lambda_i - i_{\underline c}},
\end{align*}
where we have used Lemma~\ref{recurf} for the first equality, the induction assumption for the second equality, rewriting for the third equality, and Lemma~\ref{ThePolyLem}(iv) for the last equality. This completes the induction step, and the proof.
\end{proof}

\subsection{Enumerating \((\lambda,c)\)-trees}
Now we briefly restrict our attention to the case of \((\lambda, c)\)-trees, wherein the formula (\ref{forclosed}) simplifies considerably:

\begin{Corollary}\label{TreeCor}
Let \(k \in \mathbb{N}\), \(\lambda \in \Lambda_k\), and \(c \in [1,k]\). If \(\lambda_c =0\) we have \(t_{\lambda, c} = 0\). Otherwise:
\begin{align}\label{treeformula}
t_{\lambda,c} & = 
|\lambda|^{k-2}{ |\lambda|- \lambda_c \choose \lambda_c -1}
\prod_{i  \neq c} \frac{1}{|\lambda|-\lambda_i} {|\lambda|- \lambda_i \choose \lambda_i}.
\end{align}
\end{Corollary}

\begin{proof}
Without loss of generality, we may assume \(c=k\). From Theorem~\ref{MainForThm}, we have:
\begin{align*}
t_{\lambda,k} = f_{\lambda, (k)} =
P_k(\lambda_1, \ldots, \lambda_k, 0, \ldots, 0, 1)
\cdot \frac{1}{|\lambda| - \lambda_k}
{ |\lambda| - \lambda_k \choose \lambda_k -1}
\prod_{i  \neq k} \frac{1}{|\lambda|-\lambda_i} {|\lambda|- \lambda_i \choose \lambda_i}.
\end{align*}
Inspection of (\ref{ThePoly2}) shows that
\begin{align*}
P_k(\lambda_1, \ldots, \lambda_k, 0, \ldots, 0, 1)
&=
(|\lambda| - \lambda_k)|\lambda|^{k-2},
\end{align*}
so the result follows.
\end{proof}

\section{Fuss-Catalan numbers and \((\lambda, \underline c)\)-forests}\label{BigCatSec}
In this section we explain that the set of all injectively \(k\)-colored forests with a fixed number of vertices and fixed root color sequence are counted by Fuss-Catalan numbers. Thus grouping forests by character as in \S\ref{enumfor} yields a multi-parameter distribution of these numbers.

\subsection{A Catalan-like recurrence for \((\lambda,c)\)-trees}
For \(\lambda \in \Lambda_k\) and \(c \in [1,k]\), it will be useful to define an extension set \(\widehat{\mathcal{T}}_{\lambda, c}\) by setting \(\widehat{\mathcal{T}}_{\lambda, c} = \mathcal{T}_{\lambda, c}\) for \(|\lambda| >0\) and mandating that \(\widehat{\mathcal{T}}_{\mathbf{0},c}\) contain only the empty tree. We set \(\widehat{\mathcal{T}}_{n,c} := \bigsqcup_{\lambda \in \Lambda_k} \widehat{\mathcal{T}}_{\lambda,c}\), \(\hat{t}_{\lambda, c} := |\widehat{\mathcal{T}}_{\lambda, c}|\) and \(\hat{t}_{n, c} := |\widehat{\mathcal{T}}_{n, c}|\).
Compare the following result with (\ref{CatRecur}):
\begin{Lemma}\label{trecur}
For \(\lambda \in \Lambda_k\), we have
\begin{align*}
\hat{t}_{\lambda, c} = \delta_{\lambda,\mathbf{0}}
\;\;\;\;
(\lambda_c = 0);
\qquad
\qquad
\hat
t_{\lambda,c}=
\sum
\hat
t_{\mu^{(1)},c}
\cdots 
\hat
t_{\mu^{(c-1)},c}
\hat
t_{\mu^{(c+1)},c}
\cdots
\hat
t_{\mu^{(k)},c}
\;\;\;\;
( \lambda_c > 0),
\end{align*}
where the sum ranges over all \(\mu^{(1)}, \ldots, \mu^{(c-1)}, \mu^{(c+1)}, \ldots, \mu^{(k)} \in \Lambda_k\) such that
\begin{align*}
\sum_{i\neq c} (i,c) \cdot \mu^{(i)} = \lambda - \varepsilon_c.
\end{align*}
\end{Lemma}
\begin{proof}
The initial values follow immediately from Corollary~\ref{TreeCor}. Assume \(\lambda \in \Lambda_k\) is such that \(\lambda_c>0\), and let \(T \in \widehat{\mathcal{T}}_{\lambda,c} =\mathcal{T}_{\lambda,c} \). Let \(i \in [1,k]\backslash \{c\}\). Recalling that \(\mathsf{r}_T(1)\) is the root of \(T\), if \((\mathsf{r}_T(1),x_i) \in A_T\) for some \(x_i \in (V_T)_i\), then set \(T^{(i)}\) to be the subtree of \(T\) with root \(\mathsf{r}_{T^{(i)}}(1) = x_i\), and let \(\ell_{T^{(i)}}\) be the injective \(k\)-coloring induced by restricting \(\ell_T\) to \(T^{(i)}\). Otherwise, set \(T^{(i)}\) to be the empty tree. In any case, we have that \(T^{(i)} \in \widehat{\mathcal{T}}_{\nu^{(i)},i}\), where \(\nu^{(i)} = \textup{char}(\ell_{T^{(i)}})\),
\begin{align*}
V_T = \{\mathsf{r}_{T}(1)\} \sqcup \bigsqcup_{i \neq c} V_{T^{(i)}};
\qquad
\textup{and}
\qquad
\lambda - \varepsilon_c = \sum_{i \neq c} \nu^{(i)}.
\end{align*}
This assignment defines a function 
\begin{align*}
\xi: \widehat{\mathcal{T}}^\lambda_c & \to
\bigsqcup \widehat{\mathcal{T}}_{\nu^{(1)},1} \times \cdots \times \widehat{\mathcal{T}}_{\nu^{(c-1)},c-1} \times\widehat{\mathcal{T}}_{\nu^{(c+1)},c+1} \times \cdots \times \widehat{\mathcal{T}}_{\nu^{(k)},k};
\\
T & \mapsto (T^{(1)}, \ldots, T^{(c-1)}, T^{(c+1)}, \ldots, T^{(k)}),
\end{align*}
where the union is over all \(\nu^{(1)}, \ldots, \nu^{(c-1)}, \nu^{(c+1)}, \ldots, \nu^{(k)} \in \Lambda_k\) such that \(\sum_{i \neq c} \nu^{(i)} = \lambda- \varepsilon_c\).

In the other direction, assume 
\(
(T^{(1)}, \ldots, T^{(c-1)}, T^{(c+1)}, \ldots, T^{(k)})
\)
is in the codomain above. We construct an associated tree \(T\) by adding a vertex \(v\) to the directed graph \(T^{(1)}\sqcup \ldots\sqcup T^{(c-1)}\sqcup T^{(c+1)}\sqcup \ldots \sqcup T^{(k)}\), along with arrows from \(v\) to \(\mathsf{r}_{T^{(i)}}(1)\) whenever \(T^{(i)}\) is nonempty. Define a \(k\)-coloring \(\ell_T\) on \(T\) by setting \(\ell_T(v) = c\) and \(\ell_T(w) = \ell_{T^{(i)}}(w)\) for all \(w \in V_{T^{(i)}}\), and set \(\mathsf{r}_T(1)= v\). It follows then that \(T \in \widehat{\mathcal{T}}_{\lambda,c}\). The assignment \((T^{(1)}, \ldots, T^{(c-1)}, T^{(c+1)}, \ldots, T^{(k)}) \mapsto T\) then defines a function which is inverse to \(\xi\) by construction.

Thus we have
\begin{align*}
\hat t_{\lambda,c} = \sum \hat t_{\nu^{(1)},1} \cdots \hat t_{\nu^{(c-1)},c-1} \hat t_{\nu^{(c+1)},c+1} \cdots \hat t_{\nu^{(k)},k},
\end{align*}
where the sum is over all \(\nu^{(1)}, \ldots, \nu^{(c-1)}, \nu^{(c+1)}, \ldots, \nu^{(k)} \in \Lambda_k\) such that \(\sum_{i \neq c} \nu^{(i)} = \lambda- \varepsilon_c\). Noting that \(\hat t_{\nu^{(i)},i} = \hat t_{(i, c)\cdot \nu^{(i)},c}\) for all \(i \neq c\) gives the result.
\end{proof}

Let \(\widehat{\tau}_c(u_1, \ldots, u_k)\) be the multivariate generating function for \(\{\hat t_{\lambda,c}\}_{\lambda \in \Lambda_k}\):
\begin{align}\label{genfun}
\widehat \tau_c(u_1, \ldots, u_k) := \sum_{\lambda \in \Lambda_k} \hat t_{\lambda,c} u_1^{\lambda_1}\cdots u_k^{\lambda_k}.
\end{align}
Lemma~\ref{trecur} implies the following:
\begin{Corollary}\label{polyrecur}
The generating function \(\widehat \tau_c(u_1, \ldots, u_k)\) satisfies the relation:
\begin{align*}
\widehat \tau_c(u_1, \ldots, u_k) = 1+ u_c \prod_{i \neq c}\widehat \tau_c(u_{(i,c)\cdot 1}, \ldots, u_{(i,c)\cdot k}).
\end{align*}
\end{Corollary}

\subsection{Fuss-Catalan numbers and \((\lambda,c)\)-trees}
Recall the definition of the Fuss-Catalan numbers \(A_n(p,r)\) from \S\ref{FussCatSec}.

\begin{Lemma}\label{areCat}
For \(c \in [1,k]\), \(n \in \mathbb{N}\), we have:
\begin{align*}
t_{n,c} = \frac{1}{kn-n +1}{kn - n + 1 \choose n} = A_n(k-1,1).
\end{align*}
\end{Lemma}
\begin{proof}
Writing \(q(x) = \widehat \tau_c(x,\ldots,x)\), we have that
\begin{align*}
q(x) = \sum_{\lambda \in \Lambda_k} \hat t_{\lambda,c} x^{|\lambda|} = \sum_{n \geq 0} \bigg( \sum_{\lambda \in \Lambda_k^n} \hat t_{\lambda,c} \bigg) x^n = \sum_{n \geq 0} \hat t_{n,c} x^n,
\end{align*}
so \(q(x)\) is the generating function for the sequence \(\{\hat t_{n,c}\}_{n \geq 0}\).
By Corollary~\ref{polyrecur}, \(q(x)\) satisfies
\begin{align*}
q(x) = 1+xq(x)^{k-1}.
\end{align*}
Hence, recalling \S\ref{FussCatSec}, we have that \(q(x)\) is the generating function for \(\{A_n(k-1,1) \}_{n \geq 0}\), and so \(t_{n,c} = \hat t_{n,c} = A_n(k-1,1)\) for all \(n \in \mathbb{N}\), as desired.
\end{proof}

\subsection{A distribution of \(A_n(p,1)\)}\label{Ap1distSec}
We utilize the results of the previous section to describe a \(p\)-parameter distribution \(\{\xi_n(\nu)\}_{\nu \in \Lambda_p(<n)}\) of the Fuss-Catalan numbers \(A_n(p,1)\). For \(n,p \in \mathbb{N}\), \(\nu \in \Lambda_p(<n)\), define
\begin{align*}
\xi_n(\nu) = 
n^{p-1}
{  |\nu|\choose n - |\nu| -1}
\frac{{n-\nu_1 \choose \nu_1} \cdots {n - \nu_p \choose \nu_p}}{(n-\nu_1) \cdots (n- \nu_p)}.
\end{align*}
\begin{Proposition}\label{Ap1dist}
For all \(n,p \in \mathbb{N}\), we have
\begin{align*}
A_n(p,1) = \sum_{\nu \in \Lambda_p(<n)} \xi_n(\nu).
\end{align*}
\end{Proposition}
\begin{proof}
By Lemma~\ref{areCat}, we have
\(
A_n(p,1) = \sum_{\lambda \in \Lambda_{p+1}(n)} t_{\lambda, p+1}
\). Taking \(\nu = (\lambda_1, \ldots, \lambda_p)\), we have that \(\lambda_{p+1} = n - |\nu|\), and thus \(t_{\lambda, p+1} = \xi_n(\nu)\) if \(|\nu| < n\), and zero otherwise by Corollary~\ref{TreeCor}. The result immediately follows.
\end{proof}

\begin{Remark}
See Figure~\ref{fig:CatEx} for some explicit \(\xi_n(\nu)\) values in the classical Catalan number case \(p=2\). Taken together, the southeast halves of each table may be viewed as layers of tetrahedral distribution of the Catalan sequence.
\end{Remark}

\begin{Remark}\label{RemConnex}
In the case \(p = 2\), the numbers \(\xi_n(\nu)\) refine/generalize a number of known integer sequences. We consider some examples below, with links to their instances in the OEIS \cite{OEIS}.
\begin{enumerate}
\item As noted earlier, \(\sum_{\nu \in \Lambda_2(<n)} \xi_n(\nu)\) (sum over all entries in the \(n\)th table in Figure~\ref{fig:CatEx}) is the \(n\)th Catalan number. \href{http://oeis.org/A000108}{[A000108]}
\item For \(h \in \ZZ_{\geq 0}\), the value \(\sum_{t = 0}^{n-1}\xi_n((t, h))\) (sum of entries in the \(h\)th row or column in the \(n\)th table in Figure~\ref{fig:CatEx}) is the number of {\em ordered} trees on \(n\) edges containing \(h\) nodes adjacent to a leaf \cite{Callan}.  \href{http://oeis.org/A108759}{[A108759]}
\item For odd \(n\) and \(h \in \ZZ_{\geq 0}\), the value \(\xi_n((h, (n-1)/2 - h))\) (\(h\)th entry on the main antidiagonal in the \(n\)th table in Figure~\ref{fig:CatEx}) is the number of `fighting fish' with \((n+1)/2 - h\) left lower free and \(h+1\) right lower free edges with a marked tail \cite{fish}. \href{http://oeis.org/A278880}{[A278880]}
\item For even \(n\), reading along the main antidiagonals and bottom rows in Figure~\ref{fig:CatEx} yields the sequences \href{http://oeis.org/A323324}{[A323324]} and \href{http://oeis.org/A278881}{[A278881]} respectively. 
\end{enumerate}
\end{Remark}

\begin{figure}[h]
\small
\begin{tabular}{ccccccc}
\(n=1\)& 
\(n=2\) &
\(n=3\)& \(n=4\) & \(n=5\) \\
\begin{tabular}{cc|}
\hline
\multicolumn{1}{|c|}{\diagbox[width=0.75cm, height=0.5cm, innerrightsep=1pt, innerleftsep=1pt]{\(\scriptstyle\nu_1\)}{\({}^{\nu_2}\)}}& 0 \\
\hline
\multicolumn{1}{|c|}{0} & 1 \\
\hline
\\
\\
\end{tabular}
&
\begin{tabular}{ccc|}
\hline
\multicolumn{1}{|c|}{\diagbox[width=0.75cm, height=0.5cm, innerrightsep=1pt, innerleftsep=1pt]{\(\scriptstyle\nu_1\)}{\({}^{\nu_2}\)}} & 0 & 1 \\
\hline
\multicolumn{1}{|c|}{0} & . & 1 \\
\multicolumn{1}{|c|}{1} & 1 & .   \\
\hline
\\
\end{tabular}
&
\begin{tabular}{ccc|}
\hline
\multicolumn{1}{|c|}{
\diagbox[width=0.75cm, height=0.5cm, innerrightsep=1pt, innerleftsep=1pt]{\(\scriptstyle\nu_1\)}{\({}^{\nu_2}\)}} & 0 & 1 \\
\hline
\multicolumn{1}{|c|}{0 }& . & 1 \\
\multicolumn{1}{|c|}{1} & 1 & 3   \\
\hline
\\
\end{tabular}
&
\begin{tabular}{|c|ccc|}
\hline
\diagbox[width=0.75cm, height=0.5cm, innerrightsep=1pt, innerleftsep=1pt]{\(\scriptstyle\nu_1\)}{\({}^{\nu_2}\)} & 0 & 1&2 \\
\hline
0 & . & . & 1 \\
1 & . & 8 & 2  \\
2& 1 & 2 &.\\
\hline
\end{tabular}
&
\begin{tabular}{|c|ccc|}
\hline
\diagbox[width=0.75cm, height=0.5cm, innerrightsep=1pt, innerleftsep=1pt]{\(\scriptstyle\nu_1\)}{\({}^{\nu_2}\)} & 0 & 1&2 \\
\hline
0 & . & . & 1 \\
1 & . & 5 & 15  \\
2& 1 & 15 &5\\
\hline
\end{tabular}
\vspace{3mm}
\end{tabular}
\begin{tabular}{ccccccc}
\(n=6\)& 
\(n=7\) &
\(n=8\) \\
\begin{tabular}{ccccc|}
\hline
\multicolumn{1}{|c|}{\diagbox[width=0.75cm, height=0.5cm, innerrightsep=1pt, innerleftsep=1pt]{\(\scriptstyle\nu_1\)}{\({}^{\nu_2}\)}} & 0 & 1&2&3 \\
\hline
\multicolumn{1}{|c|}{0} & . & . & . & 1\\
\multicolumn{1}{|c|}{1} & . & . & 27 & 8 \\
\multicolumn{1}{|c|}{2}& . & 27 &54 &3\\
\multicolumn{1}{|c|}{3}& 1 & 8 & 3 & . \\
\hline
\\
\end{tabular}
&
\begin{tabular}{ccccc|}
\hline
\multicolumn{1}{|c|}{\diagbox[width=0.75cm, height=0.5cm, innerrightsep=1pt, innerleftsep=1pt]{\(\scriptstyle\nu_1\)}{\({}^{\nu_2}\)} }& 0 & 1&2&3 \\
\hline
\multicolumn{1}{|c|}{0} & . & . & . & 1\\
\multicolumn{1}{|c|}{1} & . & . & 14 & 42 \\
\multicolumn{1}{|c|}{2}& . & 14 &168 &70\\
\multicolumn{1}{|c|}{3}& 1 & 42 & 70 & 7 \\
\hline
\\
\end{tabular}
&
\begin{tabular}{|c|ccccc|}
\hline
\diagbox[width=0.75cm, height=0.5cm, innerrightsep=1pt, innerleftsep=1pt]{\(\scriptstyle\nu_1\)}{\({}^{\nu_2}\)} & 0 & 1&2&3& 4 \\
\hline
0 & . & . & . & .&1\\
1 & . & . & . & 64 &20\\
2& . & . &200 &400&30\\
3& . & 64 & 400 &192 &4\\
4& 1 & 20 & 30 & 4 &.\\
\hline
\end{tabular}
\end{tabular}
\caption{The values \(\xi_n(\nu) = t_{( \nu_1, \nu_2, n- |\nu|), 3}\) for \(n \in [1,8]\) and \(\nu \in \Lambda_2(<n)\). The sum of entries in each table is the Catalan number \(C_n = A_n(2,1)\). 
}
\label{fig:CatEx}       % Give a unique label
\end{figure}

\subsection{Enumerating injectively \(k\)-colored rooted forests}
Now we enumerate the set of injectively \(k\)-colored rooted forests with \(n\) vertices and a given root color sequence \(\underline c\) of length \(m\). While, for \(\lambda \in \Lambda_k(n)\), the number \(f_{\lambda, \underline c}\) depends on \(\underline c\), the total number \(f_{n, \underline c}\) does not:

\begin{Theorem}\label{AllFor}
For all \(n,k,m \in \mathbb{N}\), \(\underline{c} \in [1,k]^m\), we have
\begin{align*}
f_{n,\underline c}
=
\frac{m}{n}{kn-n \choose n-m}
=
A_{n-m}(k-1,km-m).
\end{align*}
\end{Theorem}
\begin{proof}
We have a natural bijection given by decomposing \(F \in \mathcal{F}_{n,\underline c}\) into its tree components:
\begin{align*}
\mathcal{F}_{n,\underline c} \xrightarrow{\sim} \bigsqcup_{\nu \in \Lambda_m(n)} \mathcal{T}_{\nu_1,c_1}\times \cdots \times \mathcal{T}_{\nu_m,c_m}. 
\end{align*}
Note that \(t_{n', c} = \hat t_{n',c} - \delta_{n',0}\) for all \(n' \in \mathbb{Z}_{\geq 0}\), so by Lemma~\ref{trecur} we have
\begin{align*}
f_{n,\underline c} &= \sum_{\nu \in \Lambda_m(n)}t_{\nu_1,c_1}\cdots t_{\nu_m,c_m}
=\sum_{\nu \in \Lambda_m(n)}(A_{\nu_1}(k-1,1) - \delta_{\nu_1, 0}) \cdots (A_{\nu_m}(k-1,1)- \delta_{\nu_m, 0}).
\end{align*}
Writing \(\rho_{\underline c}(x) = \sum_{n \geq 0} t_{n, \underline c}x^n\) for the generating function of \((f_{n, \underline c})_{n\geq 0}\), we thus have 
\begin{align*}
\rho_{\underline c}(x) &= (B_{k-1,1}(x)-1)^m = \sum_{w=0}^m (-1)^{m-w}{m \choose w} B_{k-1,1}(x)^w
\\&=\sum_{w = 0}^m (-1)^{m-w} {m \choose w} B_{k-1,w}(x)
=\sum_{u \geq 0}\sum_{w=0}^m(-1)^{m-w}{m \choose w}A_u(k-1,w)x^u.
\end{align*}
Thus
\begin{align*}
f_{n, \underline c} &= 
\sum_{w = 0}^m (-1)^{m-w} {m \choose w} A_n(k-1,w)
=
\sum_{w=0}^m \frac{(-1)^{m-w}w}{w+uk -u}{m \choose w} {w + uk - u \choose u}\\
&=(-1)^m\left[
\sum (-1)^w {m \choose w}{w + uk-u \choose u}
-
(uk-u)
\sum
\frac{(-1)^w}{w+uk-u}
{m \choose w}{w +uk-u \choose u}\right]\\
&=(-1)^m\left[
\sum (-1)^w {m \choose w}{w + uk-u \choose u}
-
(k-1)
\sum
(-1)^w{m \choose w}{w +uk-u-1 \choose u-1}\right]\\
&={uk-u \choose u-m}
-
(k-1){uk-u-1 \choose u-1-m}=
\frac{m}{n} {kn -n \choose n-m},
\end{align*}
where the fifth equality above follows by dual applications of the standard binomial identity:
\begin{align*}
\sum_{r=0}^t (-1)^r {t \choose r}{ s+r \choose v} = (-1)^t {s \choose v - t},
\end{align*}
see for instance \cite[Vol. 4, (10.15)]{Gould}.
\end{proof}

\subsection{Distributions of Fuss-Catalan numbers}\label{FCdistSec}
For a partition \(\rho \in \Lambda_{p+1}^+(\ell)\) and \(\mu \in \Lambda_{p}(\leq n)\) we set \(\alpha_\mu(\rho) := \delta_{n,0}\) if \(\rho_1 = \ell\), \( \mu_1=n\). Otherwise we set:
\begin{align*}
\alpha_n(\rho, \mu) &= P_{p+1}(\rho_1 + \mu_1, \ldots, \rho_{p} + \mu_{p}, \rho_{p+1} + n - |\mu|, \rho_1, \ldots, \rho_{p+1})\\
& \hspace{10mm}\times
\frac{1}{\ell +|\mu| - \rho_{p+1}}{\ell+ |\mu| - \rho_{p+1} \choose n - |\mu|}
\prod_{i = 1}^p
\frac{1}{\ell + n - \rho_i - \mu_i }{\ell + n - \rho_i - \mu_i \choose \mu_i}.
\end{align*}

As shown in the next proposition, and exemplified in Figure~\ref{fig:6types}, for each choice of partition \(\rho \in \Lambda^+_{p+1}(\ell)\), the numbers \(\{\alpha_n(\rho,\mu) \mid \mu \in \Lambda_p(\leq n)\}\) define an associated \(p\)-parameter distribution of the Fuss-Catalan number \(A_n(p, p \ell)\).
\begin{Proposition}\label{refineBigFuss}
Assume \(n,p,\ell \in \mathbb{Z}_{\geq 0}\), \(\rho \in \Lambda_{p+1}^+(\ell)\). Then we have:
\begin{align*}
A_n(p,p\ell) = \sum_{
\mu \in \Lambda_{p}(\leq n)
}
\alpha_\mu(\rho, \mu).
\end{align*}
\end{Proposition}
\begin{proof}
Set \(\underline{c} = 1^{\rho_1} \cdots (p+1)^{\rho_{p+1}} \in [1,p+1]^{\ell}\).
Then, by Theorem~\ref{AllFor} we have
\begin{align*}
A_{n}(p,p\ell) &=
t_{n+\ell, \underline c}
= 
\sum_{\lambda \in \Lambda_{p+1}(n+\ell)} t_{\lambda, \underline c}
= \sum_{\eta \in \Lambda_{p+1}(n)} t_{\eta + \rho, \underline c} 
= \sum_{\mu \in \Lambda_p(\leq n)} t_{\hat \mu + \rho, \underline c},
\end{align*}
where \(\Lambda_{p+1}(n) \ni \hat \mu = (\mu_1, \ldots, \mu_p, n - |\mu|)\) for \(\mu \in \Lambda_p( \leq n)\). The result follows from Theorem~\ref{MainForThm}.
\end{proof}

\begin{figure}[h]
\small
\begin{tabular}{ccc}
\(\rho = (3,0,0)\)& 
\(\rho = (2,1,0)\) &
\(\rho = (1,1,1)\)\\
\begin{tabular}{|c|ccccc|}
\hline
\diagbox[width=0.75cm, height=0.5cm, innerrightsep=1pt, innerleftsep=1pt]{\(\scriptstyle\mu_1\)}{\({}^{\mu_2}\)} & 0 & 1&2&3 & 4\\
\hline
0 & . & . & 81 & 300 & 81\\
1 & . & 135 & 1350 & 1350 & 135\\
2& 18 & 540 &1458 &540 & 18\\
3& 10 & 81 & 81 & 10 & .\\
4& . & . & . & . & .\\
\hline
\end{tabular}
&
\begin{tabular}{|c|ccccc|}
\hline
\diagbox[width=0.75cm, height=0.5cm, innerrightsep=1pt, innerleftsep=1pt]{\(\scriptstyle\mu_1\)}{\({}^{\mu_2}\)} & 0 & 1&2&3 & 4\\
\hline
0 & . & . & 55 & 140 & 15\\
1 & . & 145 & 1150 & 786 & 31\\
2& 34 & 860 &1830 &460 & 6\\
3& 50 & 339 & 265 & 22 & .\\
4& . & . & . & . & .\\
\hline
\end{tabular}
&
\begin{tabular}{|c|ccccc|}
\hline
\diagbox[width=0.75cm, height=0.5cm, innerrightsep=1pt, innerleftsep=1pt]{\(\scriptstyle\mu_1\)}{\({}^{\mu_2}\)} & 0 & 1&2&3 & 4\\
\hline
0 & . & . & 10 & 64 & 10\\
1 & . & 32 & 640 & 640 & 32\\
2& 10 & 640 &2000 &640 & 10\\
3& 64 & 640 & 640 & 64 & .\\
4 & 10 & 32 & 10 & . & . \\
\hline
\end{tabular}
\end{tabular}
\caption{The values \(\alpha_6(\rho, \mu)\), for \(\rho \in \Lambda^+_3(3)\), \(\mu \in \Lambda_2(\leq 6)\). The sum of entries in each table is \(A_6(2,6)=6188\).}
\label{fig:6types}       % Give a unique label
\end{figure}

\section{3-colorings of triangulations}\label{TriSec}

In this section we consider a correspondence between injectively 3-colored rooted trees and triangulations of convex polygons. 

\begin{Definition} A {\em triangulation} \(\tau\) of an \(n\)-gon \(P\) is a decomposition of \(P\) into \(n-2\) triangles with non-intersecting interiors.
\end{Definition}

Every triangulation corresponds to a choice of \(n-3\) non-crossing interior line segments, or {\em diagonals} between vertices of \(P\). It is well-known that the number of distinct triangulations of a  convex \(n\)-gon is given by the Catalan number \(C_{n-2} = A_{n-2}(2,1)\). We will write \(\textup{Tri}_n\) for the set of triangulations of the regular convex \(n\)-gon. We treat \(\tau \in \textup{Tri}_n\) as a graph whose vertices are the vertices of \(P\) and whose edges are the sides of \(P\) and diagonals in \(\tau\).
It is also well-known (see for instance \cite[\S3.3]{Dimacs}) that every triangulation \(\tau \in\textup{Tri}_n \) possesses a proper 3-coloring which is unique up to permutation of colors. 
This gives rise to the following definition:

\begin{Definition}
For \(\tau \in \textup{Tri}_n\) and partition \(\lambda \in \Lambda_3^+(n)\) we say \(\tau\) has {\em type} \(\lambda\) provided that \(\tau\) has a proper 3-coloring wherein \(\lambda_1\) vertices are colored \(1\), \(\lambda_2\) vertices are colored \(2\), and \(\lambda_3\) vertices are colored \(3\). We write \(\textup{type}(\tau) = \lambda\), noting that such \(\lambda\) is unique in \(\Lambda_3^+(n)\).
\end{Definition}

For \(\lambda \in \Lambda_3^+(n)\), we write \(\textup{Tri}_n^\lambda := \{ \tau \in \textup{Tri}_n \mid \textup{type}(\tau) = \lambda\} \) for the set of type-\(\lambda\) triangulations of the \(n\)-gon. See Figure~\ref{fig:hextri} for an example. Our goal in the next subsection is to describe the cardinality of the set \(\textup{Tri}^\lambda_n\).

\begin{Remark}
The 3-colorings of triangulation graphs have played an important role in a number of combinatorial contexts. One such is in the proof of Chv\'atal's Theorem \cite{chvatal} which states that every simple polygon with \(n\) sides may be `guarded' by at most \(\lfloor n/3 \rfloor\) guards placed around the polygon. In \cite{fisk}, Fisk gave an elegant proof of this fact by triangulating the polygon, 3-coloring the triangulation, then placing guards at the vertices of the color which appears least often. Thus Fisk's algorithm will produce a valid guarding with exactly \(\lambda_3 \leq \lfloor n/3 \rfloor\) guards when the the triangulation is of type \(\lambda\).
\end{Remark}

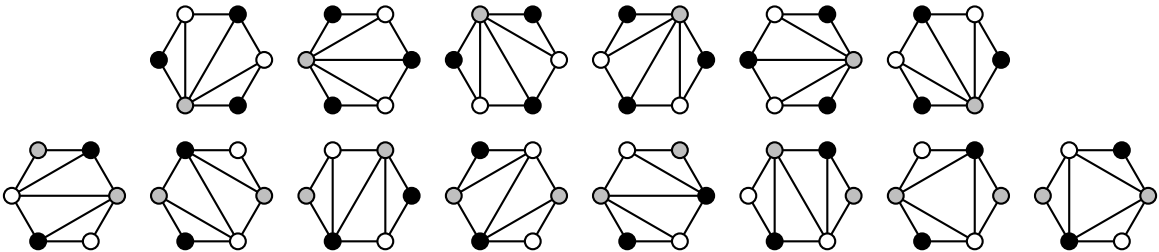
\begin{figure}[h]
\begin{align*}
{}
\hackcenter{
     \begin{tikzpicture}[scale = 0.7]
    \draw[thick, ]  (-1,0)--(-.5,.866)--(.5,.866)--(1,0)--(.5,-.866)--(-.5,-.866)--(-1,0);
    \draw[thick] (-.5,-.866)--(-.5,.866);
      \draw[thick] (-.5,-.866)--(.5,.866);
        \draw[thick] (-.5,-.866)--(1,0);
         \draw[thick, black, fill= black] (.5,-.866) circle (1.5mm);
          \draw[thick, black, fill= lightgray] (-.5,-.866) circle (1.5mm);
           \draw[thick, black, fill= black] (.5,.866) circle (1.5mm);
            \draw[thick, black, fill= black] (-1,0) circle (1.5mm);
             \draw[thick, black, fill= white] (-.5,.866) circle (1.5mm);
              \draw[thick, black, fill= white] (1,0) circle (1.5mm);
    \end{tikzpicture}
    \;\;\;
 \begin{tikzpicture}[scale = 0.7]
    \draw[thick, ]  (-1,0)--(-.5,.866)--(.5,.866)--(1,0)--(.5,-.866)--(-.5,-.866)--(-1,0)--(1/2,.866); \draw[thick ](-1,0)--(1,0); \draw[thick] (-1,0)--(.5,-.866);
    \draw[thick, black, fill=lightgray] (-1,0) circle (1.5mm);
    \draw[thick, black, fill= black] (1,0)  circle (1.5mm);
    \draw[thick, black, fill= white] (.5,.866)  circle (1.5mm);
    \draw[thick, black, fill= white] (.5,-.866)  circle (1.5mm);
     \draw[thick, black, fill= black] (-.5,.866)  circle (1.5mm);
      \draw[thick, black, fill= black] (-.5,-.866)  circle (1.5mm);
    \end{tikzpicture}
       \;\;\;
     \begin{tikzpicture}[scale = 0.7]
    \draw[thick, ]  (-1,0)--(-.5,.866)--(.5,.866)--(1,0)--(.5,-.866)--(-.5,-.866)--(-1,0);
    \draw[thick] (-.5,.866)--(-.5,-.866);
    \draw[thick] (-.5,.866)--(.5,-.866);
    \draw[thick] (-.5,.866)--(1,0);
 \draw[thick, black, fill= black] (.5,-.866) circle (1.5mm);
  \draw[thick, black, fill= lightgray] (-.5,.866) circle (1.5mm);
   \draw[thick, black, fill= black] (.5,.866) circle (1.5mm);
    \draw[thick, black, fill= black] (-1,0) circle (1.5mm);
     \draw[thick, black, fill= white] (1,0) circle (1.5mm);
      \draw[thick, black, fill=  white] (-.5,-.866) circle (1.5mm);
    \end{tikzpicture}
\;\;\;
    \begin{tikzpicture}[scale = 0.7]
    \draw[thick, ]  (-1,0)--(-.5,.866)--(.5,.866)--(1,0)--(.5,-.866)--(-.5,-.866)--(-1,0); \draw[thick ] (.5,.866)--(-1,0); \draw[thick ] (.5,.866)--(-.5,-.866);
    \draw[thick ] (.5,.866)--(.5,-.866);
     \draw[thick, black, fill= lightgray] (.5,.866)  circle (1.5mm);
      \draw[thick, black, fill= black] (1,0)  circle (1.5mm);
       \draw[thick, black, fill= black] (-.5,-.866) 
       circle (1.5mm);
        \draw[thick, black, fill= black] (-.5,.866)  circle (1.5mm);
         \draw[thick, black, fill= white] (-1,0)  circle (1.5mm);
          \draw[thick, black, fill= white] (.5,-.866)  circle (1.5mm);
    \end{tikzpicture}
    \;\;\;
     \begin{tikzpicture}[scale = 0.7]
    \draw[thick, ]  (-1,0)--(-.5,.866)--(.5,.866)--(1,0)--(.5,-.866)--(-.5,-.866)--(-1,0);
    \draw[thick](1,0)--(-.5,.866);
    \draw[thick](1,0)--(-1,0);
    \draw[thick](1,0)--(-.5,-.866);
    \draw[thick, black, fill= lightgray] (1,0) circle (1.5mm);
    \draw[thick, black, fill= black] (.5,-.866) circle (1.5mm);
    \draw[thick, black, fill= black] (-1,0) circle (1.5mm);
    \draw[thick, black, fill= black] (.5,.866) circle (1.5mm);
    \draw[thick, black, fill= white] (-.5,.866) circle (1.5mm);
    \draw[thick, black, fill= white] (-.5,-.866) circle (1.5mm);
    \end{tikzpicture}
    \;\;\;
     \begin{tikzpicture}[scale = 0.7]
    \draw[thick, ]  (-1,0)--(-.5,.866)--(.5,.866)--(1,0)--(.5,-.866)--(-.5,-.866)--(-1,0);
    \draw[thick] (.5,-.866)--(-1,0);
    \draw[thick] (.5,-.866)--(-.5,.866);
    \draw[thick] (.5,-.866)--(.5,.866);
    \draw[thick, black, fill= lightgray] (.5,-.866) circle (1.5mm);
    \draw[thick, black, fill= black] (-.5,-.866) circle (1.5mm);
    \draw[thick, black, fill= black] (-.5,.866) circle (1.5mm);
    \draw[thick, black, fill= black] (1,0) circle (1.5mm);
    \draw[thick, black, fill= white] (-1,0) circle (1.5mm);
    \draw[thick, black, fill= white] (.5,.866) circle (1.5mm);
    \end{tikzpicture}
}
\end{align*}
\begin{align*}
   \begin{tikzpicture}[scale = 0.7]
    \draw[thick, ]  (-1,0)--(-.5,.866)--(.5,.866)--(1,0)--(.5,-.866)--(-.5,-.866)--(-1,0); \draw[thick ](.5,.866)--(-1,0)--(1,0)--(-.5,-.866);
     \draw[thick, black, fill= lightgray] (1,0)  circle (1.5mm);
      \draw[thick, black, fill= lightgray] (-.5,.866)  circle (1.5mm);
       \draw[thick, black, fill= white] (-1,0)  circle (1.5mm);
       \draw[thick, black, fill= white] (.5,-.866)  circle (1.5mm);
       \draw[thick, black, fill= black] (.5,.866)  circle (1.5mm);
        \draw[thick, black, fill= black] (-.5,-.866)  circle (1.5mm);
    \end{tikzpicture}
    \;\;\;
     \begin{tikzpicture}[scale = 0.7]
    \draw[thick, ]  (-1,0)--(-.5,.866)--(.5,.866)--(1,0)--(.5,-.866)--(-.5,-.866)--(-1,0)--(.5,-.866)--(-.5,.866)--(1,0);
    \draw[thick, black, fill= lightgray] (-1,0) circle (1.5mm);
    \draw[thick, black, fill= black] (-.5,-.866) circle (1.5mm);
     \draw[thick, black, fill= black] (-.5,.866) circle (1.5mm);
      \draw[thick, black, fill= white] (.5,-.866) circle (1.5mm);
        \draw[thick, black, fill= white] (.5,.866) circle (1.5mm);
            \draw[thick, black, fill= lightgray] (1,0) circle (1.5mm);
    \end{tikzpicture}
    \;\;\;
    %%%%6
     \begin{tikzpicture}[scale = 0.7]
    \draw[thick, ]  (-1,0)--(-.5,.866)--(.5,.866)--(1,0)--(.5,-.866)--(-.5,-.866)--(-1,0);
    \draw[thick] (-.5,.866)--(-.5,-.866)--(.5,.866)--(.5,-.866);
      \draw[thick, black, fill= black] (-.5,-.866) circle (1.5mm);
        \draw[thick, black, fill= black] (1,0) circle (1.5mm);
          \draw[thick, black, fill= white] (-.5,.866) circle (1.5mm);
            \draw[thick, black, fill= white] (.5,-.866) circle (1.5mm);
             \draw[thick, black, fill= lightgray] (.5,.866) circle (1.5mm);
              \draw[thick, black, fill= lightgray] (-1,0) circle (1.5mm);
             \end{tikzpicture}
             \;\;\;
               \begin{tikzpicture}[scale = 0.7]
    \draw[thick, ]  (-1,0)--(-.5,.866)--(.5,.866)--(1,0)--(.5,-.866)--(-.5,-.866)--(-1,0); 
    \draw[thick ](.5,.866)--(-1,0);
   \draw[thick ](1,0)--(-.5,-.866);
   \draw[thick ](.5,.866)--(-.5,-.866);
     \draw[thick, black, fill= lightgray] (1,0)  circle (1.5mm);
      \draw[thick, black, fill= black] (-.5,.866)  circle (1.5mm);
       \draw[thick, black, fill= lightgray] (-1,0)  circle (1.5mm);
       \draw[thick, black, fill= white] (.5,-.866)  circle (1.5mm);
       \draw[thick, black, fill= white] (.5,.866)  circle (1.5mm);
        \draw[thick, black, fill= black] (-.5,-.866)  circle (1.5mm);
    \end{tikzpicture}
    \;\;\;
     \begin{tikzpicture}[scale = 0.7]
    \draw[thick, ]  (-1,0)--(-.5,.866)--(.5,.866)--(1,0)--(.5,-.866)--(-.5,-.866)--(-1,0)--(.5,-.866);
    \draw[thick, ](-.5,.866)--(1,0);
     \draw[thick, ](-1,0)--(1,0);
    \draw[thick, black, fill= lightgray] (-1,0) circle (1.5mm);
    \draw[thick, black, fill= black] (-.5,-.866) circle (1.5mm);
     \draw[thick, black, fill= white] (-.5,.866) circle (1.5mm);
      \draw[thick, black, fill= white] (.5,-.866) circle (1.5mm);
        \draw[thick, black, fill= lightgray] (.5,.866) circle (1.5mm);
            \draw[thick, black, fill= black] (1,0) circle (1.5mm);
    \end{tikzpicture}
    \;\;\;
     \begin{tikzpicture}[scale = 0.7]
    \draw[thick, ]  (-1,0)--(-.5,.866)--(.5,.866)--(1,0)--(.5,-.866)--(-.5,-.866)--(-1,0);
    \draw[thick] (-.5,-.866)--(-.5,.866)--(.5,-.866)--(.5,.866);
        \draw[thick, black, fill= white] (-1,0) circle (1.5mm);
            \draw[thick, black, fill= white] (.5,-.866) circle (1.5mm);
            \draw[thick, black, fill= black] (-.5,-.866) circle (1.5mm);
              \draw[thick, black, fill= lightgray] (-.5,.866) circle (1.5mm);
                \draw[thick, black, fill= black] (.5,.866) circle (1.5mm);
                  \draw[thick, black, fill= lightgray] (1,0) circle (1.5mm);
    \end{tikzpicture}
               \;\;\;
    \begin{tikzpicture}[scale = 0.7]
    \draw[thick, ]  (-1,0)--(-.5,.866)--(.5,.866)--(1,0)--(.5,-.866)--(-.5,-.866)--(-1,0); \draw[thick] (.5,.866)--(-1,0)--(.5,-.866)--(.5,.866);
     \draw[thick, black, fill= black] (.5,.866)  circle (1.5mm);
      \draw[thick, black, fill= black] (-.5,-.866) 
      circle (1.5mm);
       \draw[thick, black, fill= white] (.5,-.866)  circle (1.5mm);
          \draw[thick, black, fill= white] (-.5,.866)  circle (1.5mm);
          \draw[thick, black, fill= lightgray] (1,0) circle (1.5mm);
             \draw[thick, black, fill= lightgray] (-1,0) circle (1.5mm);
    \end{tikzpicture}
      \;\;\;
         \begin{tikzpicture}[scale = 0.7]
    \draw[thick, ]  (-1,0)--(-.5,.866)--(.5,.866)--(1,0)--(.5,-.866)--(-.5,-.866)--(-1,0);
    \draw[thick] (-.5,-.866)--(-.5,.866)--(1,0)--(-.5,-.866);
        \draw[thick, black, fill= lightgray] (-1,0) circle (1.5mm);
         \draw[thick, black, fill= black] (-.5,-.866) circle (1.5mm);
         \draw[thick, black, fill= white] (-.5,.866) circle (1.5mm);
   \draw[thick, black, fill= lightgray] (1,0) circle (1.5mm);
     \draw[thick, black, fill= white] (.5,-.866) circle (1.5mm);
     \draw[thick, black, fill= black] (.5,.866) circle (1.5mm);
    \end{tikzpicture}
\end{align*}
\caption{The 14 triangulations of the regular convex hexagon, sorted according to type. The top row is the set \(\textup{Tri}_6^{(3,2,1)}\), the bottom row is the set \(\textup{Tri}_6^{(2,2,2)}\).
}
\label{fig:hextri}       % Give a unique label
\end{figure}

\subsection{Triangulations and injectively 3-colored rooted trees}\label{tritreesec}

We describe now how to associate an injectively 3-colored rooted tree to a triangulation.
Fix an orientation of the convex \(n\)-gon so that the bottom side \(s\) is on the \(x\)-axis. Color the left vertex of \(s\) with \(1\), and the right vertex of \(s\) with \(2\). Then for \(\tau \in \textup{Tri}_n\), this assignment extends to a unique proper 3-coloring of \(\tau\).

We define a graph \(T\) associated to \(\tau\) as follows. The vertices of \(T\) will be placed at the center of every triangle in \(\tau\). We add an edge between two vertices whenever their corresponding triangles share an edge. It is clear that \(T\) as constructed is a tree---this is often called the `dual tree' of the triangulation \(\tau\). 
Now we add direction; the root \(\mathsf{r}_T(1)\) of \(T\) we designate to be the vertex belonging to the `base' triangle which contains the edge \(s\), and we orient all edges in \(T\) so that \(T\) is a rooted tree.

Now we define a 3-coloring \(\ell_T\) on \(T\) as follows. Set \(\ell_T(\mathsf{r}_T(1)) = 3\). For every other vertex \(w \in T\), there is a unique arrow \(a \in T\) with target \(w\). This arrow crosses a diagonal in \(\tau\) with endpoints colored \(c_1, c_2 \in \{1,2, 3\}\). Then we set \(\ell_T(w)\) to be the unique element of \(\{1,2, 3\} \backslash \{c_1, c_2\}\).
That \( \ell_T\) is an injective 3-coloring for \(T\) quickly follows from the fact that \(\tau\) is properly 3-colored. 

Therefore the assignment \(\chi(\tau) = T\) gives a well-defined function 
\begin{align}\label{chimap}
\chi:\textup{Tri}_n \to \mathcal{T}_{n-2,3}
\end{align}
See Figure~\ref{fig:chiex} for an example of \(\chi\) in action.
It is straightforward to see that \(\chi\) is injective, so by Theorem~\ref{areCat} we have that \(\chi\) is a bijection.

\begin{figure}[h]
\begin{align*}
\hackcenter{}
\hackcenter{
\begin{tikzpicture}[scale=1.4]
        \draw[thick, dotted, gray] (-0.7,0)--(1.7,0);
    \draw[thick, black] (0,0)--(1,0)--(1.707,.707)--(1.707,1.707)--(1,2.41)--(0,2.41)--(-.707,1.707)--(-.707,.707)--(0,0)--(0,2.41)--(1,0);
    \draw[thick, black]
    (1,0)--(1.707,1.707)--(0,2.41)--(-.707,.707);
    \draw[thick, black, fill = black ] (0,0) circle (.5 mm);
        \draw[thick, black, fill = black] (1,0) circle (.5mm);
 \draw[thick, black, fill = black] (0,2.41) circle (.5mm);
  \draw[thick, black, fill = black] (1.707,1.707) circle (.5mm);
  \draw[thick, black, fill = black] (1.707,.707) circle (.5mm);
  \draw[thick, black, fill = black] (1,2.41) circle (.5mm);
  \draw[thick, black, fill = black] (-.707,1.707) circle (.5mm);
  \draw[thick, black, fill = black] (-.707,.707) circle (.5mm);
    \end{tikzpicture}
    }
 \;\;\;
    \rightsquigarrow
\;\;\;
    \hackcenter{
    \begin{tikzpicture}[scale=1.4]
        \draw[thick, dotted, red] (-0.7,0)--(1.7,0);
    \draw[thick, red] (0,0)--(1,0)--(1.707,.707)--(1.707,1.707)--(1,2.41)--(0,2.41)--(-.707,1.707)--(-.707,.707)--(0,0)--(0,2.41)--(1,0);
    \draw[thick, red]
    (1,0)--(1.707,1.707)--(0,2.41)--(-.707,.707);
    \draw[thick, red, fill=white] (0,0) circle (1 mm);
        \draw[thick, red, fill=white] (1,0) circle (1.mm);
 \draw[thick, red, fill=white] (0,2.41) circle (1.mm);
  \draw[thick, red, fill=white] (1.707,1.707) circle (1.mm);
  \draw[thick, red, fill=white] (1.707,.707) circle (1.mm);
  \draw[thick, red, fill=white] (1,2.41) circle (1.mm);
  \draw[thick, red, fill=white] (-.707,1.707) circle (1.mm);
  \draw[thick, red, fill=white] (-.707,.707) circle (1mm);
         \node[red] at (0,0){ $\scriptstyle{\mathbf{1}}$};
           \node[red] at (1,0){ $\scriptstyle{\mathbf{2}}$};
             \node[red] at (0,2.41){ $\scriptstyle{\mathbf{3}}$};
               \node[red] at (1.707,1.707){ $\scriptstyle{\mathbf{1}}$};
                 \node[red] at (1.707,.707){ $\scriptstyle{\mathbf{3}}$};
                        \node[red] at (1,2.41){ $\scriptstyle{\mathbf{2}}$};
                         \node[red] at (-.707,1.707){ $\scriptstyle{\mathbf{1}}$};
                             \node[red] at (-.707,.707) { $\scriptstyle{\mathbf{2}}$};
                     \draw[thick, black, fill=white] (.3,.8) circle (1.3mm);
  \draw[thick,->, shorten >=0.16cm] (.3,.8)--(1,1.3);
    \draw[thick,->, shorten >=0.16cm] (1,1.3)--(1.5,.8);
  \draw[thick,->, shorten >=0.16cm](1,1.3)--(1,2.15);
  \draw[thick,->, shorten >=0.16cm](.3,.8)--(-.3,1);
  \draw[thick,->, shorten >=0.16cm](-.3,1)--(-.5,1.6);
  \draw[thick, black, fill=violet!30!red!30] (.3,.8) circle (1mm);
   \draw[thick, black, fill=pink!50!yellow!80] (1,1.3) circle (1mm);
    \draw[thick, black, fill=violet!30!red!30] (1.5,.8) circle (1mm);
     \draw[thick, black, fill=cyan!30] (1,2.15) circle (1mm);
      \draw[thick, black, fill=cyan!30] (-.3,1) circle (1mm);
       \draw[thick, black, fill=pink!50!yellow!80] (-.5,1.6) circle (1mm);
       \node[] at (.3,.8){ $\scriptstyle{\mathbf{3}}$};
        \node[] at (1,1.3){ $\scriptstyle{\mathbf{1}}$};
         \node[] at (1.5,.8){ $\scriptstyle{\mathbf{3}}$};
          \node at (1,2.15){ $\scriptstyle{\mathbf{2}}$};
           \node at (-.3,1){ $\scriptstyle{\mathbf{2}}$};
            \node[] at (-.5,1.6){ $\scriptstyle{\mathbf{1}}$};
    \end{tikzpicture}
    }
   \;\;\;
    \rightsquigarrow
\;\;\;
    \hackcenter{
    \begin{tikzpicture}[scale=0.42]
 \draw[thick,->, shorten >=0.16cm]  (1,0)--(0,1.5);
  \draw[thick,->, shorten >=0.16cm]  (1,0)--(2, 1.5);
    \draw[thick,->, shorten >=0.16cm]  (0,1.5)--(-1,3);
      \draw[thick,->, shorten >=0.16cm]  (2, 1.5)--(1,3);
        \draw[thick,->, shorten >=0.16cm]  (2, 1.5)--(3,3);
    %%%%%%%%
    %%%%%%%%
        \draw[ thick, black] (1,0) circle (13pt);
     \draw[ thick, black, fill =violet!30!red!30] (1,0) circle (10pt);
         \node[] at (1,0){ $\scriptstyle{\mathbf{3}}$};  
              \draw[ thick, black, fill =cyan!30] (0,1.5) circle (10pt);
         \node[] at  (0,1.5){ $\scriptstyle{\mathbf{2}}$};  
                 \draw[ thick, black, fill =pink!50!yellow!80] (-1,3) circle (10pt);
         \node[] at  (-1,3){ $\scriptstyle{\mathbf{1}}$};  
           \draw[ thick, black, fill =pink!50!yellow!80] (2,1.5) circle (10pt);
         \node[] at  (2,1.5){ $\scriptstyle{\mathbf{1}}$};  
                    \draw[ thick, black, fill =cyan!30] (1,3) circle (10pt);
         \node[] at  (1,3){ $\scriptstyle{\mathbf{2}}$}; 
                 \draw[ thick, black, fill =violet!30!red!30] (3,3) circle (10pt);
         \node[] at  (3,3){ $\scriptstyle{\mathbf{3}}$}; 
         %
         % \draw[ thick, black, fill = violet!30!red!30] (0+8,1) circle (10pt);
            % \draw[ thick, black, fill = pink!50!yellow!80] (-2+8,0.5) circle (10pt);
           % \draw[ thick, black, fill = lime] (2+8,0.5) circle (10pt);
           % \draw[ thick, black, fill = cyan!30] (0+8,0.5) circle (10pt);
\end{tikzpicture}
}
\end{align*}
\caption{The map \(\chi\) gives a bijection between triangulations of the regular convex \(n\)-gon and injectively 3-colored rooted trees with \(n-2\) vertices and root colored by 3. 
On the left, a triangulation \(\tau \in \textup{Tri}_8\), on the right, its image \(\chi(\tau) \in \mathcal{T}_{6,3}\).
}
\label{fig:chiex}    
\end{figure}
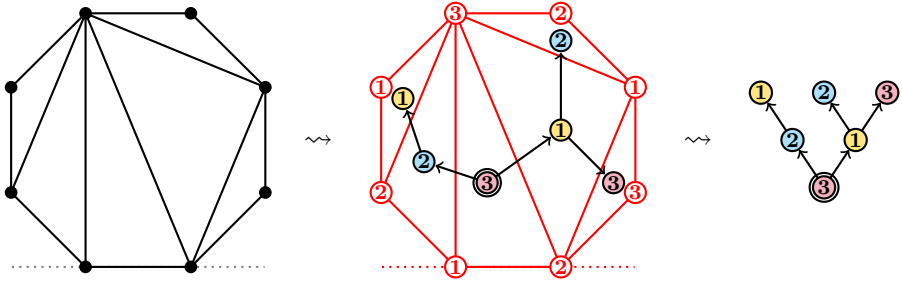

We can say more about this map. After drawing \(T\), each triangle \(\Delta\) in \(\tau\) save the base triangle has one edge \(e\) which is crossed by an incoming arrow. If the vertex \(w \in T\) in the interior of \(\Delta\) is colored \(i\) in \(T\), then the vertex of \(\Delta\) opposite \(e\) must be colored \(i\). Thus for every non-root \(w \in T\) colored \(i\), there must be a distinct associated vertex in \(\tau\) colored \(i\). We have not accounted for the colors of any of the vertices of the base triangle in \(\tau\), wherein one of each color must appear. Therefore \(\tau\), colored according to the convention above, has character \(\lambda \in \Lambda_3\) if and only if \(\chi(\tau) \in \mathcal{T}_{\lambda - \varepsilon_1 - \varepsilon_2,3}\). The reader may again compare this result with Figure~\ref{fig:chiex}.

\subsection{Enumerating triangulations by type} We now give a formula for the number of triangulations of the regular convex \(n\)-gon of a given type. See Figure~\ref{fig:tridata} for these values for \(3 \leq n \leq 16\).
\begin{Theorem}\label{TriThm}
Let \(n \geq 3\) and \(\lambda \in \Lambda_3^+(n)\). If \(\lambda_3 = 0\) then \(|\textup{Tri}^\lambda_n| = 0\). Otherwise, recalling \S\ref{CombSec}, we have
\begin{align*}
\big| \textup{Tri}_n^\lambda \big| =\frac{\textup{o}(\lambda)n(n-2)}{3(n-\lambda_1 - 1)(n-\lambda_2-1)(n-\lambda_3-1)}
{ n- \lambda_1 - 1 \choose \lambda_1 -1}
{ n- \lambda_2 - 1 \choose \lambda_2 -1}
{ n- \lambda_3 - 1 \choose \lambda_3 -1}.
\end{align*}
\end{Theorem}
\begin{proof}
Every triangle in a triangulation must contain one of each color vertex, so there is no triangulation of type \((\lambda_1, \lambda_2, 0)\), establishing the first claim. Now assume \(\lambda \in \Lambda_n^+\) and \(\lambda_3 > 0\). 

By the discussion of \ref{tritreesec} and Corollary~\ref{TreeCor}, we have that:
\begin{align}
\big| \textup{Tri}_n^\lambda \big| 
&= \sum_{\mu \in \mathcal{O}(\lambda)}
\big| \mathcal{T}_{\mu - \varepsilon_1 - \varepsilon_2, 3} \big|
= \sum_{\mu \in \mathcal{O}(\lambda)}
t_{\mu - \varepsilon_1 - \varepsilon_2, 3}\nonumber
\\
&=\sum_{\mu \in \mathcal{O}(\lambda)}
\frac{n-2}{(n - \lambda_1 - 1)(n - \lambda_2 -1)}{n - \lambda_1 - 1 \choose \lambda_1 - 1} {n - \lambda_2 - 1 \choose \lambda_2 -1}{ n -  \lambda_3 -2 \choose \lambda_3 -1}.\label{trisum}
\end{align}

\noindent{\em Case 1.} First assume that all components in \(\lambda\) are distinct: \(\lambda_1 > \lambda_2 > \lambda_3\). Then we have that (\ref{trisum}) is equal to:
\begin{align*}
&2\left[
\frac{n-2}{(n - \lambda_1 - 1)(n - \lambda_2 -1)}{n - \lambda_1 - 1 \choose \lambda_1 - 1} {n - \lambda_2 - 1 \choose \lambda_2 -1}{ n -  \lambda_3 -2 \choose \lambda_3 -1}
\right]\\
&\hspace{10mm}+
2\left[
\frac{n-2}{(n - \lambda_1 - 1)(n - \lambda_3 -1)}{n - \lambda_1 - 1 \choose \lambda_1 - 1} {n - \lambda_3 - 1 \choose \lambda_3 -1}{ n -  \lambda_2 -2 \choose \lambda_2 -1}
\right]\\
&\hspace{20mm}+
2\left[
\frac{n-2}{(n - \lambda_2 - 1)(n - \lambda_3 -1)}{n - \lambda_2 - 1 \choose \lambda_2 - 1} {n - \lambda_3 - 1 \choose \lambda_3 -1}{ n -  \lambda_1 -2 \choose \lambda_1 -1}
\right]\\
&=
\frac{2(n-2)[(n - 2\lambda_1) + (n-2 \lambda_2) + (n - 2\lambda_3)] }{(n - \lambda_1 - 1)(n - \lambda_2 -1)(n- \lambda_3 -1)}{n - \lambda_1 - 1 \choose \lambda_1 - 1} {n - \lambda_2 - 1 \choose \lambda_2 -1}{ n -  \lambda_3 -2 \choose \lambda_3 -1} \\
&=
\frac{2n(n-2)}{(n - \lambda_1 - 1)(n - \lambda_2 -1)(n- \lambda_3 -1)}{n - \lambda_1 - 1 \choose \lambda_1 - 1} {n - \lambda_2 - 1 \choose \lambda_2 -1}{ n -  \lambda_3 -2 \choose \lambda_3 -1},
\end{align*}
and the result follows since in this situation \(o(\lambda) =6\).\\

\noindent{\em Case 2.} Next assume that exactly two components in \(\lambda\) are equal; without loss of generality, we assume \(\lambda_1 = \lambda_2 > \lambda_3\). Then we have that (\ref{trisum}) is equal to:
\begin{align*}
&
\frac{n-2}{(n - \lambda_1 - 1)(n - \lambda_1 -1)}{n - \lambda_1 - 1 \choose \lambda_1 - 1} {n - \lambda_1 - 1 \choose \lambda_1 -1}{ n -  \lambda_3 -2 \choose \lambda_3 -1}
\\
&\hspace{10mm}+
2\left[
\frac{n-2}{(n - \lambda_1 - 1)(n - \lambda_3 -1)}{n - \lambda_1 - 1 \choose \lambda_1 - 1} {n - \lambda_3 - 1 \choose \lambda_3 -1}{ n -  \lambda_1 -2 \choose \lambda_1 -1}
\right]\\
&=\frac{(n-2)[(n- 2 \lambda_3) + 2(n-2 \lambda_1)]}{(n- \lambda_1 -1 )^2(n - \lambda_3 -1)}
{n - \lambda_1 - 1 \choose \lambda_1 -1}^2 {n- \lambda_3 - 1 \choose \lambda_3 -1}\\
&=\frac{n(n-2)}{(n- \lambda_1 -1 )^2(n - \lambda_3 -1)}
{n - \lambda_1 - 1 \choose \lambda_1 -1}^2 {n- \lambda_3 - 1 \choose \lambda_3 -1},
\end{align*}
and the result follows since in this situation \(o(\lambda) =3\).\\

\noindent{\em Case 3.} Next assume that all components in \(\lambda\) are equal: \(\lambda_1 = \lambda_2 = \lambda_3\). Noting then that \(\lambda_1 = n/3\), we have that (\ref{trisum}) is equal to:
\begin{align*}
\frac{n-2}{(n- \lambda_1 - 1)^2}{n - \lambda_1 - 1 \choose \lambda_1 -1}^2 {n - \lambda_1 - 2 \choose \lambda_1 - 1}
&=
\frac{(n-2)(n - 2 \lambda_1)}{(n - \lambda_1 - 1)^3}{ n- \lambda_1 - 1 \choose \lambda_1 -1}^3\\
&=
\frac{n(n-2)}{3(n - \lambda_1 - 1)^3}{ n- \lambda_1 - 1 \choose \lambda_1 -1}^3,
\end{align*}
and the result follows since in this situation \(o(\lambda) = 1\). 
Thus, in any case, the theorem holds.
\end{proof}

\begin{figure}[h]
\begin{align*}
\begin{array}{cc|}
\hline
\multicolumn{1}{|c|}{\lambda \in \Lambda_3^+(n)} &\big|\textup{Tri}^\lambda_n\big| \\
\hline
\multicolumn{1}{|c|}{(1,1,1)} & 1\\
\hline
\multicolumn{1}{|c|}{(2,1,1)} & 2\\
\hline
\multicolumn{1}{|c|}{(2,2,1)} & 5\\
\hline
\multicolumn{1}{|c|}{(3,2,1)} & 6\\
\multicolumn{1}{|c|}{(2,2,2)} & 8\\
\hline
\multicolumn{1}{|c|}{(3,3,1)} & 7\\
\multicolumn{1}{|c|}{(3,2,2)} & 35\\
\hline
\multicolumn{1}{|c|}{(4,3,1)} & 8\\
\multicolumn{1}{|c|}{(4,2,2)} & 16\\
\multicolumn{1}{|c|}{(3,3,2)} & 108\\
\hline
\multicolumn{1}{|c|}{(4,4,1)} & 9\\
\multicolumn{1}{|c|}{(4,3,2)} & 252 \\
\multicolumn{1}{|c|}{(3,3,3)} & 168 \\
\hline
\\
\\
\\
\end{array}
\;\;
\begin{array}{cc|}
\hline
\multicolumn{1}{|c|}{\lambda \in \Lambda_3^+(n)} & \big|\textup{Tri}^\lambda_n\big|\\
\hline
\multicolumn{1}{|c|}{(5,4,1)}
& 10\\
\multicolumn{1}{|c|}{(5,3,2) }& 100\\
\multicolumn{1}{|c|}{(4,4,2)} & 320\\
\multicolumn{1}{|c|}{(4,3,3) }& 1000\\
\hline
\multicolumn{1}{|c|}{(5,5,1)} & 11\\
\multicolumn{1}{|c|}{(5,4,2)} & 660\\
\multicolumn{1}{|c|}{(5,3,3)} & 891\\
\multicolumn{1}{|c|}{(4,4,3)} & 3300\\
\hline
\multicolumn{1}{|c|}{(6,5,1)} & 12\\
\multicolumn{1}{|c|}{(6,4,2)} & 240\\
\multicolumn{1}{|c|}{(6,3,3)} & 294\\
\multicolumn{1}{|c|}{(5,5,2)} & 750\\
\multicolumn{1}{|c|}{(5,4,3) }& 10500\\
\multicolumn{1}{|c|}{(4,4,4)} & 5000\\
\hline
\\
\\
\end{array}
\;\;
\begin{array}{cc|}
\hline
\multicolumn{1}{|c|}{\lambda \in \Lambda_3^+(n)} & \big|\textup{Tri}^\lambda_n\big| \\
\hline
\multicolumn{1}{|c|}{(6,6,1)} & 13\\
\multicolumn{1}{|c|}{(6,5,2)} & 1430\\
\multicolumn{1}{|c|}{(6,4,3)} & 8008\\
\multicolumn{1}{|c|}{(5,5,3) }& 14300\\
\multicolumn{1}{|c|}{(5,4,4) }& 35035\\
\hline
\multicolumn{1}{|c|}{(7,6,1) }& 14\\
\multicolumn{1}{|c|}{(7,5,2) }& 490\\
\multicolumn{1}{|c|}{(7,4,3) }& 2352\\
\multicolumn{1}{|c|}{(6,6,2) }& 1512\\
\multicolumn{1}{|c|}{(6,5,3) }& 39690\\
\multicolumn{1}{|c|}{(6,4,4) }& 43904\\
\multicolumn{1}{|c|}{(5,5,4) }& 120050\\
\hline
\\
\\
\\
\\
\end{array}
\;\;
\begin{array}{|c|c|}
\hline
\multicolumn{1}{|c|}{\lambda \in \Lambda_3^+(n)} & \textup{Tri}^\lambda_n \\
\hline
\multicolumn{1}{|c|}{(7,7,1)} & 15\\
\multicolumn{1}{|c|}{(7,6,2)} & 2730\\
\multicolumn{1}{|c|}{(7,5,3)} & 27300\\
\multicolumn{1}{|c|}{(7,4,4)} & 28080\\
\multicolumn{1}{|c|}{(6,6,3)} & 47775\\
\multicolumn{1}{|c|}{(6,5,4)} & 458640\\
\multicolumn{1}{|c|}{(5,5,5)} & 178360\\
\hline
\multicolumn{1}{|c|}{(8,7,1)}& 16\\
\multicolumn{1}{|c|}{(8,6,2)}& 896\\
\multicolumn{1}{|c|}{(8,5,3)} & 7392\\
\multicolumn{1}{|c|}{(8,4,4)} & 7200\\
\multicolumn{1}{|c|}{(7,7,2)} & 2744\\
\multicolumn{1}{|c|}{(7,6,3)} & 120736\\
\multicolumn{1}{|c|}{(7,5,4)} & 493920\\
\multicolumn{1}{|c|}{(6,6,4)} & 658560\\
\multicolumn{1}{|c|}{(6,5,5)} & 1382976\\
\hline
\end{array}
\end{align*}
\caption{Enumeration of triangulations of the regular convex \(n\)-gon by type, for \(3 \leq n \leq 16\).
}
\label{fig:tridata}       % Give a unique label
\end{figure}

\subsection{Equitable 3-colorings of triangulations}
As defined in \cite{Meyer}, a proper \(k\)-coloring of a graph \(G = (V, E)\) is said to be {\em equitable} if the number of vertices of distinct colors differs by at most one; i.e. \(|V_i - V_j| \leq 1\) for all \(i,j \in \{1, \ldots, k\}\).

The following proposition, in conjunction with Theorem~\ref{TriThm}, can be used to count the number of equitably 3-colorable triangulations of the convex polygon, or assess the probability that a randomly chosen triangulation is equitably 3-colorable. 

\begin{Proposition}
Take \(\lambda = (\lceil n/3 \rceil, n - \lceil n/3 \rceil - \lfloor n/3 \rfloor, \lfloor n/3 \rfloor)\). Then \(\textup{Tri}^\lambda_n\) is the set of triangulations \(\tau\in \textup{Tri}_n\) that possess an equitable proper 3-coloring.
\end{Proposition}

\begin{proof}
Let \(\tau \in \textup{Tri}_n\). As \(\tau\) has a unique proper 3-coloring, it follows that \(\tau\) is equitably 3-colorable if and only if \(\textup{type}(\tau)_1 - \textup{type}(\tau)_3 \leq 1\). The result follows.
\end{proof}

\section{Future work/related questions}

\subsection{Further refinement} The authors came across Theorem~\ref{BigThm}, and in particular the key polynomial \(P_k\) in (\ref{ThePoly1}) via trial and error and a large number of calculations. The current proof relies on induction and heavy manipulations of binomial products and sums. It would be interesting to find a proof which more satisfactorily explains the combinatorial `raison d'\^etre' for \(P_k\).

In this direction we make a brief note. The coefficients of monomials in the variables \(x_i\), \(y_i\) in \(P_k\) take on values in \(\mathbb{Z}\). But if we utilize a change of variables \(z_i := x_i - y_i\), the coefficients of monomials in \(z_i, y_i\) in \(P_k\) appear to take values purely in \(\mathbb{Z}_{\geq 0}\). 
Writing \(\mu_i:= \lambda_i - i_{\underline c}\) for the number of non-root \(i\)-colored vertices in a \((\lambda, \underline c)\)-forest, we would have then that
\begin{align*}
P(\lambda_1, \ldots, \lambda_k, 1_{\underline c}, \ldots, k_{\underline c}) = 
P(\mu_1 + 1_{\underline c}, \ldots, \mu_k + k_{\underline c}, 1_{\underline c}, \ldots, k_{\underline c})
\end{align*}
evaluates to a sum of non-negative integers. This suggests there may be a statistic on \((\lambda, \underline{c})\)-forests which allows for more refined enumeration, and perhaps a more satisfactory explanation of the involvement of the polynomial \(P_k\) in Theorem~\ref{BigThm}.

\section{Appendix}\label{AppSec}

\subsection{Preparing to prove Lemma~\ref{ThePolyLem}(iv)}
We begin by recounting some basic binomial coefficient identities.
\begin{Lemma}\label{combid}
For all \(s,m,r,u \in \mathbb{Z}_{\geq 0}\), we have
\begin{enumerate}
\item \(
\sum_{r=0}^u {s \choose m+r} {u \choose r} = {s+u \choose m+u}
\)
\item 
\(
\sum_{r=0}^u(-1)^r{s-r \choose m}{u \choose r} = {s-u \choose m-u}
\)
\item 
\(
\sum_{r=0}^u(-1)^r r{s-r \choose m}{u \choose r} = -u{s-u \choose m-u+1}
\)
\end{enumerate}
\end{Lemma}
\begin{proof}
Parts (i) and (ii) are standard; see for instance \cite[Vol.4, (6.69), (10.20)]{Gould}.
For (iii), we have:
\begin{align*}
\sum_{r=0}^u (-1)^r{s-r \choose m} {u \choose r} &= u \sum_{r=0}^u (-1)^r{s-r \choose m} {u-1 \choose r-1}
= u\left[ \sum_{r=0}^u {s-r \choose m} \left({u \choose r} - {u-1 \choose r} \right)  \right]\\
&=u \left[ {s-u \choose m- u} - {s - u+1 \choose m-u+1}\right]
=
-u {s-u \choose m-u+1},
\end{align*}
where we have applied (ii) for the third equality.
\end{proof}

The following technical result will be used repeatedly in the proof of Lemma~\ref{LRequal}.\begin{Lemma}\label{wildLem}  
Let \(\kappa, \phi, \tau_1, \tau_2, \tau_3,c,d \in \mathbb{Z}_{\geq 0}\), and \(\varepsilon \in \{0,1\}\). Then we have that
\begin{align}\label{S}
&\sum (-1)^{\phi + \alpha_1 + \alpha_3 + \omega + \eta_1 + \eta_2 + \eta_3}\,
2^{\tau_1 + \tau_2 - \alpha_1 - \eta_1 - \eta_2}{\kappa - \alpha_1 - \tau_2 - \tau_3 + \omega - \alpha_3 - c \choose \phi - \kappa + \tau_1 + \tau_2 + \omega + d}\\
& \hspace{40mm} \times
{\tau_3 \choose \omega} {\tau_1 \choose \alpha_1} {\omega \choose \alpha_3} {\tau_1 - \alpha_1 \choose \eta_1} {\tau_2 \choose \eta_2} {\omega- \alpha_3 \choose \eta_3}(\alpha_1 - \omega + \alpha_3)^\varepsilon \nonumber
\end{align}
is equal to
\begin{align*}
(-1)^{\varphi + \varepsilon}\tau_1^\varepsilon
{\kappa - \tau_1 - \tau_2 - c \choose \phi - \kappa + \tau_2 + \tau_3 + d + \varepsilon},
\end{align*}
where the sum in (\ref{S}) ranges over 
\begin{align*}
\omega \in [0,\tau_3];
\;\;\;
\alpha_1 \in [0, \tau_1];
\;\;\;
\alpha_3 \in [0, \omega];
\;\;\;
\eta_1 \in [0, \tau_1 - \alpha_1];
\;\;\;
\eta_2 \in [0,\tau_2];
\;\;\;
\eta_3 \in [0, \omega- \alpha_3].
\end{align*}
\end{Lemma}
\begin{proof}
Writing \(S\) for the sum in (\ref{S}), we may organize terms so that
\begin{align*}
S=
\sum (-1)^{\phi + \alpha_1} H_1H_2H_3{\kappa - \alpha_1 - \tau_2 - \tau_3 + \omega - \alpha_3 - c \choose \phi - \kappa + \tau_1 + \tau_2 + \omega + d}{\tau_3 \choose \omega} {\tau_1 \choose \alpha_1} {\omega \choose \alpha_3}( \alpha_1 - \omega + \alpha_3)^\varepsilon,
\end{align*}
where the sum is over all
\begin{align*}
\omega \in [0,\tau_3]; \qquad \alpha_1 \in [0,\tau_1]; \qquad \alpha_3 \in [0, \omega], 
\end{align*}
and
\begin{align*}
H_1 &= 
\sum_{\eta_1 = 0}^{\tau_1 - \alpha_1} (-1)^{\eta_1} 2^{\tau_1 - \alpha_1 - \eta_1}{\tau_1 - \alpha_1 \choose \eta_1}=1
\\
H_2 &=
\sum_{\eta_2 = 0}^{\tau_2} (-1)^{\tau_2 - \eta_2} 2^{\tau_2 - \eta_2}{\tau_2 \choose \eta_2}=1
\\
H_3&=
\sum_{\eta_3 = 0}^{\omega- \alpha_3} (-1)^{\omega - \alpha_3 - \eta_3} {\omega - \alpha_3 \choose \eta_3} = \delta_{\omega, \alpha_3}.
\end{align*}
Therefore
\begin{align*}
S &= \sum_{\alpha_1 =0}^{\tau_1} (-1)^{\phi + \alpha_1} \alpha_1^\varepsilon {\tau_1 \choose \alpha_1} \left(\sum_{\omega = 0}^{\tau_3} {\kappa - \alpha_1 - \tau_2 - \tau_3 - c \choose \phi - \kappa + \tau_1 + \tau_2 + \omega + d}{\tau_3 \choose \omega}\right)\\
&=
\sum_{\alpha_1 =0}^{\tau_1} (-1)^{\phi + \alpha_1} \alpha_1^\varepsilon
{ \kappa - \alpha_1 - \tau_2 - c \choose \phi - \kappa + \tau_1 + \tau_2 + \tau_3 + d} {\tau_1 \choose \alpha_1}.
\end{align*}
The second equality above follows from Lemma~\ref{combid}(i). Now, applying Lemma~\ref{combid}(ii) in the case \(\varepsilon=0\), or Lemma~\ref{combid}(iii) in the case \(\varepsilon = 1\), gives the final result.
\end{proof}

\subsection{ Proving Lemma~\ref{ThePolyLem}(iv)}
We first define a number of ancillary polynomials in the polynomial ring \(R_k' = \mathbb{Z}[x_1, \ldots, x_{k-1}, z, y_1, \ldots, y_k]\). For \(A \subseteq [1,k-1]\), let \(A^c\) denote the complement \(A^c = [1,k-1] \backslash A\), and define \(Q(x_1, \ldots, x_{k-1}, z, y_1, \ldots, y_k)\) by:
\begin{align}
Q &=\hspace{-3mm} \sum_{A \subseteq [1,k-1]}
(-1)^{|A|} \left( 
x_{(A, + )}z^{k-|A|-1} - |A| z^{k-|A|} 
+x_{( A^c, +)} z^{k-|A|-2}y_k + |A|z^{k-|A|-1}y_k
\right)
y_{(A, \times )}.
\label{ThePoly3}
\end{align}
We then set:
\begin{align}
\mathscr{L} = Q \cdot \prod_{i \in [1,k-1]} (z - x_i - 1). \label{Ldef}
\end{align}
We also define 
\begin{align*}
J^{(1)}&= \sum
(-1)^{k+\ell + 1 + |M| + |V|}
2^{|V| - |V \cap M|}
|A|
{k - |A| - 1 \choose \ell}
z^{k + \ell - 1 - |V| - |W|}
x_{(V, \times)}y_{(W, \times)}y_{(B, \times)}y_k;\\
J^{(2)}&=
\sum
(-1)^{k+\ell + |M| + |V|}2^{|V| - |V \cap M|}
{k - |A| - 2 \choose \ell}
z^{k + \ell - 1 - |V| - |W|}x_{(A^c,+)}x_{(V, \times)}y_{(W, \times)}y_{(B, \times)}y_k;\\
J^{(3)}&=
\sum
(-1)^{k+\ell + 1+|M| + |V|}2^{|V| - |V \cap M|}
{k - |A| - 2 \choose \ell }
z^{k + \ell - 1 - |V| - |W|}
x_{(A^c,+)}x_{(V, \times)}y_{(W, \times)}y_{(B, \times)};\\
J^{(4)}&=
\sum
(-1)^{k+\ell + 1+|M| + |V|}2^{|V| - |V \cap M|}
{k - |A| - 1 \choose \ell}
z^{k + \ell - 1 - |V| - |W|}
x_{(A,+)}x_{(V, \times)}y_{(W, \times)}y_{(B, \times)};\\
J^{(5)}&=
\sum
(-1)^{k+\ell + 1+|M| + |V|}2^{|V| - |V \cap M|}
|A|
{k - |A| - 1 \choose \ell-1}
z^{k + \ell - 1 - |V| - |W|}
x_{(V, \times)}y_{(W, \times)}y_{(B, \times)},
\end{align*}
where in each case, the sum is over:
\begin{align*}
A,M,B,V,W \subseteq [1,k-1]; 
\;\;\;
\ell \in \mathbb{Z}_{\geq 0}
\end{align*}
such that
\begin{align}\label{AMBcrit}
A \backslash M \subseteq B \subseteq A; \qquad
V \cap W = \varnothing; \qquad
M \subseteq V \cup W.
\end{align}
We set:
\begin{align*}
\mathscr{R} = J^{(1)} + J^{(2)} + J^{(3)} + J^{(4)} + J^{(5)}.
\end{align*}

\begin{Lemma}\label{subz}
Substituting \(z = x_{([1,k],+)}\) in \(\mathscr{L}\) and \(\mathscr{R}\) yields the lefthand and righthand sides of (\ref{TheEqn}), respectively.
\end{Lemma}
\begin{proof}
It is straightforward to check from (\ref{ThePoly2}, \ref{ThePoly3}) that 
\begin{align}
P_k (x_1, \ldots, x_k, y_1, \ldots, y_k) = Q(x_1, \ldots, x_{k-1}, x_{([1,k],+)}, y_1, \ldots, y_k),\label{PQsame}
\end{align}
and thus the claim for \(\mathscr{L}\) immediately follows from (\ref{Ldef}).

Expanding and collecting powers of \(z\), we have that \( Q(x_1, \ldots, x_{k-1}, z-1, y_1, \ldots, y_{k-1}, y_k - 1)\) is equal to
\begin{align*}
&\sum_{A \subseteq [1,k-1]}\sum_{ \ell =0}^\infty(-1)^{k - \ell - 1}\bigg[|A|{k - |A| - 1 \choose \ell} y_k - {k - |A| - 2 \choose \ell} x_{(A^c, +)} y_k\\
& \hspace{10mm}+ {k - |A| - 2 \choose \ell} x_{(A^c, +)}
 + {k - |A| -1 \choose \ell}x_{(A, +)} + |A|{k - |A| -1 \choose \ell -1}\bigg]
 z^\ell y_{(A, \times)}.
\end{align*}
Therefore, for \(M \subseteq [1,k-1]\), expanding 
\begin{align}
Q^{(M)}:=
Q(x_1, \ldots, x_{k-1}, z-1, y_1 + \delta_{1 \in M}, \ldots, y_{k-1} + \delta_{(k-1) \in M}, y_k -1)
\label{QMdef}
\end{align}
yields
\begin{align*}
Q^{(M)}&=
\sum_{A \subseteq [1,k-1]}
\sum_{A\backslash M \subseteq B \subseteq A}
\sum_{ \ell =0}^\infty(-1)^{k - \ell - 1}\bigg[|A|{k - |A| - 1 \choose \ell} y_k - {k - |A| - 2 \choose \ell} x_{(A^c, +)} y_k\\
& \hspace{10mm}+ {k - |A| - 2 \choose \ell} x_{(A^c, +)}
 + {k - |A| -1 \choose \ell}x_{(A, +)} + |A|{k - |A| -1 \choose \ell -1}\bigg]
 z^\ell y_{(B, \times)}.
\end{align*}
We also have that
\begin{align*}
\prod_{s \in M} (x_s - y_s) \prod_{s \in M^c} (z- 2x_s +y_s) 
&=
\sum_{
\substack{
V,W \subseteq [1,k-1]\\
V \cap W = \varnothing\\
M \subseteq V \cup W
}}
(-1)^{|M| + |V|}2^{|V| - |V \cap M|} z^{k - 1 - |V| - |W|}x_{(V,\times)} y_{(W, \times)}.
\end{align*}
Therefore, expanding and collecting like terms gives
\begin{align}
\sum_{M \subseteq [1,k-1]}Q^{(M)}\prod_{s \in M} (x_s - y_s) \prod_{s \in M^c} (z- 2x_s +y_s) 
=
J^{(1)} + J^{(2)} + J^{(3)} + J^{(4)} + J^{(5)} = \mathscr{R}.\label{compareRight}
\end{align}
Now, it follows from (\ref{TweakedPoly}, \ref{PQsame}, \ref{QMdef}) that
\begin{align*}
P_k^{(M)}(x_1, \ldots, x_k, y_1, \ldots, y_k) = Q^{(M)}(x_1, \ldots, x_{k-1}, x_{([1,k],+)}, y_1, \ldots, y_k).
\end{align*}
So, comparing the lefthand side of (\ref{compareRight}) to the righthand side of (\ref{TheEqn}) yields the result for \(\mathscr{R}\).
\end{proof}

For pairwise disjoint \(T_1, T_2, T_3 \subseteq [1,k-1]\), \(q \in \mathbb{Z}_{\geq 0}\), we define associated monomials in \(R_k'\):
\begin{align*}
m_{T_1, T_2, T_3, q} := x_{(T_1 \sqcup T_2, \times)} y_{(T_2 \sqcup T_3, \times)} z^q
\end{align*}
We now show that \(\mathscr{L} = \mathscr{R}\), by expanding each in terms of these monomials.

\begin{Lemma}\label{LRequal}
Both \(\mathscr{L}\), \(\mathscr{R}\) are equal to:
\begin{align}
&
\sum
(-1)^q
\left[
(|T_2| + |T_3|)
{k - |T_1| - |T_2| - 1
\choose
q - k + |T_2| + |T_3| + 1}
-
|T_1|
{k - |T_1| - |T_2| \choose q - k + |T_2| + |T_3| + 2}
\right]
y_km_{T_1, T_2, T_3, q}\nonumber\\
&\hspace{5mm}+
\sum
(-1)^{q+1}
{k - |T_1| - |T_2| - 1 \choose q - k + |T_2| + |T_3| + 2} x_j y_k m_{T_1, T_2, T_3, q}\nonumber\\
&\hspace{10mm}+
\sum
(-1)^q
\left[
(|T_2| + |T_3|)
{k - |T_1| - |T_2| - 1 \choose q - k + |T_2| + |T_3|}
+|T_2|
{k - |T_1| - |T_2| \choose q - k + |T_2| + |T_3| +1}
\right]
m_{T_1, T_2, T_3, q}\nonumber\\
&\hspace{15mm}+
\sum (-1)^q
{k - |T_1| - |T_2| - 1 \choose q - k + |T_2| + |T_3| + 1}
x_j m_{T_1, T_2, T_3, q},\label{bigugly}
\end{align}
where the first and third sums are over disjoint \(T_1, T_2, T_3 \subseteq [1,k-1]\), the second sum is over disjoint \(T_1, T_2, T_3 \subseteq [1,k-1]\), \(j \in T_1\), and the fourth sum is over disjoint \(T_1, T_2, T_3 \subseteq [1,k-1]\), \(j \in T_2\).
\end{Lemma}

\begin{proof}
For compactness of notation, we will often use lowercase letters to refer to cardinality of sets throughout this proof, writing for instance \(t_1 = |T_1|\), etc. We also use \([m]f\) to denote the coefficient of a monomial \(m\) in the expansion of a polynomial \(f\). 

We first prove the claim for \(\mathscr{L}\). 
Applying the binomial theorem, we have
\begin{align}\label{usefulbin}
\prod_{i \in [1,k-1]}
(z-x_i - 1)
=
\sum_{V \subseteq [1,k-1], r \in \mathbb{Z}_{\geq 0}}
(-1)^{k-1 - r}{k -1 -|V| \choose r}x_{(V, \times )}z^r.
\end{align}

By consideration of the definition of \(\mathscr{L}\), it is clear that the only monomials involved in \(\mathscr{L}\) are those which appear in the lemma statement. So one must only compute:
\begin{align}
[y_km_{T_1, T_2, T_3, q}]\mathscr{L};
\;\;\;
[x_jy_km_{T_1, T_2, T_3, q}]\mathscr{L};
\;\;\;
[m_{T_1, T_2, T_3, q}]\mathscr{L};
\;\;\;
[x_iy_km_{T_1, T_2, T_3, q}]\mathscr{L},
\label{allmonoms}
\end{align}
for disjoint \(T_1, T_2, T_3 \subseteq [1,k-1]\), \(q \in \mathbb{Z}_{\geq 0}\), \(j \in T_1\), \(i \in T_2\), and demonstrate agreement with lines 1--4 of (\ref{bigugly}). We do this via a series of claims:

\vspace{2mm}
\noindent {\em Claim A1. For all disjoint \(T_1, T_2, T_3 \subseteq [1,k-1]\), \(q \in \mathbb{Z}_{\geq 0}\), we have:
\begin{align*}
[y_km_{T_1, T_2, T_3, q}]\mathscr{L}
=(-1)^{q}(t_2+ t_3){k - t_1 - t_2 -1 \choose q - k + t_2 + t_3 + 1} + (-1)^{q+1}t_1{k - t_1 - t_2 \choose q - k + t_2 + t_3 + 2}.
\end{align*}
}

\noindent {\em Claim A2. For all disjoint \(T_1, T_2, T_3 \subseteq [1,k-1]\), \(q \in \mathbb{Z}_{\geq 0}\), \(j \in T_1\), we have:
\begin{align*}
[x_jy_km_{T_1, T_2, T_3, q}]\mathscr{L}
=(-1)^{q+1}
{k - t_1 - t_2 - 1 \choose q - k + t_2 + t_3 + 2} .
\end{align*}
}

\noindent {\em Claim A3. For all disjoint \(T_1, T_2, T_3 \subseteq [1,k-1]\), \(q \in \mathbb{Z}_{\geq 0}\), we have:
\begin{align*}
[m_{T_1, T_2, T_3, q}]\mathscr{L}
=(-1)^q
(t_2 + t_3)
{k - t_1 - t_2 - 1 \choose q - k +t_2 + t_3}
+
(-1)^q
t_2
{k - t_1 - t_2 \choose q - k +t_2 + t_3 +1}.
\end{align*}
}

\noindent {\em Claim A4. For all disjoint \(T_1, T_2, T_3 \subseteq [1,k-1]\), \(q \in \mathbb{Z}_{\geq 0}\), \(j \in T_2\), we have:
\begin{align*}
[x_jm_{T_1, T_2, T_3, q}]\mathscr{L}
=(-1)^{q}
{k - t_1 - t_2 - 1 \choose q - k + t_2 + t_3 + 1} .
\end{align*}
}

\noindent
{\em Proof of Claim A1.}
By consideration of (\ref{ThePoly3},\ref{usefulbin}),
we have that \([y_k m_{T_1, T_2, T_3, q}] \mathscr{L}\) is equal to:
\begin{align}
&
[y_k m_{T_1, T_2, T_3, q}] \left(\sum
(-1)^{|A|}(-1)^{k - 1 - r}
{k -1 -|V| \choose r}
z^{k- |A|- 2}z^rx_{(A^c, +)} x_{(V, \times)}y_{(A, \times)}y_k
\right)\label{PO1}\\
&+
[y_k m_{T_1, T_2, T_3, q}] \left(\sum
(-1)^{|A|}(-1)^{k - 1 -r}|A|
{k -1 -|V| \choose r}
z^{k - |A| - 1}z^r x_{(V, \times)}y_{(A, \times)}y_k
\right),\label{PO2}
\end{align}
where the sums are over \(A,V \subseteq [1,k-1]\), \(r \in \mathbb{Z}_{\geq 0}\). First consider (\ref{PO1}). The only contribution to the coefficient is from \(A = T_2 \sqcup T_3\), \(r = q - k +t_2 + t_3 + 2\). Then \((T_1 \sqcup T_2) \cap A^c = T_1\), so \(x_{(A^c, +)}\) contributes any choice of \(x_j\) for \(j \in T_1\), and this forces \(V = (T_1 \sqcup T_2)\backslash \{j\}\). Therefore we have that (\ref{PO1}) is equal to
\begin{align}\label{PO3}
(-1)^{q+1}t_1{k - t_1 - t_2 \choose q - k + t_2 + t_3 + 2}.
\end{align}

Now consider (\ref{PO2}). The only contribution to the coefficient is from \(A = T_2 \sqcup T_3\), \(r = q - k +t_2 + t_3 + 1\), \(V = T_1 \sqcup T_2\). Therefore we have that (\ref{PO2}) is equal to
\begin{align}\label{PO4}
(-1)^{q}(t_2+ t_3){k - t_1 - t_2 -1 \choose q - k + t_2 + t_3 + 1}.
\end{align}

Thus, by (\ref{PO3}, \ref{PO4}), we have that Claim A1 is verified, and \([y_k m_{T_1, T_2, T_3, q}]\mathscr{L}\) agrees with the first line of (\ref{bigugly}). 

Checking Claims A2--A4 proceeds in a very similar fashion, so we leave this verification to the reader in the main version of this paper. Complete details are available in the arXiv version; see \S\ref{ArxivVersion}.

\begin{answer}
\vspace{2mm}
\noindent{ \em Proof of Claim A2.} 
By consideration of (\ref{ThePoly3},\ref{usefulbin}),
we have that \([x_jy_k m_{T_1, T_2, T_3, q}] \mathscr{L}\) is equal to:
\begin{align}
&
[x_jy_k m_{T_1, T_2, T_3, q}] \left(\sum
(-1)^{|A|}(-1)^{k - 1 - r}
{k -1 -|V| \choose r}
z^{k- |A|- 2}z^rx_{(A^c, +)} x_{(V, \times)}y_{(A, \times)}y_k
\right)\label{PO11}
\end{align}
where the sum is over \(A,V \subseteq [1,k-1]\), \(r \in \mathbb{Z}_{\geq 0}\). The only contribution to the coefficient is from \(A = T_2 \sqcup T_3\), \(V = T_1 \sqcup T_2\), \(r = q - k +t_2 + t_3 + 2\). We also have \((T_1 \sqcup T_2) \cap A^c = T_1\), and \(x_{(A^c, +)}\) contributes \(x_j\) to the product. Therefore we have that (\ref{PO11}) is equal to
\begin{align}\label{PO13}
(-1)^{q+1}{k - t_1 - t_2 -1 \choose q - k + t_2 + t_3 + 2}.
\end{align}
Thus, by (\ref{PO11}, \ref{PO13}), we have that Claim A2 is verified, and \([x_j y_k m_{T_1, T_2, T_3, q}]\mathscr{L}\) agrees with the second line of (\ref{bigugly}).

\vspace{2mm}
\noindent
{\em Proof of Claim A3.}
By consideration of (\ref{ThePoly3},\ref{usefulbin}),
we have that \([ m_{T_1, T_2, T_3, q}] \mathscr{L}\) is equal to:
\begin{align}
&
[ m_{T_1, T_2, T_3, q}] \left(\sum
(-1)^{|A|}(-1)^{k - 1 - r}
{k -1 -|V| \choose r}
z^{k- |A|- 1}z^rx_{(A, +)} x_{(V, \times)}y_{(A, \times)}
\right)\label{PO21}\\
&+
[m_{T_1, T_2, T_3, q}] \left(\sum
(-1)^{|A|+1}(-1)^{k - 1 -r}|A|
{k -1 -|V| \choose r}
z^{k - |A|}z^r x_{(V, \times)}y_{(A, \times)}
\right),\label{PO22}
\end{align}
where the sums are over \(A,V \subseteq [1,k-1]\), \(r \in \mathbb{Z}_{\geq 0}\). First consider (\ref{PO21}). The only contribution to the coefficient is from \(A = T_2 \sqcup T_3\), \(r = q - k +t_2 + t_3 + 1\). Then \((T_1 \sqcup T_2) \cap A = T_2\), so \(x_{(A^c, +)}\) contributes any choice of \(x_j\) for \(j \in T_2\), and this forces \(V = (T_1 \sqcup T_2)\backslash \{j\}\). Therefore we have that (\ref{PO21}) is equal to
\begin{align}\label{PO23}
(-1)^{q}t_2{k - t_1 - t_2 \choose q - k + t_2 + t_3 + 1}.
\end{align}

Now consider (\ref{PO22}). The only contribution to the coefficient is from \(A = T_2 \sqcup T_3\), \(r = q - k +t_2 + t_3 \), \(V = T_1 \sqcup T_2\). Therefore we have that (\ref{PO2}) is equal to
\begin{align}\label{PO24}
(-1)^{q}(t_2+ t_3){k - t_1 - t_2 -1 \choose q - k + t_2 + t_3}.
\end{align}

Thus, by (\ref{PO23}, \ref{PO24}), we have that Claim A3 is verified, and \([m_{T_1, T_2, T_3, q}]\mathscr{L}\) agrees with the third line of (\ref{bigugly}). 

\vspace{2mm}
\noindent{ \em Proof of Claim A4.} 
By consideration of (\ref{ThePoly3},\ref{usefulbin}),
we have that \([x_j m_{T_1, T_2, T_3, q}] \mathscr{L}\) is equal to:
\begin{align}
[ x_jm_{T_1, T_2, T_3, q}] \left(\sum
(-1)^{|A|}(-1)^{k - 1 - r}
{k -1 -|V| \choose r}
z^{k- |A|- 1}z^rx_{(A, +)} x_{(V, \times)}y_{(A, \times)}
\right)\label{PO31}
\end{align}
where the sum is over \(A,V \subseteq [1,k-1]\), \(r \in \mathbb{Z}_{\geq 0}\). The only contribution to the coefficient is from \(A = T_2 \sqcup T_3\), \(V = T_1 \sqcup T_2\), \(r = q - k +t_2 + t_3 + 1\). We also have \((T_1 \sqcup T_2) \cap A = T_2\), and \(x_{(A, +)}\) contributes \(x_j\) to the product. Therefore we have that (\ref{PO31}) is equal to
\begin{align}\label{PO33}
(-1)^{q}{k - t_1 - t_2 -1 \choose q - k + t_2 + t_3 + 1}.
\end{align}
Thus, by (\ref{PO31}, \ref{PO33}), we have that Claim A4 is verified, and \([x_j m_{T_1, T_2, T_3, q}]\mathscr{L}\) agrees with the fourth line of (\ref{bigugly}). Thus the lemma is verified for \(\mathscr{L}\).
\end{answer}

As Claims A1--A4 prove the veracity of the equality in the lemma statement for \(\mathscr{L}\), we now move on to prove the statement for \(\mathscr{R}\). We do this via a series of claims:

\vspace{2mm}
\noindent
{\em Claim B0. \([my_t^2] \mathscr{R} = 0\) for all \(t \in [1,k]\) and monomials \(m \in R'_k\).}

\vspace{2mm}
\noindent
{\em Claim B1. We have
\begin{align*}
J^{(1)} = 
\sum
\left[
(-1)^{q}(t_2+t_3)
{ k - t_1 - t_2 -1 \choose q-k+t_2 + t_3 +1}
+
(-1)^{q+1}
t_1
{k - t_1 - t_2 -1 \choose q- k + t_2 + t_3 + 2}\right]
y_km_{T_1, T_2, T_3, q},
\end{align*}
where the sum is over disjoint \(T_1, T_2, T_3 \subseteq [1,k-1]\) and \(q \in \mathbb{Z}_{\geq 0}\).
}

\vspace{2mm}
\noindent
{\em Claim B2. We have
\begin{align*}
J^{(2)} = &\sum 
(-1)^{q+1}t_1 {k - t_1 - t_2- 1 \choose q - k +t_2 + t_3 + 1}y_km_{T_1,T_2, T_3, q}\\
&\hspace{10mm}+
\sum
(-1)^{q+1} {k - t_1 - t_2 - 1 \choose q - k + t_2 +t_3+ 2}x_jy_km_{T_1,T_2, T_3, q},
\end{align*}
where the first sum is over disjoint \(T_1, T_2, T_3 \subseteq [1,k-1]\), \(q \in \mathbb{Z}_{\geq 0}\), and the second sum is over disjoint \(T_1, T_2, T_3 \subseteq [1,k-1]\), \(q \in \mathbb{Z}_{\geq 0}\), \(j \in T_1\).}

\vspace{2mm}
\noindent
{\em Claim B3. We have
\begin{align*}
J^{(3)} = &\sum 
(-1)^{q}t_1 {k -t_1 - t_2- 1 \choose q - k + t_2 + t_3 + 1}m_{T_1,T_2, T_3, q}\\
&\hspace{10mm}+
\sum
(-1)^{q} {k -t_1 - t_2 - 1 \choose q - k +t_2 + t_3 + 2}x_jm_{T_1,T_2, T_3, q}
\end{align*}
where the first sum is over disjoint \(T_1, T_2, T_3 \subseteq [1,k-1]\), \(q \in \mathbb{Z}_{\geq 0}\), and the second sum is over disjoint \(T_1, T_2, T_3 \subseteq [1,k-1]\), \(q \in \mathbb{Z}_{\geq 0}\), \(j \in T_1\).
}

\vspace{2mm}
\noindent
{\em Claim B4. We have
\begin{align*}
J^{(4)} &= 
\sum
(-1)^q
t_2{k - t_1 - t_2 \choose q - k +t_2 + t_3 + 1}m_{T_1,T_2,T_3,q}\\
&\hspace{5mm}+\sum
(-1)^q
\left[
\delta_{j \in T_2}
{k - t_1 - t_2- 1 \choose q - k + t_2 + t_3 + 1}
-
\delta_{j \in T_1}
{k -t_1 - t_2 - 1 \choose q - k + t_2 + t_3 + 2}
\right]
x_j m_{T_1, T_2, T_3, q}.
\end{align*}
where the first sum is over disjoint \(T_1, T_2, T_3 \subseteq [1,k-1]\) and \(q \in \mathbb{Z}_{\geq 0}\), and the second sum is over disjoint \(T_1, T_2, T_3 \subseteq [1,k-1]\), \(q \in \mathbb{Z}_{\geq 0}\), and \(j \in T_1 \sqcup T_2\).
}

\vspace{2mm}
\noindent
{\em Claim B5. We have
\begin{align*}
J^{(5)} = 
\sum
\left[
(-1)^{q}(t_2+t_3)
{ k - t_1 - t_2 -1 \choose q-k+t_2 + t_3 }
+
(-1)^{q+1}
t_1
{k - t_1 - t_2 -1 \choose q- k + t_2 + t_3 + 1}\right]
m_{T_1, T_2, T_3, q},
\end{align*}
where the sum is over disjoint \(T_1, T_2, T_3 \subseteq [1,k-1]\) and \(q \in \mathbb{Z}_{\geq 0}\).
}

\vspace{2mm}
\noindent
{\em Proof of Claim B0.}
Say by way of contradiction that \([my_t^2] J^{(i)} \neq 0\) for some monomial \(m\), \(i \in [1,5]\). Then there exist \(A,M,B,V,W \subseteq [1,k-1]\) which satisfy (\ref{AMBcrit}), with \(t \in B \cap W\). Then \(t \in A\) and \(t \notin V\). Assuming without loss of generality that \(t \in M\), we have that replacing \(M\) with \(M' = M \backslash \{t\}\) also satisfies (\ref{AMBcrit}). But, from consideration of the definition of \(J^{(i)}\), the coefficients associated to \(A,M,B,V,W\) and \(A,M',B,V,W\) in \(J^{(i)}\) have equal magnitude and opposite sign. Therefore \([my_t^2] J^{(i)} = 0\), giving the contradiction.

\vspace{2mm}
\noindent
{\em Proof of Claim B1.} By Claim B0 and the definition of \(J^{(1)}\), we have that
\[
J^{(1)} = \sum ([y_km_{T_1, T_2, T_3,q}]J^{(1)})y_km_{T_1, T_2, T_3,q},
\] summing over disjoint \(T_1, T_2, T_3 \subseteq [1,k-1]\), and \(q \in \mathbb{Z}_{\geq 0}\).
We also have that
\begin{align}\label{LLL2P}
[y_k m_{T_1, T_2,T_3,q}]J^{(1)} = \sum
(-1)^{k+\ell + 1 + |M| + |V|}
2^{|V| - |V \cap M|}
|A|
{k - |A| - 1 \choose \ell},
\end{align}
where the sum is over \(A,M,B,V,W \subseteq[1,k-1]\), \(\ell \in \mathbb{Z}_{\geq 0}\) which satisfy (\ref{AMBcrit}) and such that
\begin{align}\label{LLL1P}
B\sqcup W = T_2 \sqcup T_3; 
\;\;\;
V = T_1 \sqcup T_2;
\;\;\;
\ell = q - k + |V| + |W| +1.
\end{align}
For such \(A,M,B,V,W\), write:
\begin{align*}
A_1 = A \cap T_1;
\;\;\;
A_3 = A \cap W;
\;\;\;
M_1 = M \cap (T_1 \backslash A_1);
\;\;\;
M_2 = M \cap T_2;
\;\;\;
M_3 = M \cap (W \backslash A_3).
\end{align*}
Since \(M \subseteq V \sqcup W\), we have \(M \subseteq T_1 \sqcup T_2 \sqcup T_3\). Since \(B \subseteq T_2 \sqcup T_3\) and \(A \backslash M \subseteq B\), it follows that \(A \subseteq T_1 \sqcup T_2 \sqcup T_3\) as well.
Since \(V = T_1 \sqcup T_2\) and \(V \cap W = \varnothing\), we must have \(W \subseteq T_3\) and \(B = T_2 \cup (T_3 \backslash W)\). It follows then that \(A \cap T_2 = T_2\), and \(A \cap T_3 = (T_3 \backslash W) \sqcup A_3\). Since \(A \backslash M \subseteq B = T_2 \sqcup (T_3 \backslash W)\), we have \(A_1 \subseteq M \cap T_1\). Since \(M \subseteq V \sqcup W = T_1 \sqcup T_2 \sqcup W\),  we have \(M\cap T_3 \subseteq W\). Since \(A\backslash M \subseteq B = T_2 \sqcup (T_3 \backslash W)\) and \(A_3 \subseteq W\), we have that \(A_3 \subseteq M \cap T_3\). It follows then that
\begin{align}\label{AWMrulesP}
\begin{array}{ll}
V = T_1 \sqcup T_2;
&B = T_2 \sqcup (T_3 \backslash W);\\
A = A_1 \sqcup T_2 \sqcup (T_3 \backslash W) \sqcup A_3;
&M = A_1 \sqcup M_1 \sqcup M_2 \sqcup A_3 \sqcup M_3.
\end{array}
\end{align}
Then it is straightforward to see that ranging over all 
\begin{align}\label{AWMrules2}
W \subseteq T_3;
\;\;\;
A_1 \subseteq T_1;
\;\;\;
A_3 \subseteq W;
\;\;\;
M_1 \subseteq T_1 \backslash A_1;
\;\;\;
M_2 \subseteq T_2;
\;\;\;
M_3 \subseteq W \backslash A_3,
\end{align}
and defining \(A,M,B,V,W\) as in (\ref{AWMrulesP}) gives a complete set of all \(A,M,B,V,W \subseteq [1,k-1]\) which satisfy (\ref{AMBcrit}) and (\ref{LLL1P}). Thus (\ref{LLL2P}) is equal to
\begin{align}\label{LLL7P}
&\sum
(-1)^{ q + w  + a_1 + m_1 + m_2 + a_3 + m_3 } 2^{t_1 + t_2 - a_1 - m_1 - m_2}
{k - a_1 - t_2 - t_3 + w - a_3 -1 \choose  q- k +t_1 + t_2+w +1}\\
&\hspace{10mm}
( t_2 + t_3 + a_1 + a_3 - w)
{t_3 \choose w}{t_1 \choose a_1}{w \choose a_3}{t_1 - a_1 \choose m_1}{t_2 \choose m_2}{w-a_3 \choose m_3}, \nonumber
\end{align}
where the sum ranges over
\begin{align*}
w \in [0,t_3]; 
\;\;\;
a_1 \in [0,t_1];
\;\;\;
a_3 \in [0,w];
\;\;\;
m_1 \in [0,t_1 - a_1];
\;\;\;
m_2 \in [0,t_2];
\;\;\;
m_3 \in [0,w-a_3].
\end{align*}
Taking 
\[
\tau_1 = t_1;
\;\;\;
\tau_2 = t_2;
\;\;\;
\tau_3 = t_3;
\;\;\;
\omega = w;
\;\;\;
\alpha_1 = a_1;
\;\;\;
\alpha_3 = a_3;
\;\;\;
\mu_1 = m_1;
\;\;\;
\mu_2 = m_2;
\;\;\;
\mu_3 = m_3;
\]\[
\kappa = k;
\;\;\;
\phi = q,
\;\;\;
c=1;
\;\;\;
d=1,
\]
it follows from Lemma~\ref{wildLem} that (\ref{LLL7P}) is equal to
\begin{align*}
(-1)^{q}(t_2+t_3)
{ k - t_1 - t_2 -1 \choose q-k+t_2 + t_3 +1 }
+
(-1)^{q+1}
t_1\
{k - t_1 - t_2 -1 \choose q- k + t_2 + t_3 + 2}.
\end{align*}
Thus Claim B1 holds.

With minor tweaks, checking Claims B2--B5 proceeds in a very similar fashion to Claim B1, so we leave this verification to the reader in the main version of this paper. Complete details are available in the arXiv version; see \S\ref{ArxivVersion}.

\begin{answer}
\vspace{2mm}
\noindent
{\em Proof of Claim B2.}
By Claim B0 and the definition of \(J^{(2)}\), we have that
\begin{align}\label{QQQ19}
J^{(2)} = 
\sum
([y_km_{T_1, T_2, T_3, q}]J^{(2)}) y_km_{T_1, T_2, T_3, q} + \sum ([x_jy_km_{T_1, T_2, T_3, q}]J^{(2)}) x_jy_km_{T_1, T_2, T_3, q}
\end{align}
where the first sum is over disjoint \(T_1, T_2, T_3 \subseteq [1,k-1]\), \(q \in \mathbb{Z}_{\geq 0}\),  and the second sum is over disjoint \(T_1, T_2, T_3 \subseteq [1,k-1]\), \(q \in \mathbb{Z}_{\geq 0}\), \(j \in T_1 \sqcup T_2\).

First we consider \([y_k m_{T_1, T_2,T_3,q}]J^{(2)}\). We have that
\begin{align}\label{LLL8Q}
[y_k m_{T_1, T_2,T_3,q}]J^{(2)} = \sum
(-1)^{k+\ell  + |M| + |V|}
2^{|V| - |V \cap M|}
{k - |A| - 2 \choose \ell},
\end{align}
where the sum is over 
\(A,M,B,V,W \subseteq[1,k-1]\), \(j \in A^c\) which satisfy (\ref{AMBcrit}) and such that
\begin{align}\label{LLL9Q}
B\sqcup W = T_2 \sqcup T_3; 
\;\;\;
V \sqcup \{j\} = T_1 \sqcup T_2;
\;\;\;
\ell = q - k + |V| + |W| +1.
\end{align}
Note, if \(j \in T_2\), then since \(j\notin V\), \(j \notin A \supseteq B\), it follows that \(j \in W\). Assuming without loss of generality that \(j \in M\), we have that replacing \(M\) with \(M' = M \backslash \{j\}\) also satisfies (\ref{AMBcrit}). But, from consideration of the definition of \(J^{(2)}\), the contributions to (\ref{LLL8Q}) associated to \(A,M,B,V,W,j\) and \(A,M',B,V,W, j\) in \(J^{(2)}\) have equal magnitude and opposite sign, and thus cancel. Thus we may assume that the sum in (\ref{LLL8Q}) ranges only over \(j \in T_1 \cap A^c\).

For \(A,M,B,V,W,j\) which satisfy (\ref{AMBcrit}) and (\ref{LLL9Q}), write
\begin{align*}
\begin{array}{llll}
A_1 = A \cap T_1; 
&
A_3 = A \cap W;
\\
M_1 = (M \cap T_1) \backslash (A_1\cup\{j\});&
M_2 = M \cap T_2;
&
M_3 = M \cap (W \backslash A_3).
\end{array}
\end{align*}
Since \(V\sqcup \{j\} = T_1 \sqcup T_2\) and \(j \in T_1\), we have that \(T_2 \subseteq V\), and since \(V \cap W = \varnothing\), we must have \(W \subseteq T_3\) and \(B = T_2 \sqcup (T_3 \backslash W)\). It follows then that \(A \cap T_2 = T_2\), and \(A \cap T_3 = (T_3 \backslash W) \sqcup A_3\). Since \(A \backslash M \subseteq B = T_2 \sqcup (T_3 \backslash W)\), we have \(A_1 \subseteq M \cap T_1\). Since \(M \subseteq V \sqcup W = (T_1\backslash \{j\}) \sqcup T_2 \sqcup W\),  we have \(M\cap T_3 \subseteq W\). Since \(A\backslash M \subseteq B = T_2 \sqcup (T_3 \backslash W)\) and \(A_3 \subseteq W\), we have that \(A_3 \subseteq M \cap T_3\). We finally note that since \(f \in T_1\), we have \(j \notin V \sqcup W \supseteq M\), so \(j \notin M\). It follows then that
\begin{align}\label{AWMrulesQ}
\begin{array}{ll}
V = (T_1\backslash \{j\}) \sqcup T_2;
&B = T_2 \sqcup (T_3 \backslash W);\\
A = A_1 \sqcup T_2 \sqcup (T_3 \backslash W) \sqcup A_3;
&M = A_1 \sqcup M_1 \sqcup M_2 \sqcup A_3 \sqcup M_3.
\end{array}
\end{align}
Then it is straightforward to see that ranging over all 
\begin{align}\label{AWMrules2Q}
\begin{array}{llllll}
j \in T_1;
&
W \subseteq T_3;
&
A_1 \subseteq T_1 \backslash \{j\};
&
A_3 \subseteq W;
\\
M_1 \subseteq T_1 \backslash (A_1 \cup \{j\});
&
M_2 \subseteq T_2;
&
M_3 \subseteq W \backslash A_3,
\end{array}
\end{align}
and defining \(A,M,B,V\) as in (\ref{AWMrulesQ}) gives a complete set of all \(A,M,B,V,W \subseteq [1,k-1]\) which satisfy (\ref{AMBcrit}) and (\ref{LLL9Q}). 

Thus (\ref{LLL8Q}) is equal to
\begin{align}\label{LLL13QT}
&\sum
(-1)^{ q  + w  + a_1 + m_1 + m_2 + a_3 + m_3+1 } 2^{t_1 + t_2 - a_1 - m_1 - m_2 -1}
{k - a_1 - t_2 - t_3 + w - a_3 -2 \choose  q- k +t_1 + t_2+w}\\
&\hspace{10mm}\times
{t_1 \choose 1}
{t_3 \choose w}{t_1 -1\choose a_1}{w \choose a_3}{t_1 - a_1-1 \choose m_1}{t_2 \choose m_2}{w-a_3 \choose m_3}, \nonumber
\end{align}
where the sum ranges over
\begin{align*}
w \in [0,t_3]; 
\;\;\;
a_1 \in [0,t_1-1];
\;\;\;
a_3 \in [0,w];
\;\;\;
m_1 \in [0,t_1 - a_1-1];
\;\;\;
m_2 \in [0,t_2];
\;\;\;
m_3 \in [0,w-a_3].
\end{align*}
Taking 
\[
\tau_1 = t_1-1;
\;\;\;
\tau_2 = t_2;
\;\;\;
\tau_3 = t_3;
\;\;\;
\omega = w;
\;\;\;
\alpha_1 = a_1;
\;\;\;
\alpha_3 = a_3;
\;\;\;
\mu_1 = m_1;
\;\;\;
\mu_2 = m_2;
\;\;\;
\mu_3 = m_3;
\]\[
\kappa = k;
\;\;\;
\phi = q,
\;\;\;
c=2;
\;\;\;
d=1,
\]
it follows from Lemma~\ref{wildLem} that (\ref{LLL13QT}) is equal to
\begin{align}\label{QQQ20Q}
(-1)^{q+1} t_1{k - t_1 - t_2 - 1 \choose q - k + t_2 + t_3 + 1}.
\end{align}

Now we consider \([x_jy_k m_{T_1, T_2, T_3, q}]J^{(2)}\), where \(j \in T_1 \sqcup T_2\). We have that
\begin{align}\label{QQQ1Q}
[x_jy_k m_{T_1, T_2, T_3, q}]J^{(2)} 
=
\sum
(-1)^{k+\ell + |M| + |V|}2^{|V| - |V \cap M|}
{k - |A| - 2 \choose \ell},
\end{align}
where the sum is over 
\(A,M,B,V,W \subseteq[1,k-1]\) which satisfy (\ref{AMBcrit}) and such that
\begin{align}\label{QQQ2Q}
j \in A^c; 
\;\;\;
B\sqcup W = T_2 \sqcup T_3; 
\;\;\;
V= T_1 \sqcup T_2;
\;\;\;
\ell = q - k + |V| + |W| +1.
\end{align}
For such \(A,M,B,V,W\), write:
\begin{align*}
A_1 = A \cap T_1;
\;\;\;
A_3 = A \cap W;
\;\;\;
M_1 = (M \cap T_1) \backslash (A_1 \cap \{j\});
\;\;\;
M_2 = M \cap T_2;
\;\;\;
M_3 = M \cap (W \backslash A_3).
\end{align*}
Since \(M \subseteq V \sqcup W\), we have \(M \subseteq T_1 \sqcup T_2 \sqcup T_3\). Since \(B \subseteq T_2 \sqcup T_3\) and \(A \backslash M \subseteq B\), it follows that \(A \subseteq T_1 \sqcup T_2 \sqcup T_3\) as well.
Since \(V = T_1 \sqcup T_2\) and \(V \cap W = \varnothing\), we must have \(W \subseteq T_3\) and \(B = T_2 \cup (T_3 \backslash W)\). It follows then that \(A \cap T_2 = T_2\), and \(A \cap T_3 = (T_3 \backslash W) \sqcup A_3\). Since \(j \in A^c \cap (T_1 \sqcup T_2)\), it follows that \(j \in T_1\). Since \(A \backslash M \subseteq B = T_2 \sqcup (T_3 \backslash W)\), we have \(A_1 \subseteq M \cap T_1\). Since \(M \subseteq V \sqcup W = T_1 \sqcup T_2 \sqcup W\),  we have \(M\cap T_3 \subseteq W\). Since \(A\backslash M \subseteq B = T_2 \sqcup (T_3 \backslash W)\) and \(A_3 \subseteq W\), we have that \(A_3 \subseteq M \cap T_3\). It follows then that
\begin{align}\label{AWMrules5Q}
&V = T_1  \sqcup T_2;
\;\;\;\;B = T_2 \sqcup (T_3 \backslash W);
\;\;\;\;A = A_1 \sqcup T_2 \sqcup (T_3 \backslash W) \sqcup A_3;\\
&M  \in \{ A_1 \sqcup M_1 \sqcup M_2 \sqcup A_3 \sqcup M_3, 
A_1 \sqcup M_1 \sqcup M_2 \sqcup A_3 \sqcup M_3 \sqcup\{j\}\}.
\nonumber
\end{align}
Then it is straightforward to see that ranging over all 
\begin{align*}
W \subseteq T_3;
\;\;
A_1 \subseteq T_1;
\;\;
A_3 \subseteq W;
\;\;
M_1 \subseteq T_1 \backslash (A_1 \cup \{j\});
\;\;
M_2 \subseteq T_2;
\;\;
M_3 \subseteq W \backslash A_3,
\end{align*}
and taking \(A,M,B,V,W\) as in (\ref{AWMrules5Q}) gives a complete set of all \(A,M,B,V,W \subseteq [1,k-1]\) which satisfy (\ref{AMBcrit}) and (\ref{QQQ2Q}). 
Thus (\ref{QQQ1Q}) is equal to
\begin{align}\label{LLL13Q}
&\delta_{j \in T_1}\sum
(-1)^{q  + w+ a_1 + m_1 + m_2 + a_3 +m_3 +1 }
2^{t_1 + t_2 - a_1 - m_1 - m_2 -1}
{k - a_1 - t_2 - t_3 + w - a_3 - 2 \choose q - k + t_1 + t_2 + w+1}\\
&\hspace{10mm}\times
{t_3 \choose w}{t_1 -1\choose a_1}{w \choose a_3}{t_1 - a_1-1 \choose m_1}{t_2 \choose m_2}{w-a_3 \choose m_3}, \nonumber
\end{align}
where the sum ranges over
\begin{align*}
w \in [0,t_3]; 
\;\;\;
a_1 \in [0,t_1-1];
\;\;\;
a_3 \in [0,w];
\;\;\;
m_1 \in [0,t_1 - a_1-1];
\;\;\;
m_2 \in [0,t_2];
\;\;\;
m_3 \in [0,w-a_3].
\end{align*}
Taking 
\[
\tau_1 = t_1-1;
\;\;\;
\tau_2 = t_2;
\;\;\;
\tau_3 = t_3;
\;\;\;
\omega = w;
\;\;\;
\alpha_1 = a_1;
\;\;\;
\alpha_3 = a_3;
\;\;\;
\mu_1 = m_1;
\;\;\;
\mu_2 = m_2;
\;\;\;
\mu_3 = m_3;
\]\[
\kappa = k;
\;\;\;
\phi = q,
\;\;\;
c=2;
\;\;\;
d=2,
\]
it follows from Lemma~\ref{wildLem} that (\ref{LLL13Q}) is equal to
\begin{align}\label{QQQ21Q}
(-1)^{q+1} {k - t_1 - t_2 - 1 \choose q - k + t_2 + t_3 + 2}.
\end{align}
Finally, combining (\ref{QQQ20Q}, \ref{QQQ21Q}) proves Claim B2.

\vspace{2mm}
\noindent
{\em Proof of Claim B3.} By Claim B0 and the definition of \(J^{(3)}\), we have that
\begin{align}\label{YYY19}
J^{(3)} = \sum ([m_{T_1, T_2, T_3,q}]J^{(3)})m_{T_1, T_2, T_3,q}
+
 \sum ([x_jm_{T_1, T_2, T_3,q}]J^{(3)})x_jm_{T_1, T_2, T_3,q},
\end{align}
where the first sum is over disjoint \(T_1, T_2, T_3 \subseteq [1,k-1]\), \(q \in \mathbb{Z}_{\geq 0}\),  and the second sum is over disjoint \(T_1, T_2, T_3 \subseteq [1,k-1]\), \(q \in \mathbb{Z}_{\geq 0}\), \(j \in T_1 \sqcup T_2\).

First we consider \([m_{T_1, T_2,T_3,q}]J^{(3)}\). We have that
\begin{align}\label{YYY8}
[m_{T_1, T_2,T_3,q}]J^{(2)} = \sum
(-1)^{k+\ell  + |M| + |V|+1}
2^{|V| - |V \cap M|}
{k - |A| - 2 \choose \ell},
\end{align}
where the sum is over 
\(A,M,B,V,W \subseteq[1,k-1]\), \(j \in A^c\) which satisfy (\ref{AMBcrit}) and such that
\begin{align}\label{YYY9}
B\sqcup W = T_2 \sqcup T_3; 
\;\;\;
V \sqcup \{j\} = T_1 \sqcup T_2;
\;\;\;
\ell = q - k + |V| + |W| +1.
\end{align}
Note, if \(j \in T_2\), then since \(j\notin V\), \(j \notin A \supseteq B\), it follows that \(j \in W\). Assuming without loss of generality that \(j \in M\), we have that replacing \(M\) with \(M' = M \backslash \{j\}\) also satisfies (\ref{AMBcrit}). But, from consideration of the definition of \(J^{(3)}\), the contributions to (\ref{YYY8}) associated to \(A,M,B,V,W,j\) and \(A,M',B,V,W, j\) in \(J^{(3)}\) have equal magnitude and opposite sign, and thus cancel. Thus we may assume that the sum in (\ref{YYY8}) ranges only over \(j \in T_1 \cap A^c\).

For \(A,M,B,V,W,j\) which satisfy (\ref{AMBcrit}) and (\ref{YYY9}), write
\begin{align*}
\begin{array}{llll}
A_1 = A \cap T_1; 
&
A_3 = A \cap W;
\\
M_1 = (M \cap T_1) \backslash (A_1\cup\{j\});&
M_2 = M \cap T_2;
&
M_3 = M \cap (W \backslash A_3).
\end{array}
\end{align*}
Since \(V\sqcup \{j\} = T_1 \sqcup T_2\) and \(j \in T_1\), we have that \(T_2 \subseteq V\), and since \(V \cap W = \varnothing\), we must have \(W \subseteq T_3\) and \(B = T_2 \sqcup (T_3 \backslash W)\). It follows then that \(A \cap T_2 = T_2\), and \(A \cap T_3 = (T_3 \backslash W) \sqcup A_3\). Since \(A \backslash M \subseteq B = T_2 \sqcup (T_3 \backslash W)\), we have \(A_1 \subseteq M \cap T_1\). Since \(M \subseteq V \sqcup W = (T_1\backslash \{j\}) \sqcup T_2 \sqcup W\),  we have \(M\cap T_3 \subseteq W\). Since \(A\backslash M \subseteq B = T_2 \sqcup (T_3 \backslash W)\) and \(A_3 \subseteq W\), we have that \(A_3 \subseteq M \cap T_3\). We finally note that since \(f \in T_1\), we have \(j \notin V \sqcup W \supseteq M\), so \(j \notin M\). It follows then that
\begin{align}\label{AWMrulesX}
\begin{array}{ll}
V = (T_1\backslash \{j\}) \sqcup T_2;
&B = T_2 \sqcup (T_3 \backslash W);\\
A = A_1 \sqcup T_2 \sqcup (T_3 \backslash W) \sqcup A_3;
&M = A_1 \sqcup M_1 \sqcup M_2 \sqcup A_3 \sqcup M_3.
\end{array}
\end{align}
Then it is straightforward to see that ranging over all 
\begin{align}\label{AWMrules2}
\begin{array}{llllll}
j \in T_1;
&
W \subseteq T_3;
&
A_1 \subseteq T_1 \backslash \{j\};
&
A_3 \subseteq W;
\\
M_1 \subseteq T_1 \backslash (A_1 \cup \{j\});
&
M_2 \subseteq T_2;
&
M_3 \subseteq W \backslash A_3,
\end{array}
\end{align}
and defining \(A,M,B,V\) as in (\ref{AWMrulesX}) gives a complete set of all \(A,M,B,V,W \subseteq [1,k-1]\) which satisfy (\ref{AMBcrit}) and (\ref{YYY9}). 

Thus (\ref{YYY8}) is equal to
\begin{align}\label{YYY13Z}
&\sum
(-1)^{ q  + w  + a_1 + m_1 + m_2 + a_3 + m_3 } 2^{t_1 + t_2 - a_1 - m_1 - m_2 -1}
{k - a_1 - t_2 - t_3 + w - a_3 -2 \choose  q- k +t_1 + t_2+w}\\
&\hspace{10mm}\times
{t_1 \choose 1}
{t_3 \choose w}{t_1 -1\choose a_1}{w \choose a_3}{t_1 - a_1-1 \choose m_1}{t_2 \choose m_2}{w-a_3 \choose m_3}, \nonumber
\end{align}
where the sum ranges over
\begin{align*}
w \in [0,t_3]; 
\;\;\;
a_1 \in [0,t_1-1];
\;\;\;
a_3 \in [0,w];
\;\;\;
m_1 \in [0,t_1 - a_1-1];
\;\;\;
m_2 \in [0,t_2];
\;\;\;
m_3 \in [0,w-a_3].
\end{align*}
Taking 
\[
\tau_1 = t_1-1;
\;\;\;
\tau_2 = t_2;
\;\;\;
\tau_3 = t_3;
\;\;\;
\omega = w;
\;\;\;
\alpha_1 = a_1;
\;\;\;
\alpha_3 = a_3;
\;\;\;
\mu_1 = m_1;
\;\;\;
\mu_2 = m_2;
\;\;\;
\mu_3 = m_3;
\]\[
\kappa = k;
\;\;\;
\phi = q,
\;\;\;
c=2;
\;\;\;
d=1,
\]
it follows from Lemma~\ref{wildLem} that (\ref{YYY13Z}) is equal to
\begin{align}\label{YYY20}
(-1)^{q} t_1{k - t_1 - t_2 - 1 \choose q - k + t_2 + t_3 + 1}.
\end{align}

Now we consider \([x_j m_{T_1, T_2, T_3, q}]J^{(3)}\), where \(j \in T_1 \sqcup T_2\). We have that
\begin{align}\label{YYY1}
[x_j m_{T_1, T_2, T_3, q}]J^{(3)} 
=
\sum
(-1)^{k+\ell + |M| + |V|+1}2^{|V| - |V \cap M|}
{k - |A| - 2 \choose \ell},
\end{align}
where the sum is over 
\(A,M,B,V,W \subseteq[1,k-1]\) which satisfy (\ref{AMBcrit}) and such that
\begin{align}\label{YYY2}
j \in A^c; 
\;\;\;
B\sqcup W = T_2 \sqcup T_3; 
\;\;\;
V= T_1 \sqcup T_2;
\;\;\;
\ell = q - k + |V| + |W| +1.
\end{align}
For such \(A,M,B,V,W\), write:
\begin{align*}
A_1 = A \cap T_1;
\;\;\;
A_3 = A \cap W;
\;\;\;
M_1 = (M \cap T_1) \backslash (A_1 \cap \{j\});
\;\;\;
M_2 = M \cap T_2;
\;\;\;
M_3 = M \cap (W \backslash A_3).
\end{align*}
Since \(M \subseteq V \sqcup W\), we have \(M \subseteq T_1 \sqcup T_2 \sqcup T_3\). Since \(B \subseteq T_2 \sqcup T_3\) and \(A \backslash M \subseteq B\), it follows that \(A \subseteq T_1 \sqcup T_2 \sqcup T_3\) as well.
Since \(V = T_1 \sqcup T_2\) and \(V \cap W = \varnothing\), we must have \(W \subseteq T_3\) and \(B = T_2 \cup (T_3 \backslash W)\). It follows then that \(A \cap T_2 = T_2\), and \(A \cap T_3 = (T_3 \backslash W) \sqcup A_3\). Since \(j \in A^c \cap (T_1 \sqcup T_2)\), it follows that \(j \in T_1\). Since \(A \backslash M \subseteq B = T_2 \sqcup (T_3 \backslash W)\), we have \(A_1 \subseteq M \cap T_1\). Since \(M \subseteq V \sqcup W = T_1 \sqcup T_2 \sqcup W\),  we have \(M\cap T_3 \subseteq W\). Since \(A\backslash M \subseteq B = T_2 \sqcup (T_3 \backslash W)\) and \(A_3 \subseteq W\), we have that \(A_3 \subseteq M \cap T_3\). It follows then that
\begin{align}\label{AWMrules5X}
&V = T_1  \sqcup T_2;
\;\;\;\;B = T_2 \sqcup (T_3 \backslash W);
\;\;\;\;A = A_1 \sqcup T_2 \sqcup (T_3 \backslash W) \sqcup A_3;\\
&M  \in \{ A_1 \sqcup M_1 \sqcup M_2 \sqcup A_3 \sqcup M_3, 
A_1 \sqcup M_1 \sqcup M_2 \sqcup A_3 \sqcup M_3 \sqcup\{j\}\}.
\nonumber
\end{align}
Then it is straightforward to see that ranging over all 
\begin{align*}
W \subseteq T_3;
\;\;
A_1 \subseteq T_1;
\;\;
A_3 \subseteq W;
\;\;
M_1 \subseteq T_1 \backslash (A_1 \cup \{j\});
\;\;
M_2 \subseteq T_2;
\;\;
M_3 \subseteq W \backslash A_3,
\end{align*}
and taking \(A,M,B,V,W\) as in (\ref{AWMrules5X}) gives a complete set of all \(A,M,B,V,W \subseteq [1,k-1]\) which satisfy (\ref{AMBcrit}) and (\ref{YYY2}). 
Thus (\ref{YYY1}) is equal to
\begin{align}\label{YYY13}
&\delta_{j \in T_1}\sum
(-1)^{q  + w+ a_1 + m_1 + m_2 + a_3 +m_3 }
2^{t_1 + t_2 - a_1 - m_1 - m_2 -1}
{k - a_1 - t_2 - t_3 + w - a_3 - 2 \choose q - k + t_1 + t_2 + w+1}\\
&\hspace{10mm}\times
{t_3 \choose w}{t_1 -1\choose a_1}{w \choose a_3}{t_1 - a_1-1 \choose m_1}{t_2 \choose m_2}{w-a_3 \choose m_3}, \nonumber
\end{align}
where the sum ranges over
\begin{align*}
w \in [0,t_3]; 
\;\;\;
a_1 \in [0,t_1-1];
\;\;\;
a_3 \in [0,w];
\;\;\;
m_1 \in [0,t_1 - a_1-1];
\;\;\;
m_2 \in [0,t_2];
\;\;\;
m_3 \in [0,w-a_3].
\end{align*}
Taking 
\[
\tau_1 = t_1-1;
\;\;\;
\tau_2 = t_2;
\;\;\;
\tau_3 = t_3;
\;\;\;
\omega = w;
\;\;\;
\alpha_1 = a_1;
\;\;\;
\alpha_3 = a_3;
\;\;\;
\mu_1 = m_1;
\;\;\;
\mu_2 = m_2;
\;\;\;
\mu_3 = m_3;
\]\[
\kappa = k;
\;\;\;
\phi = q,
\;\;\;
c=2;
\;\;\;
d=2,
\]
it follows from Lemma~\ref{wildLem} that (\ref{YYY13}) is equal to
\begin{align}\label{YYY21}
(-1)^{q} {k - t_1 - t_2 - 1 \choose q - k + t_2 + t_3 + 2}.
\end{align}
Finally, combining (\ref{YYY19}, \ref{YYY20}, \ref{YYY21}) proves Claim B3.

\vspace{2mm}
\noindent
{\em Proof of Claim B4.}
By Claim B0 and the definition of \(J^{(4)}\), we have that
\begin{align}\label{QQQ19}
J^{(4)} = 
\sum
([m_{T_1, T_2, T_3, q}]J^{(4)}) m_{T_1, T_2, T_3, q} + \sum ([x_jm_{T_1, T_2, T_3, q}]J^{(4)}) x_jm_{T_1, T_2, T_3, q}
\end{align}
where the first sum is over disjoint \(T_1, T_2, T_3 \subseteq [1,k-1]\), \(q \in \mathbb{Z}_{\geq 0}\),  and the second sum is over disjoint \(T_1, T_2, T_3 \subseteq [1,k-1]\), \(q \in \mathbb{Z}_{\geq 0}\), \(j \in T_1 \sqcup T_2\).

First we consider \([ m_{T_1, T_2,T_3,q}]J^{(4)}\). We have that
\begin{align}\label{UUU8}
[m_{T_1, T_2,T_3,q}]J^{(4)} = \sum
(-1)^{k+\ell  +1+ |M| + |V|}
2^{|V| - |V \cap M|}
{k - |A| - 1 \choose \ell},
\end{align}
where the sum is over 
\(A,M,B,V,W \subseteq[1,k-1]\), \(j \in A\) which satisfy (\ref{AMBcrit}) and such that
\begin{align}\label{UUU9}
B\sqcup W = T_2 \sqcup T_3; 
\;\;\;
V \sqcup \{j\} = T_1 \sqcup T_2;
\;\;\;
\ell = q - k + |V| + |W| +1.
\end{align}
Note, if \(j \in T_1\), then \(j\notin V \cup W \supseteq M\), so \(j \notin M\). Then \(j \in A \backslash M \subseteq B \subseteq T_2 \sqcup T_3\), a contradiction since \(T_1, T_2, T_3\) are disjoint. Thus we may assume that the sum in (\ref{UUU8}) ranges only over \(j \in T_2 \cap A\).

For \(A,M,B,V,W,j\) which satisfy (\ref{AMBcrit}) and (\ref{UUU9}), write
\begin{align*}
\begin{array}{llll}
A_1 = A \cap T_1; 
&
A_3 = A \cap T_3 \cap W;
&
W_3 = W \cap T_3;
\\
M_1 = T_1 \backslash A_1;&
M_2 = M \cap (T_2 \backslash \{j\})
&
M_3 = M \cap (T_3 \backslash A_3).
\end{array}
\end{align*}

(i) First assume that \(j \in W\). Then, since \(V\sqcup \{j\} = T_1 \sqcup T_2\) and \(j \in T_2\), we have that \(T_1 \subseteq V\), and since \(V \cap W = \varnothing\), we have \(W = W_3 \sqcup \{j\}\), and \(B = (T_2 \backslash \{j\}) \sqcup (T_3 \backslash W_3)\). Since \(j \in A\) and \(j \notin B \supseteq A \backslash M\), it follows that \(j \in M\). Thus we have that
\begin{align}\label{AWMrulesjinW}
\begin{array}{ll}
V = T_1 \sqcup (T_2 \backslash \{j\});
&B = (T_2\backslash \{j\}) \sqcup (T_3 \backslash W_3);\\
A = A_1 \sqcup T_2 \sqcup (T_3 \backslash W_3) \sqcup A_3;
&M = A_1 \sqcup M_1 \sqcup M_2 \sqcup\{ j\} \sqcup A_3 \sqcup M_3.
\end{array}
\end{align}
Then it is straightforward to see that ranging over all 
\begin{align}\label{AWMrulesjinW2}
\begin{array}{llllll}
j \in T_2;
&
W_3 \subseteq T_3;
&
A_1 \subseteq T_1;
&
A_3 \subseteq W_3;
\\
M_1 \subseteq T_1 \backslash A_1;
&
M_2 \subseteq T_2 \backslash \{j\};
&
M_3 \subseteq W_3 \backslash A_3,
\end{array}
\end{align}
and defining \(A,M,B,V\) as in (\ref{AWMrulesjinW}) and \(W = W_3 \sqcup \{j\}\) gives a complete set of all \(A,M,B,V,W \subseteq [1,k-1]\) which satisfy (\ref{AMBcrit}) and (\ref{UUU9}). 

(ii) On the other hand assume that \(j \notin W\). Then we have that \(j \in B\), and since \(j \notin V \sqcup W \supseteq M\), we have that \(j \notin M\). Since \(V \cap W = \varnothing\), we then have that \(W = W_3\). Then it follows that
\begin{align}\label{AWMrulesjnotinW}
\begin{array}{ll}
V = T_1 \sqcup (T_2 \backslash \{j\});
&B = T_2 \sqcup (T_3 \backslash W_3);\\
A = A_1 \sqcup T_2 \sqcup (T_3 \backslash W_3) \sqcup A_3;
&M = A_1 \sqcup M_1 \sqcup M_2 \sqcup A_3 \sqcup M_3.
\end{array}
\end{align}
Then it is straightforward to see that ranging over (\ref{AWMrulesjinW2})
and defining \(A,M,B,V\) as in (\ref{AWMrulesjnotinW}) and \(W = W_3\) gives a complete set of all \(A,M,B,V,W \subseteq [1,k-1]\) which satisfy (\ref{AMBcrit}) and (\ref{UUU9}). 

Thus (\ref{UUU8})
 is equal to
\begin{align}\label{UUU13}
&\sum
(-1)^{  q + w_3 + a_1 +m_1 + m_2 + a_3 + m_3   }2^{t_1 + t_2 -1 -  a_1 +m_1 + m_2  }\nonumber\\
&\hspace{10mm}\times
\left[
{k - a_1 - t_2 - t_3 +w_3 + a_3 - 1 \choose 
q- k + t_1 + t_2  + w_3 + 1 }
+
{k - a_1 - t_2 - t_3 +w_3 + a_3 - 1 \choose 
q- k + t_1 + t_2  + w_3 }
\right]\nonumber
\\
&\hspace{20mm}\times
{t_2 \choose 1}
{t_3 \choose w_3}{t_1\choose a_1}{w_3 \choose a_3}{t_1 - a_1 \choose m_1}{t_2-1 \choose m_2}{w_3-a_3 \choose m_3}, \nonumber\\
&=\sum
(-1)^{  q + w_3 + a_1 +m_1 + m_2 + a_3 + m_3   }2^{t_1 + t_2 -1 -  a_1 +m_1 + m_2  }
{k - a_1 - t_2 - t_3 +w_3 + a_3  \choose 
q- k + t_1 + t_2  + w_3 + 1 }
\nonumber
\\
&\hspace{10mm}\times
{t_2 \choose 1}
{t_3 \choose w_3}{t_1\choose a_1}{w_3 \choose a_3}{t_1 - a_1 \choose m_1}{t_2-1 \choose m_2}{w_3-a_3 \choose m_3}, 
\end{align}
where the sum ranges over
\begin{align*}
w_3 \in [0,t_3]; 
\;\;\;
a_1 \in [0,t_1];
\;\;\;
a_3 \in [0,w];
\;\;\;
m_1 \in [0,t_1 - a_1];
\;\;\;
m_2 \in [0,t_2-1];
\;\;\;
m_3 \in [0,w_3-a_3].
\end{align*}
Taking 
\[
\tau_1 = t_1;
\;\;\;
\tau_2 = t_2-1;
\;\;\;
\tau_3 = t_3;
\;\;\;
\omega = w_3;
\;\;\;
\alpha_1 = a_1;
\;\;\;
\alpha_3 = a_3;
\;\;\;
\mu_1 = m_1;
\;\;\;
\mu_2 = m_2;
\;\;\;
\mu_3 = m_3;
\]\[
\kappa = k;
\;\;\;
\phi = q,
\;\;\;
c=1;
\;\;\;
d=2,
\]
it follows from Lemma~\ref{wildLem} that (\ref{UUU13}) is equal to
\begin{align}\label{QQQ20}
(-1)^{q} t_2{k - t_1 - t_2  \choose q - k + t_2 + t_3 + 1}.
\end{align}

Now we consider \([x_j m_{T_1, T_2, T_3, q}]J^{(4)}\), where \(j \in T_1 \sqcup T_2\). We have that
\begin{align}\label{VVV1}
[x_j m_{T_1, T_2, T_3, q}]J^{(4)} 
=
\sum
(-1)^{k+\ell + 1+|M| + |V|}2^{|V| - |V \cap M|}
{k - |A| - 1 \choose \ell},
\end{align}
where the sum is over 
\(A,M,B,V,W \subseteq[1,k-1]\) which satisfy (\ref{AMBcrit}) and such that
\begin{align}\label{VVV2}
j \in A; 
\;\;\;
B\sqcup W = T_2 \sqcup T_3; 
\;\;\;
V= T_1 \sqcup T_2;
\;\;\;
\ell = q - k + |V| + |W| +1.
\end{align}

(i) First assume that \(j \in T_1\). Then since \(j \in A \backslash B\) and \(A \backslash M \subseteq B\), it follows that \(j \in M\). For \(A,M,B,V,W\) which satisfy (\ref{AMBcrit}, \ref{VVV2}), write:
\begin{align*}
\begin{array}{lll}
A_1 = A \cap (T_1 \cup \backslash\{j\});
&
A_3 = A \cap W;\\
M_1 = (M \cap T_1) \backslash (A_1 \cup \{j\});
&
M_2 = M \cap T_2;
&
M_3 = M \cap (W \backslash A_3).
\end{array}
\end{align*}
Since \(M \subseteq V \sqcup W\), we have \(M \subseteq T_1 \sqcup T_2 \sqcup T_3\). Since \(B \subseteq T_2 \sqcup T_3\) and \(A \backslash M \subseteq B\), it follows that \(A \subseteq T_1 \sqcup T_2 \sqcup T_3\) as well.
 It follows then that
\begin{align}\label{AWMrulesV}
&V = T_1  \sqcup T_2;
\;\;\;\;B = T_2 \sqcup (T_3 \backslash W);
\;\;\;\;A = A_1 \sqcup \{j\} \sqcup T_2 \sqcup (T_3 \backslash W) \sqcup A_3;\\
&M= A_1 \sqcup \{j\} \sqcup M_1 \sqcup M_2 \sqcup A_3 \sqcup M_3.
\nonumber
\end{align}
Then it is straightforward to see that ranging over all 
\begin{align*}
W \subseteq T_3;
\;\;
A_1 \subseteq T_1\backslash \{j\};
\;\;
A_3 \subseteq W;
\;\;
M_1 \subseteq T_1 \backslash (A_1 \cup \{j\});
\;\;
M_2 \subseteq T_2;
\;\;
M_3 \subseteq W \backslash A_3,
\end{align*}
and taking \(A,M,B,V,W\) as in (\ref{AWMrulesV}) gives a complete set of all \(A,M,B,V,W \subseteq [1,k-1]\) which satisfy (\ref{AMBcrit}) and (\ref{VVV2}). 
Thus when \(j \in T_1\), (\ref{VVV1}) is equal to
\begin{align}\label{VVV13}
&\sum
(-1)^{
 q+w  +a_1 + 1 + m_1 + m_2 + a_3 + m_3
 }
2^{
t_1 + t_2 - a_1 - 1 - m_1 - m_2
}
{k - a_1 -t_2 - t_3 + w - a_3 -2 \choose q-k+t_1 + t_2+w+1}\\
&\hspace{10mm}\times
{t_3 \choose w}{t_1 -1\choose a_1}{w \choose a_3}{t_1 - a_1-1 \choose m_1}{t_2 \choose m_2}{w-a_3 \choose m_3}, \nonumber
\end{align}
where the sum ranges over
\begin{align*}
w \in [0,t_3]; 
\;\;\;
a_1 \in [0,t_1-1];
\;\;\;
a_3 \in [0,w];
\;\;\;
m_1 \in [0,t_1 - a_1-1];
\;\;\;
m_2 \in [0,t_2];
\;\;\;
m_3 \in [0,w-a_3].
\end{align*}
Taking 
\[
\tau_1 = t_1-1;
\;\;\;
\tau_2 = t_2;
\;\;\;
\tau_3 = t_3;
\;\;\;
\omega = w;
\;\;\;
\alpha_1 = a_1;
\;\;\;
\alpha_3 = a_3;
\;\;\;
\mu_1 = m_1;
\;\;\;
\mu_2 = m_2;
\;\;\;
\mu_3 = m_3;
\]\[
\kappa = k;
\;\;\;
\phi = q,
\;\;\;
c=2;
\;\;\;
d=2,
\]
it follows from Lemma~\ref{wildLem} that (\ref{VVV13}) is equal to
\begin{align}\label{VVV21}
(-1)^{q+1} {k - t_1 - t_2 - 1 \choose q - k + t_2 + t_3 + 2}.
\end{align}

(ii) Now assume that \(j \in T_2\). We note then that \(j \in V \cap B \cap A\), so we may have \(j \in M\) or \(j \notin M\).
For \(A,M,B,V,W\) which satisfy (\ref{AMBcrit}, \ref{VVV2}), write:
\begin{align*}
\begin{array}{lll}
A_1 = A \cap T_1;
&
A_3 = A \cap W;\\
M_1 = (M \cap T_1) \backslash A_1;
&
M_2 = M \cap T_2;
&
M_3 = M \cap (W \backslash A_3).
\end{array}
\end{align*}
Since \(M \subseteq V \sqcup W\), we have \(M \subseteq T_1 \sqcup T_2 \sqcup T_3\). Since \(B \subseteq T_2 \sqcup T_3\) and \(A \backslash M \subseteq B\), it follows that \(A \subseteq T_1 \sqcup T_2 \sqcup T_3\) as well.
 It follows then that
\begin{align}\label{AWMrulesVI}
&V = T_1  \sqcup T_2;
\;\;\;\;B = T_2 \sqcup (T_3 \backslash W);
\;\;\;\;A = A_1 \sqcup T_2 \sqcup (T_3 \backslash W) \sqcup A_3;\\
&M= A_1 \sqcup M_1 \sqcup M_2 \sqcup A_3 \sqcup M_3.
\nonumber
\end{align}
Then it is straightforward to see that ranging over all 
\begin{align*}
W \subseteq T_3;
\;\;
A_1 \subseteq T_1;
\;\;
A_3 \subseteq W;
\;\;
M_1 \subseteq T_1 \backslash A_1 ;
\;\;
M_2 \subseteq T_2;
\;\;
M_3 \subseteq W \backslash A_3,
\end{align*}
and taking \(A,M,B,V,W\) as in (\ref{AWMrulesVI}) gives a complete set of all \(A,M,B,V,W \subseteq [1,k-1]\) which satisfy (\ref{AMBcrit}) and (\ref{VVV2}). 
Thus when \(j \in T_2\), (\ref{VVV1}) is equal to
\begin{align}\label{WWW13}
&\sum
(-1)^{
q+w + a_1 + m_1 + m_2 + a_3 + m_3 
 }
2^{
t_1 + t_2 - a_1 - m_1 - m_2 
}
{k - a_1 - t_2 - t_3 + w - a_3 -1\choose q - k + t_1 + t_2 + w + 1}\\
&\hspace{10mm}\times
{t_3 \choose w}{t_1 \choose a_1}{w \choose a_3}{t_1 - a_1 \choose m_1}{t_2 \choose m_2}{w-a_3 \choose m_3}, \nonumber
\end{align}
where the sum ranges over
\begin{align*}
w \in [0,t_3]; 
\;\;\;
a_1 \in [0,t_1];
\;\;\;
a_3 \in [0,w];
\;\;\;
m_1 \in [0,t_1 - a_1];
\;\;\;
m_2 \in [0,t_2];
\;\;\;
m_3 \in [0,w-a_3].
\end{align*}
Taking 
\[
\tau_1 = t_1;
\;\;\;
\tau_2 = t_2;
\;\;\;
\tau_3 = t_3;
\;\;\;
\omega = w;
\;\;\;
\alpha_1 = a_1;
\;\;\;
\alpha_3 = a_3;
\;\;\;
\mu_1 = m_1;
\;\;\;
\mu_2 = m_2;
\;\;\;
\mu_3 = m_3;
\]\[
\kappa = k;
\;\;\;
\phi = q,
\;\;\;
c=1;
\;\;\;
d=1,
\]
it follows from Lemma~\ref{wildLem} that (\ref{WWW13}) is equal to
\begin{align}\label{WWW21}
(-1)^{q} {k - t_1 - t_2 - 1 \choose q - k + t_2 + t_3 + 1}.
\end{align}
Finally, combining (\ref{QQQ19}, \ref{QQQ20}, \ref{VVV21}, \ref{WWW21}) completes the proof of Claim B4.

\vspace{2mm}
\noindent
{\em Proof of Claim B5.} By Claim B0 and the definition of \(J^{(5)}\), we have that
\[
J^{(5)} = \sum ([y_km_{T_1, T_2, T_3,q}]J^{(1)})m_{T_1, T_2, T_3,q},
\] summing over disjoint \(T_1, T_2, T_3 \subseteq [1,k-1]\), and \(q \in \mathbb{Z}_{\geq 0}\).
We also have that
\begin{align}\label{HHH2}
[ m_{T_1, T_2,T_3,q}]J^{(1)} = \sum
(-1)^{k+\ell + 1 + |M| + |V|}
2^{|V| - |V \cap M|}
|A|
{k - |A| - 1 \choose \ell-1},
\end{align}
where the sum is over \(A,M,B,V,W \subseteq[1,k-1]\), \(\ell \in \mathbb{Z}_{\geq 0}\) which satisfy (\ref{AMBcrit}) and such that
\begin{align}\label{HHH1}
B\sqcup W = T_2 \sqcup T_3; 
\;\;\;
V = T_1 \sqcup T_2;
\;\;\;
\ell = q - k + |V| + |W| +1.
\end{align}
For such \(A,M,B,V,W\), write:
\begin{align*}
A_1 = A \cap T_1;
\;\;\;
A_3 = A \cap W;
\;\;\;
M_1 = M \cap (T_1 \backslash A_1);
\;\;\;
M_2 = M \cap T_2;
\;\;\;
M_3 = M \cap (W \backslash A_3).
\end{align*}
Since \(M \subseteq V \sqcup W\), we have \(M \subseteq T_1 \sqcup T_2 \sqcup T_3\). Since \(B \subseteq T_2 \sqcup T_3\) and \(A \backslash M \subseteq B\), it follows that \(A \subseteq T_1 \sqcup T_2 \sqcup T_3\) as well.
Since \(V = T_1 \sqcup T_2\) and \(V \cap W = \varnothing\), we must have \(W \subseteq T_3\) and \(B = T_2 \cup (T_3 \backslash W)\). It follows then that \(A \cap T_2 = T_2\), and \(A \cap T_3 = (T_3 \backslash W) \sqcup A_3\). Since \(A \backslash M \subseteq B = T_2 \sqcup (T_3 \backslash W)\), we have \(A_1 \subseteq M \cap T_1\). Since \(M \subseteq V \sqcup W = T_1 \sqcup T_2 \sqcup W\),  we have \(M\cap T_3 \subseteq W\). Since \(A\backslash M \subseteq B = T_2 \sqcup (T_3 \backslash W)\) and \(A_3 \subseteq W\), we have that \(A_3 \subseteq M \cap T_3\). It follows then that
\begin{align}\label{AWMrules30}
\begin{array}{ll}
V = T_1 \sqcup T_2;
&B = T_2 \sqcup (T_3 \backslash W);\\
A = A_1 \sqcup T_2 \sqcup (T_3 \backslash W) \sqcup A_3;
&M = A_1 \sqcup M_1 \sqcup M_2 \sqcup A_3 \sqcup M_3.
\end{array}
\end{align}
Then it is straightforward to see that ranging over all 
\begin{align}\label{AWMrules31}
W \subseteq T_3;
\;\;\;
A_1 \subseteq T_1;
\;\;\;
A_3 \subseteq W;
\;\;\;
M_1 \subseteq T_1 \backslash A_1;
\;\;\;
M_2 \subseteq T_2;
\;\;\;
M_3 \subseteq W \backslash A_3,
\end{align}
and defining \(A,M,B,V,W\) as in (\ref{AWMrules30}) gives a complete set of all \(A,M,B,V,W \subseteq [1,k-1]\) which satisfy (\ref{AMBcrit}) and (\ref{HHH1}). Thus (\ref{HHH2}) is equal to
\begin{align}\label{HHH3}
&\sum
(-1)^{ q + w  + a_1 + m_1 + m_2 + a_3 + m_3 } 2^{t_1 + t_2 - a_1 - m_1 - m_2}
{k - a_1 - t_2 - t_3 + w - a_3 -1 \choose  q- k +t_1 + t_2+w }\\
&\hspace{10mm}
( t_2 + t_3 + a_1 + a_3 - w)
{t_3 \choose w}{t_1 \choose a_1}{w \choose a_3}{t_1 - a_1 \choose m_1}{t_2 \choose m_2}{w-a_3 \choose m_3}, \nonumber
\end{align}
where the sum ranges over
\begin{align*}
w \in [0,t_3]; 
\;\;\;
a_1 \in [0,t_1];
\;\;\;
a_3 \in [0,w];
\;\;\;
m_1 \in [0,t_1 - a_1];
\;\;\;
m_2 \in [0,t_2];
\;\;\;
m_3 \in [0,w-a_3].
\end{align*}
Taking 
\[
\tau_1 = t_1;
\;\;\;
\tau_2 = t_2;
\;\;\;
\tau_3 = t_3;
\;\;\;
\omega = w;
\;\;\;
\alpha_1 = a_1;
\;\;\;
\alpha_3 = a_3;
\;\;\;
\mu_1 = m_1;
\;\;\;
\mu_2 = m_2;
\;\;\;
\mu_3 = m_3;
\]\[
\kappa = k;
\;\;\;
\phi = q,
\;\;\;
c=1;
\;\;\;
d=0,
\]
it follows from Lemma~\ref{wildLem} that (\ref{HHH3}) is equal to
\begin{align*}
(-1)^{q}(t_2+t_3)
{ k - t_1 - t_2 -1 \choose q-k+t_2 + t_3  }
+
(-1)^{q+1}
t_1
{k - t_1 - t_2 -1 \choose q- k + t_2 + t_3 + 1},
\end{align*}
which completes the proof of Claim B5.
\end{answer}

With Claims B1--B5 verified, it is straightforward to check that \(\mathscr{R} = J^{(1)} + J^{(2)} = J^{(3)} = J^{(4)} = J^{(5)}\) is equal to the expression in the lemma statement, completing the proof.
\end{proof}

Finally, Lemma~\ref{subz}, \ref{LRequal} together imply
\begin{Corollary}\label{ProLiv}
Lemma~\ref{ThePolyLem}(iv) holds. 
\end{Corollary}

\end{document}